\documentclass[11pt, letterpaper]{amsart}
\usepackage{amssymb, amsfonts, amsmath, amsthm, amscd, graphicx, mathrsfs, mathtools, stix, xypic, hyperref, caption, svg, epstopdf, epsfig, latexsym}
\usepackage{color}
\usepackage[all]{xy}
\numberwithin{figure}{section}
\numberwithin{equation}{section}
\usepackage{hyperref}
\usepackage{soul}
\hypersetup{colorlinks=true,linkcolor=blue,citecolor=blue}
\newtheorem{Theorem}{Theorem}[section]
\newtheorem{Lemma}[Theorem]{Lemma}
\newtheorem{Proposition}[Theorem]{Proposition}
\newtheorem{Corollary}[Theorem]{Corollary}
\newtheorem{Conjecture}[Theorem]{Conjecture}
\newtheorem{Example}[Theorem]{Example}
\theoremstyle{definition}

\theoremstyle{remark}
\newtheorem{Remark}[Theorem]{Remark}
\newcommand{\Mod}{\mathrm{Mod}}
\newcommand{\eps}{\epsilon}

\newcommand{\calH}{\mathcal{H}}

\newcommand{\G}{\Gamma}

\newcommand{\D}{\Delta}

\newcommand{\dist}{\mathrm{dist}}
\newcommand{\diam}{\mathrm{diam}}
\renewcommand{\mod}{\mathrm{mod}}

\newcommand{\Card}{\mathrm{Card}}
\newcommand{\SLE}{\mathrm{SLE}}
\DeclareMathAlphabet{\mathpzc}{OT1}{pzc}{m}{it}
\begin{document}
	\title{Conformal Dimension of the Brownian Graph}
	
	\author{Ilia Binder}
	\address{Department of Mathematics, University of Toronto, 40 St. George Street, Toronto, Ontario, Canada M5S2E4}
	\email{ilia@math.toronto.edu}
	\thanks{I.B. was supported by an NSERC Discovery grant.}
	
	\author{Hrant Hakobyan}
	\address{Department of Mathematics, Kansas State University, Manhattan, KS, 66506}
	\email{hakobyan@math.ksu.edu}
	\thanks{H.~H. was partially supported by Simons Foundation Collaboration Grant, award ID: 638572.}
	
	\author{Wen-Bo Li}
	\address{Beijing International Center for Mathematical Research, Peking University, No. 5 Yiheyuan Road, Haidian District, Beijing, China, 100871}
	\email{liwenbo@bicmr.pku.edu.cn}
	
	
	\begin{abstract}
		Conformal dimension of a metric space $X$, denoted by $\dim_C X$, is the infimum of the Hausdorff dimension among all its quasisymmetric images.  If conformal dimension of $X$ is equal to its Hausdorff dimension,   $X$  is said to be \emph{minimal for conformal dimension}.  In this paper we show that the graph of one dimensional Brownian motion is almost surely minimal for conformal dimension.  We also give other examples of sets that are minimal for conformal dimension. These include  Bedford-McMullen self-affine carpets with uniform fibers as well as graphs of continuous functions of Hausdorff dimension $d$, for every $d\in[1,2]$.  The main technique in the proofs is the construction of ``rich families of minimal sets of conformal dimension one''. The latter concept is quantified using Fuglede's modulus of measures.
	\end{abstract}
	
	\maketitle
	
	\tableofcontents
	
	\section{Introduction}\label{Introduction}
	Let $\eta:[0, \infty) \longrightarrow [0, \infty)$ be an increasing continuous function such that  $\eta(0)=0$ and $\eta(t)\underset{t\to\infty}{\longrightarrow}\infty$.  A homeomorphism  $f: X \to Y$ between metric spaces is called \emph{$\eta$-quasisymmetric},  if for all $x, y, z \in X$ with $x \neq z$ and $t>0$ we have
	\begin{align}\label{def:quasisymmetry}
		\frac{d_Y(f(x),f(y))}{d_Y(f(x),f(z))} \leq \eta \left(t\right),
	\end{align}
	whenever ${d_X(x,y)} \leq t {d_X(x,z)}$.
	A mapping $f$ is called \emph{quasisymmetric} if it is $\eta$-quasisymmetric for some \emph{distortion function} $\eta$.  Informally, quasisymmetries can be thought of as deformations which distort approximate shapes of objects by a bounded amount, cf.  \cite{Hei01}.
	
	Quasisymmetries were introduced by Tukia and V\"ais\"al\"a as a generalization of conformal and quasiconformal mappings to the setting of general metric spaces \cite{Tukia-Vaisala}.  In fact, for $n\geq 2$ a self map of  $\mathbb{R}^n$ is quasiconformal if and only if it is quasisymmetric, see e.g.  \cite[Section 34]{Vaisala}.  This equivalence holds true in a much greater generality,  that is for so-called Loewner spaces introduced by Heinonen and Koskela in \cite{HK98},
	where also the foundations of the field of analysis on metric spaces were laid.
	
	\subsection{Conformal dimension} A central problem in geometric mapping theory and analysis on metric spaces is the classification of metric spaces up to quasisymmetries. The study of quasisymmetric invariants thus plays an important role in this context.  One such invariant is obtained by studying how Hausdorff dimension of $X$, denoted by $\dim_H X$,  varies under quasisymmetries.
	
	It is well known that Hausdorff dimension of any metric space $(X,d_X)$ can be arbitrarily increased by quasisymmetries. Indeed, if $p<1$ then the identity  mapping of $(X,d_X)$ to $\left(X,d_X^p\right)$,  also known as ``snowflaking'',  is a quasisymmetry which increases the Hausdorff dimension of $X$ by a factor of $1/p>1$.  On the other hand, for some spaces the Hausdorff dimension cannot be made smaller (e.g. $\mathbb{R}^n$).
	
	\emph{Conformal dimension} of a metric space $X$ is the infimum of the Hausdorff dimension among all quasisymmetric images of $X$:
	\begin{align}
		\dim_C(X) = \inf \left\{\dim_H f(X) : \mbox{$f$ is a quasisymmetry} \right\}.
	\end{align}
	
	Clearly, $\dim_C X\leq \dim_H X$.  A metric space $X$ is said to be \emph{minimal for conformal dimension} or \emph{minimal} for short if
	\begin{align}
		\dim_C X = \dim_H X.
	\end{align}
	
	Since being  introduced in \cite{Pan89}, conformal dimension has had a profound influence on analysis on metric spaces and also found important applications in geometric group theory and dynamics, see e.g. \cite{Bonk-Kleiner:confdim,Bonk-Meyer,Kleiner:icm,MT10}.
	
	The problem of finding or estimating conformal dimension has been studied for many general metric spaces and concrete deterministic fractal sets, see \cite{BT01, DS97,KOR18, Kov06,Kwa20, Mac11, MT10,TW06}.  Nevertheless,  many interesting questions remain wide open.  For instance, the precise value of conformal dimension of the standard Sierpi\'nski carpet $S_3$ is not known, and also it is not known if $\dim_C S_3$  is achieved for any specific quasisymmetric mapping, cf. \cite{MT10}.  Perhaps more strikingly it is not even known if there is a quasisymmetric mapping $f$ of the complex plane s.t. $\dim_H f(S_3)< \dim_H S_3$. The fact that $\dim_C S_3$ is strictly smaller than $\dim_H S_3$ follows from the work of Keith and Laakso \cite{Keith-Laakso}. For the best currently known estimates of $\dim_C S_3$ the reader may consult \cite{Kwa20}.
	
	\subsection{Main result} Given the progress in understanding the quasisymmetric geometry of many deterministic fractals,  a very natural and interesting further step is to study quasiconformal geometry of stochastic objects and, in particular,  their conformal dimension.  Namely, given a random space $X$ with almost sure Hausdorff dimension $d$, it is natural to ask if conformal dimension of $X$ is almost surely a constant, and if so,  is $X$ minimal almost surely.  In \cite{RS21}, it was shown that random sets known as fractal percolation were not minimal almost surely, though no explicit values for the conformal dimension were obtained. See also \cite{Mackay12} for results on the conformal dimension of boundaries of random groups. 
	
	In this paper, we denote by $W\!(t)$ the $1$-dimensional Brownian motion, and by $\Gamma(W)$ its graph. It is well known that Hausdorff dimension of the graph $\Gamma(W)$ is ${3}/{2}$ a.s.,  see e.g. \cite[Theorem $4.29$]{MP10}. The following is the main result of this paper.
	
	\begin{Theorem}\label{bg}
		The graph $\Gamma(W)$ of the $1$-dimensional Brownian motion is minimal for conformal dimension almost surely.
	\end{Theorem}
	
	As far as we know, Theorem \ref{bg}  is the first result giving minimality and an explicit value for conformal dimension of continuous random spaces and ``natural'' stochastic processes.
	
	In what follows we briefly describe the usual tools needed to obtain lower bounds for conformal dimension and how we needed to generalize them to apply to the case of the graph of  Brownian motion.
	
	\subsection{Lower bounds and modulus}
	Lower bounds for the conformal dimension, and therefore also proofs of minimality of a metric  space $X$,  are usually obtained by utilizing various concepts of ``sufficiently rich families of rectifiable curves'' in $X$.  The latter can be quantified through the concepts of modulus of path families, the validity of Poincar\'e inequalities,  or the existence of certain diffuse enough measures on families of curves, see \cite[Chapter 4]{MT10} and references therein. The most relevant to the present work is the classical concept of the \emph{modulus of a curve family} $\G$ in $X$, a fundamental tool in geometric function theory \cite{HK98, Hei01}. Recall that if $p\geq 1$,  $(X,d,\mu)$ is a metric measure space,  and $\G$ is a family of locally rectifiabe curves in $X$, then $p$-modulus of $\G$ is defined by
	\begin{align}
		\mathrm{mod}_p(\Gamma) = \inf_{\rho} \int_X \rho^p d \mu,
	\end{align}
	where the infimum is over all Borel functions $\rho$ such that $\int_{\gamma} \rho ds\geq 1$ for all $\gamma\in\Gamma$. Sometimes we will denote the modulus by $\mathrm{mod}_p(\Gamma,\mu)$ to emphasize the dependence on the background measure $\mu$.
	
	An important result of Tyson \cite{Tyson}  states that if $(X,d,\mu)$ is an Ahlfors $p$-regular metric measure space for some $p>1$, i.e., for every ball $B(x,r)\subset X$ one has $\mu(B(x,r))\asymp r^p$, then $\dim_C X=p$,  provided $X$ contains a curve family $\Gamma$ of positive $p$-modulus.  See Theorem \ref{modest} for a little more general formulation.  The prototypical example of a metric space minimal for conformal dimension is the product space $X=E\times [0,1]$, where $E$ is a compact set in $\mathbb{R}^n$.  The curve family $\{ \{x\} \times [0,1]: x\in E\}$ in $X$ then has positive $1$-modulus and hence $X$ is minimal,  see \cite{BT01}.
	
	Since the graph of the 1-dimensional Brownian motion does not contain any rectifiable curves,  and the modulus of any family of such curves vanishes,  one cannot use Tyson's theorem to prove Theorem \ref{bg}.
	
	In \cite{Hak09} the second named author showed that an analogue of Tyson's theorem still holds if $X$ contains a \emph{sufficiently rich family $\mathcal{E}$ of minimal sets of conformal dimension one}.  To quantify the ``richness'' of such families of minimal sets, Fuglede's notion of modulus of families of measures \cite{Fug57} was used.
	Specifically,   Theorem 5.5 in \cite{Hak09} states that if $\mathcal{E}=\{E\}$ is a family of subsets of $X$, of conformal dimension one,  each of which supports a measure $\lambda_E$,  such that for every ball $B(x,r)\subset X$  we have
	\begin{align}\label{ineq:growth}
		\begin{split}
			\mu(B(x,r))&\lesssim r^d, \\
			\lambda_{E}(B(x,r))&\gtrsim r,  \forall E\in\mathcal{E},  \mbox{ and } x\in E,
		\end{split}
	\end{align}
	then a lower bound $\dim_C X\geq d$ can be obtained if
	\begin{align}\label{ineq:modpositive}
		\Mod_1(\{\lambda_E\})&>0.
	\end{align}
	
	One of our main results, Theorem \ref{fmodest},  generalizes the minimality theorems from \cite{Tyson} and \cite{Hak09} by requiring that conditions (\ref{ineq:growth}) hold only locally.  When applying Theorem \ref{modest} it is important to exhibit a family of (minimal) sets of conformal dimension $1$. A key result in this direction is Theorem \ref{cd1}, which extends theorems from \cite{Hak06} and \cite{Hak09} about minimal Cantor sets by providing many new examples of minimal sets of dimension 1. The sets in Theorem \ref{cd1} are not necessarily homeomorphic to the Cantor set but still have a sort of a hierarchical structure, cf. Section \ref{section:cantor}.
	Theorems \ref{fmodest} and \ref{cd1} allow us to obtain many new examples of minimal spaces,  including minimal sets which are graphs of continuous functions of any dimension $d\in[1,2]$,  self-affine sets,  and the Brownian graph. Next we describe these applications in more detail.
	
	\subsection{Minimality of Bedford-McMullen sets with uniform fibers}

	We say a set $K\subset[0,1]^2$ is a \emph{Bedford-McMullen set} if it can be constructed as follows.
	
	Suppose  $1{<}\ell{<}m$ are positive integers.  Divide the unit square $[0,1]^2$ into $m\times \ell$ congruent closed rectangles of width $1/m$ and height $1/\ell$ with disjoint interiors. Choose a pattern/subset $D_1\subset \{1,\ldots,m\}\times\{1,\ldots,\ell\}$. Keep the closed rectangles $K_{1,j}$, with $1{<}j{<}\mathrm{Card}(D_1)$, corresponding to $D_1$, and remove the rest.  Choose  $D_2\subset \{1,\ldots,m\}\times\{1,\ldots,\ell\}$. Divide each of the rectangles $K_{1,j}$ into $m\times\ell$ congruent rectangles with disjoint interiors of width $1/m^2$ and height $1/\ell^2$, keep the ones  corresponding to the pattern $D_2$, and remove the rest.   Denote by $K_2$ the union of these rectangles of generation $2$. 
	
	Continuing by induction we suppose that the set $K_n$ obtained at step $n$  is a union of closed rectangles of width $1/m^n$ and height $1/\ell^n$. Each of these rectangles are then again divided into $m\times \ell$ rectangles and we keep the ones corresponding to a pattern $D_{n+1}\subset\{1,\ldots,m \}\times\{1\ldots,\ell\}$. The union of these rectangles is denoted by $K_{n+1}$. We finally define $$K=K(\{D_i\})= \bigcap_{n=0}^\infty K_n.$$
	
	We say that a Bedford-McMullen set $K\subset[0,1]^2$ has \emph{uniform fibers} if there is an integer $k$ between  $1$ and $m$, such that $D_i$ has exactly $k$ elements in each row, for all $i\geq 1$.
	See Figure \ref{minimalgraph} for an example of a Bedford-McMullen set with $m=4$ and $k=\ell=2$. 
	
	In Section \ref{MSS} we prove the following.

	\begin{figure}[t]
		\centering
		\includegraphics[height=1.2in]{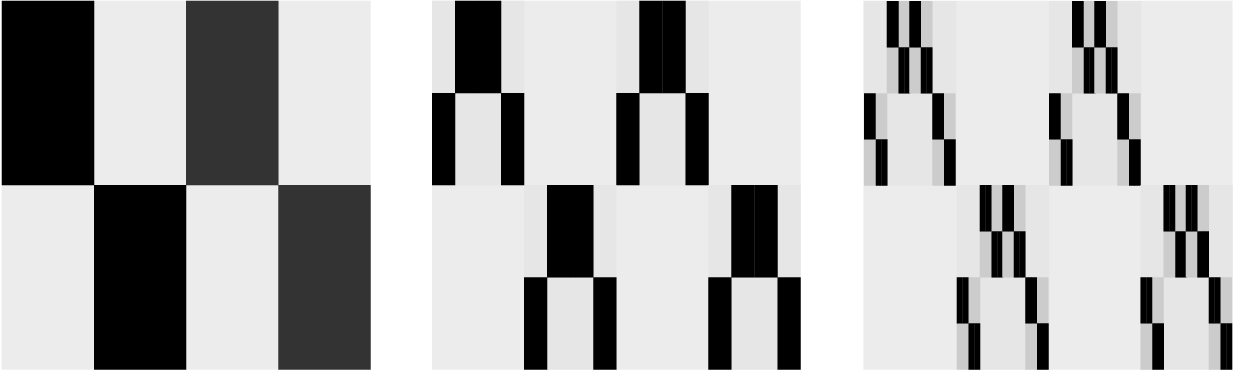}
		\caption{\small{The first three generations of a Bedford-McMullen type set with uniform fibers, with $m=4$, $k=\ell=2$.  At each step one chooses $4$ rectangles out of $8$ possible ones, so that $2$ are chosen from each row.  The resulting set $K$ has Hausdorff dimension $3/2$, and \emph{every} nontrivial intersection of $K$ with a horizontal line is of Hausdorff dimension $1/2$.  By Theorem \ref{thm:minimalBM}  conformal dimension of $K$ is also $3/2$. }} \label{minimalgraph}
	\end{figure}
	
	\begin{Theorem}\label{thm:minimalBM}
		If  $K\subset[0,1]^2$  is a Bedford-McMullen set with uniform fibers then it is minimal for conformal dimension. Thus,
		\begin{align}\label{eqn:dimK}
			\dim_C K = \dim_H K = 1+\log_m k.
		\end{align}
	\end{Theorem}
	
	If $D_i=D$ for every $i\geq 1$, following the notation in \cite[p.120]{BP17} we denote the sets constructed above by $K(D)$, even if $D$ does not have the same number of elements in each row. These sets are examples of self-affine fractals which  have been studied extensively since the famous papers of Bedford \cite{Bedford} and McMullen \cite{McMullen} and are usually called  \emph{Bedford-McMullen carpets}, cf.  \cite{Fraser} (note that in \cite{BP17} they are called McMullen sets). One important feature of Bedford-McMullen carpets is that in general their Hausdorff and Minkowski dimensions do not coincide, see e.g. the example after Theorem 4.2.1 in \cite{BP17}.  However, if the pattern $D\subset \{1,\ldots,m\}\times\{1,\ldots,\ell\}$ has the same number of elements in all the rows, in which case $K(D)$ is said to be a \emph{Bedford-McMullen carpet with uniform fibers}, we do have that $K(D)$ has the same Hausdorff and Minkowski dimensions.  Moreover, the following is an immediate consequence of Theorem \ref{thm:minimalBM}.
	\begin{Corollary}\label{BMSM}
		Every Bedford-McMullen carpet with uniform fibers $K(D)$ is minimal for conformal dimension.
	\end{Corollary}
	
	In \cite{Mac11} Mackay proved that all Bedford-McMullen carpets are either minimal (when projection in one direction covers an interval) for \emph{conformal Assouad dimension} or have conformal dimension $0$. See \cite{Mac11} for the precise definition of Assouad dimension.  In the case of the Bedford-McMullen carpets with uniform fibers, Hausdorff dimension is equal to Minkowski dimension and Assouad dimension. It then follows from Corollary \ref{BMSM} that conformal dimension, conformal Minkowski dimension, and conformal Assouad dimension of the Bedford-McMullen carpets with uniform fibers are all equal to each other. It is not known if the analogue of Mackay's theorem holds for conformal dimension in the case of Bedford-McMullen carpets $K(D)$ which do not have uniform fibers (and in particular if $1<\dim_H K(D)<\dim_M K(D)$). On a related note, we remark that Eriksson-Bique recently proved  that conformal dimension and conformal Assouad dimension coincide for quasi-self-similar metric spaces and are also equal to the so called Ahlfors regular conformal dimension, see \cite{Eriksson-Bique}.
	
	From the proof of Theorem \ref{thm:minimalBM} it would be clear that it works even if in the construction of $K$ above one replaces each generation $n$ rectangle by an \emph{arbitrary} pattern $D\subset \{1,\ldots,m\}\times\{1,\ldots,\ell\}$  of $k\ell$ rectangles, provided that there are $k$ rectangles in each row, see also Remark \ref{remark:different-patterns}. In other words, the replacement patterns can be different even in the same generation. This would allow us to construct minimal sets which are graphs of continuous functions.  Specifically, in Section \ref{MSS} we prove the following.
	
	\begin{Corollary}\label{thm:mg}
		For every $d\in[1,2]$ there is a continuous function $f:[0,1]\to\mathbb{R}$ such that the graph $\G(f)$ of $f$ is minimal for conformal dimension,  and $\dim_C \Gamma(f) = d$.
	\end{Corollary}
	
	\begin{figure}[tbp]
		\centering
		\includegraphics[height= 1.2 in]{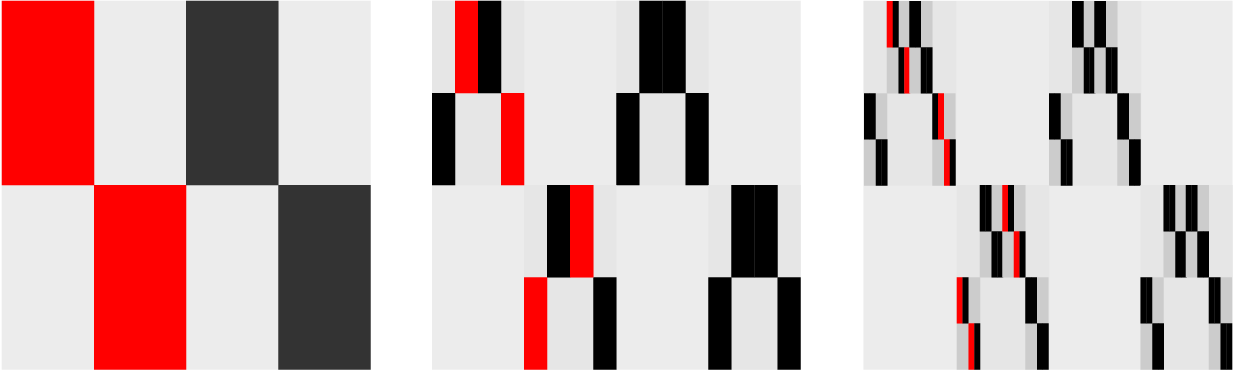}
		\caption{\small{To obtain the family  $\mathcal{E}$ of ``vertical'' Cantor sets in $K$ one has to choose one rectangle from each row at every step.  Every vertical Cantor set $E$ is minimal and has conformal dimension $1$ and is equipped with the measure $\lambda_E$ which is just the pullback of Lebesgue measure from $\{0\}\times [0,1]$ under the orthogonal projection onto $y$-axis.  The family $\{\lambda_E\}$ has positive modulus with respect to the natural measure $\mu$ on $K$.}}\label{figure:vertical-cantor-set}
	\end{figure}
	
	To prove minimality of sets $K$ in Section \ref{MSS} we observe that they contain large collections of minimal sets of conformal dimension $1$. Indeed,  choosing \emph{one}  rectangle (out of $k$ possible) of generation $n$ from each row at every step one obtains an uncountable collection of  \emph{vertical Cantor sets} in $K$, see Fig. \ref{figure:vertical-cantor-set}.  We prove that each such vertical Cantor set $E\subset K$ has conformal dimension $1$, see Lemma \ref{lemma:vertical-cantor1}.  Moreover, $E$  is equipped with a measure $\lambda_E$,  which is  the pullback of the Lebesgue measure under the horizontal projection (or equivalently the $\mathcal{H}^1$ measure restricted to $E$).  To apply Theorem \ref{fmodest} it is enough to show that there is a measure $\mu$ on $K$ such that $\mu(B_r)\lesssim r^d$, with $d=\dim_H K$, so that inequality (\ref{ineq:modpositive}) holds. Equivalently, one needs to show that if $\rho$ is a Borel function on $K$ such that $\int \rho d\lambda_E\geq 1$ for every vertical Cantor set $E$ in $K$,  then
	\begin{align}\label{ineq:intpositive}
		\int_K \rho d\mu>0.
	\end{align}
	It is not hard to see that if $\pi_y{:}(x,y)\mapsto(0,y)$ then for (\ref{ineq:intpositive}) to hold it is necessary for the push forward  measure $(\pi_y)_{*}(\mu)$ to be absolutely continuous with respect to the Lebesgue measure $\mathcal{L}^1$ on the vertical  interval $\{0\}\times [0,1]$.  Such a $\mu$ can be obtained easily by setting the measure of every $n$-block $Q\subset K_n$ to be the same, i.e.,  $\mu(Q)=1/m^n$.  It turns out that with this choice of $\mu$ and the family $\{\lambda_E\}$ as above, conditions of Theorem \ref{fmodest} are satisfied and hence $K$ is minimal.
	
	\subsection{Conformal dimension of the Brownian graph}\label{CDBG}
	Minimal sets of the previous section  have the special property that almost every nontrivial intersection of $K$ with a horizontal line has Hausdorff dimension equal to $\log_m k=\dim_H(K)-1$. This can be thought of as some type of a product-like structure on $K$ (at least  at the level of Hausdorff dimensions).  It turns out that the graph of Brownian motion has this property as well almost surely.  The details of the proof of Theorem \ref{bg} are quite involved and are carefully discussed in Sections \ref{BM} and \ref{BGMAS}. Here we give a rough sketch of the idea.
	
	The well-known dimension doubling theorem of Kaufman \cite[Theorem $9.28$]{MP10} states that almost surely for \emph{every} $a\in\mathbb{R}$ we have $\dim_H \left( \Gamma(W) \cap Z_a \right) = {1}/{2}$, where $Z_a$ is the horizontal line through $(0,a)$,  see \cite[Corollary $9.30$]{MP10}.  On the other hand,  $\dim_H\G(W)=3/2$ almost surely,  see \cite[Theorem $4.29$]{MP10}. Therefore, almost surely we have
	\begin{align}\label{equation:BMslices}
		\dim_H (\G(W)\cap Z_a) = \dim_H \G(W) - 1
	\end{align}
	for every $a\in\mathbb{R}$.
	
	Equation (\ref{equation:BMslices}) is one of the reasons that led us to suspect that $\Gamma(W)$ could be minimal for conformal dimension as well.  However,  there are immediate questions that have to be addressed in order to apply Theorem \ref{fmodest}. Namely, are there ``vertical'' subsets $E$, measures $\lambda_E$, and measure  $\mu$ on $\G(W)$ so that (\ref{ineq:growth}) and (\ref{ineq:modpositive}) hold?
	
	\begin{figure}[htbp]
		\centering
		\includegraphics[height= 1.2 in]{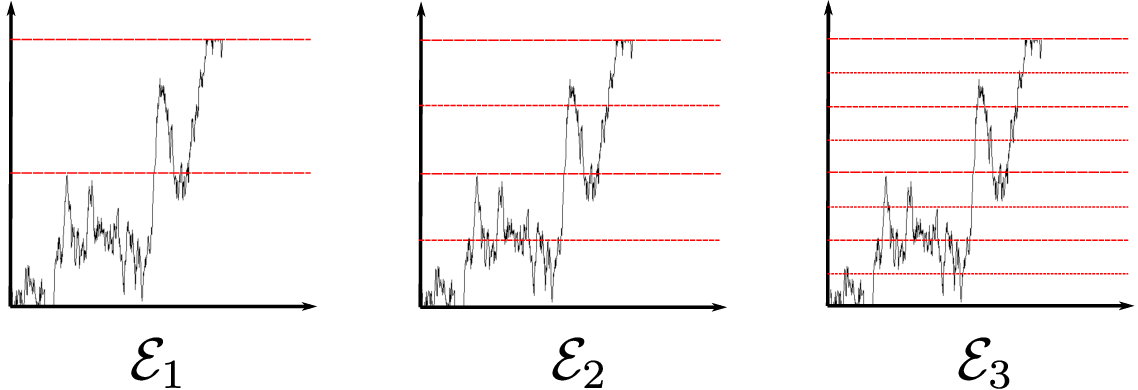}
		\caption{\small{ The family $\mathcal{E}$ of ``vertical'' subsets of the Brownian graph $\G(W)$ is constructed by intersecting $\G$ with horizontal dyadic strips as in the picture above and carefully selecting those parts in which Brownian sample path ``grows fast''. }}
	\end{figure}
	
	To mimic what was done for the Bedford-McMullen sets above we construct the subsets of $\G(W)$ by looking at the parts of the graph obtained by intersecting it with ``dyadic strips'', i.e.,
	\[
	\G(W)\cap\left\{(x,y) :  y \in\left(\frac{m}{2^n},\frac{m+1}{2^n}\right]\right\}.
	\]
	The resulting subsets are not as nicely behaved as in the case of $K$ above.  In fact we choose not all such intersections, but only those on which Brownian motion ``increases fast'',  see Section \ref{BGMAS} for details.  It turns out that it is possible to do this so that the resulting subsets of $\G(W)$ are of conformal dimension at least one. For this reason we provide quite general criteria for minimality of Cantor-like sets.  See Theorems \ref{cd1}, \ref{unioncd1} and \ref{unioncd1new}. These criteria allow us to construct  a family of vertical Cantor sets $\mathcal{E}$ in the Brownian graph $\G(W)$ which carry a family of measures $\mathbf{E}=\{\lambda_E\}_{E\in\mathcal{E}}$ satisfying the second condition in (\ref{ineq:growth}).
	
	Perhaps the most interesting property of the Brownian graph $\Gamma(W)$, from the point of view of this paper is the fact that there exists a measure $\mu$ on $\G(W)$ which ``sees'' the family $\mathbf{E}$ supported on vertical minimal subsets of $\Gamma(W)$, i.e. , for which (\ref{ineq:modpositive}) holds.  This measure $\mu$ is constructed using the so-called \emph{local time of $W\!(t)$ at level $a$}, denoted by  $L^{a}(t)$,  ``which measures the amount of time Brownian motion spends at  $a$''.  Local time $L^{a}(t)$ gives rise to a Borel measure on $\Gamma(W) \cap Z_a$ denoted by $l^{a}$, and integrating the latter with respect to $a$ gives the desired measure $\mu$ on $\G(W)$. That is, for every measurable set $A \subseteq \Gamma(W)$, we let
	\begin{equation}\label{spacemeasure}
		\mu\left( A \right) =  \int_{-\infty}^{\infty} l^a(A \cap Z_a) \ da.
	\end{equation}
	Thus, the ``product like structure'' of $\mu$ is built into its definition.
	
	The most technical part of the paper, Section \ref{BGMAS}, is devoted to proving various properties of the families of sets $\mathcal{E}$, families of measures $\mathbf{E}=\{\lambda_E\}_{E\in\mathcal{E}},$ and  the measure $\mu$. We are able to show that these quantities satisfy the conditions of Theorem \ref{fmodest}, and applying it we prove Theorem \ref{bg}.
	
	\subsection{Structure of the paper}
	The paper is organized as follows. In Section \ref{Preliminary}, we introduce basic notation, definitions, and properties that will be used in this paper. In Section \ref{CDM} we prove Theorem \ref{fmodest} which gives a general sufficient condition for minimality of a space in terms of Fuglede modulus of measures and weakens the conditions established previously in \cite{Hak09}. Also in this section we introduce the notions of sets equipped with a hierarchical structure and flatness and provide a sufficient condition for such sets to have conformal dimension at least 1, see Theorem \ref{cd1}. In Section \ref{MSS} we apply the general theory developed in Section \ref{CDM} to prove minimality of Bedford-McMullen set and carpets with uniform fibers, Theorem \ref{thm:minimalBM}. Here we also prove Corollary \ref{thm:mg}  by constructing explicit examples of minimal graphs of continuous functions over $[0,1]$.  In Section \ref{BM} we recall some of the notions  and properties of Brownian motion which are used in the proof of our main result later. Section \ref{BGMAS} is the technical core of the paper, and this is where we prove Theorem \ref{bg} by showing that we can apply Theorem \ref{fmodest} to the graph $\Gamma(W)$ of Brownian motion. In particular, in Section \ref{verticalCS} we provide a delicate construction of vertical Cantor sets in $\Gamma(W)$ and appropriate measures on them and check that these sets are minimal. In Section \ref{Conclusion} we record some observations and formulate several open problems about conformal dimension. 
	
	\subsection*{Acknowledgments}
	The authors would like to thank Malik Younsi for his helpful comments. They would also like to thank the anonymous referees, whose numerous substantive comments greatly improved the paper.
	
	\section{Preliminaries}\label{Preliminary}
	For the convenience of the reader, in this section we collect some of the basic notation, conventions,  and elementary results used in the paper.  These are quite standard and experts may want to skip them without much harm.
	\subsection{Notation and definitions}
	In this paper we write $C = C(a_1, \ldots, a_n)$ if a constant $C$ can be chosen to depend only on the parameters $a, \ldots, a_n$. By $\Card(A)$ we denote the cardinality of a set $A$.
	
	If there is a constant $C > 0$ such that $f(x) \leq C g(x)$ or $f(x) \geq C g(x)$, we may write $f(x) \lesssim g(x)$ {or} $f(x) \gtrsim g(x)$,
	respectively,  if the value of $C$ is not important.  If both inequalities hold we may write  $f(x )\asymp g(x)$ .
	
	By  $\mathbb{R}^n$ we will denote the $n$-dimensional Euclidean space, which will be equipped with the Lebesgue measure $\mathcal{L}^n$ and the standard Euclidean metric denoted by $| \ \cdot \ |$.
	
	We call an element of the family of intervals $\left\{\left[\frac{i}{2^n}, \frac{i+1}{2^n} \right]  \right\}_{i \in \mathbb{Z}}$ an \emph{$n^{\mathrm{th}}$-generation dyadic interval}.  Similarly,  an element of $\left\{\left[\frac{i}{2^n}, \frac{i+1}{2^n} \right] \times \left[\frac{j}{2^n}, \frac{j+1}{2^n} \right]  \right\}_{i,j \in \mathbb{Z}}$ will be called an \emph{$n^{\mathrm{th}}$-generation dyadic square}.
	
	The projections of $\mathbb{R}^2$ onto the $x$ and $y$ axes will be denoted by $\pi_x$ and $\pi_y$, respectively.  More precisely,
	$\pi_x((x,y)) = x$ {and} $\pi_y((x,y)) = y$.
	For $a\in\mathbb{R}$ we will denote by $Z_a$ the horizontal line passing through $(0,a)$, that is  $Z_a = \{(x,a): x \in \mathbb{R}\}$.
	
	Let $(X, d)$ be  a metric space.  We will denote by $B(x,r)$ the open ball in $X$ of radius $r>0$ centered at $x\in X$, i.e., $
	B(x,r) = \{y \in X: d(x,y) < r \}.$
	For a ball $B=B(x,r)\subseteq X$ and $\lambda>0$ we let $\lambda B = B(x,\lambda r)$.
	
	Given subsets $E$ and $F$ of $X$ and $r>0$,  we define the diameter of $E$ in $X$, the distance between subsets $E$ and $F$ of $X$, and $r$-neighborhood of $E$ in $X$ as follows:
	\begin{align*}
		\textrm{diam}(E) &=\sup\{d(x,y) : x,y \in E\},
		\\
		\textrm{dist}(E,F) &= \inf\{d(x,y): x \in E, y \in F\},
		\\
		N_r(E) &= \{x \in X: \exists \ y \in E \ \textrm{such that} \ d(x,y) < r\}.
	\end{align*}
	
	If $\textrm{diam}(E) ,\textrm{diam}(F) > 0$, the \emph{relative distance between $E$ and $F$} is the following quantity
	\begin{align}
		\Delta(E, F) = \frac{\textrm{dist}(E, F)}{\min\{\textrm{diam}(E), \textrm{diam}(F)\}}.
	\end{align}
	
	We will denote by $\mathcal{H}^\alpha$ the \emph{Hausdorff $\alpha$-measure} on $X$ for some $\alpha > 0$. More specifically, $\calH^\alpha(X)=\lim_{\epsilon\to0} \calH^\alpha_{\epsilon}(X)$, where
	\begin{align*}
		\calH^\alpha_{\epsilon}(X) = \inf \left\{ \sum_{i=1}^{\infty} \diam(E_i)^\alpha : X \subseteq \bigcup_{i=1}^{\infty} E_i, \, \diam(E_i) \leq \epsilon\right\}.
	\end{align*}
	
	Hausdorff dimension of $X$ is defined to be
	\[
	\dim_H(X) = \inf \{\alpha: \mathcal{H}^\alpha(X) = 0 \}.
	\]
	
	The Hausdorff measure on $X$ can be made more general by considering a \emph{gauge function}, i..e, a positive, increasing function $\phi$ on $[0, \infty)$ with $\phi(0) = 0$. We may associate to it the Hausdorff measure $\calH^\phi(X)=\lim_{\epsilon\to0} \calH^\phi_{\epsilon}(X)$, where
	\[
	\calH^\phi_{\epsilon}(X) = \inf \left\{ \sum_{i=1}^{\infty} \phi\left(\diam(E_i) \right) : X \subseteq \bigcup_{i=1}^{\infty} E_i, \, \diam(E_i) \leq \epsilon\right\}.
	\]
	
	The following theorem is often used to obtain lower bounds for the Hausdorff dimension of a metric space, see \cite[Lemma $1.2.8$]{BP17}.
	
	\begin{Theorem}[Mass Distribution Principle]\label{mdp}
		Let $(X, d_X, \mu)$ be a metric measure space where $\mu(X) > 0$. If there exists a constant $C = C(X)>0$ such that for any $U \subseteq X$, $\mu(U) \leq C (\diam(U))^\alpha$ for some $\alpha>0$, then $\dim_H(X) \geq \alpha$.
	\end{Theorem}
	
	Let $(X, \mu)$ be a measurable space with $\mu(X) < \infty$, then
	\[
	\intbar_{X} f d\mu = \frac{1}{\mu(X)}\int_X f d \mu.
	\]
	
	We say a sequence of measures $\mu_n$ \emph{converges weakly} to a measure $\mu$ on $X$ if for every continuous compactly
	supported function $f : X \to \mathbb{R}$, we have
	\[
	\int_X f d\mu_n \to \int_X f  d\mu.
	\]
	
	A measurable space $(X, \mu)$ is called \emph{doubling} if there exists a  constant $C \geq 1$ such that for any $x \in X$ and $r > 0$, $\mu(B(x,r)) \leq C \mu(B(x, r/2))$.
	
	\subsection{Covering lemma and measure theory}
	The following lemma is the basic covering lemma and reader may refer to \cite[Theorem~$1.2$]{Hei01} for a proof.
	\begin{Lemma}\label{cl}
		Every family $\mathcal{B}$ of balls of uniformly bounded diameter in a metric space $X$ contains a pairwise disjointed subfamily $\mathcal{B}_1$ such that
		\[
		\bigcup_{B \in \mathcal{B}}B \subseteq \bigcup_{B_1 \in \mathcal{B}_1}5B_1.
		\]
	\end{Lemma}
	
	The following result  is known as the  \emph{Bojarski lemma} and is often used to estimate modulus from above, cf., \cite[Lemma $4.2$]{Boj88}.
	
	\begin{Theorem}[Bojarski]\label{Lpinequality}
		Suppose that $\mathcal{B} = \{B_i\}_{i=1}^\infty$ is a countable collection of balls in a doubling measure space $(X ,\mu)$ and that $a_i \geq 0$ are real numbers, then
		\[
		\int_X\left( \sum_{\mathcal{B}} a_i \chi_{\lambda B_i} \right)^p d \mu \leq C(\lambda, p, \mu) \int_X\left( \sum_{\mathcal{B}} a_i \chi_{B_i} \right)^p d \mu
		\]
		for $1 < p < \infty$ and $\lambda > 1$.
	\end{Theorem}
	
	Let $f$ be a real-valued function on a topological space $X$. If $\{x : f(x) > a\}$ is open for every $a$, $f$ is said to be \emph{lower semicontinuous}. Equivalently, $f$ is lower semicontinuous if and only if for any $x_0 \in X$,
	\[
	\liminf _{x\to x_0}f(x) \geq f(x_0).
	\]
	
	If $\{x : f(x) < a\}$ is open for every $a$, $f$ is said to be \emph{upper semicontinuous}. Equivalently, $f$ is upper semicontinuous if and only if for any $x_0 \in X$,
	\[
	\limsup _{x\to x_0}f(x)\leq f(x_0).
	\]
	
	A Borel measure $\mu$ on a locally compact Hausdorff space $X$ is a \emph{Radon measure} if
	\begin{enumerate}
		\item $\mu$ is finite on any compact set;
		
		\item $\mu$ is \emph{inner regular} on any open set $U$, i.e.,
		\[
		\mu(U) = \sup \left\{ \mu(K): K \subset E, K \ \textnormal{compact}  \right\};
		\]
		
		\item $\mu$ is \emph{outer regular} on any Borel set $E$, i.e.,
		\[
		\mu(E) = \sup \left\{ \mu(U): U \supset E, K \ \textnormal{open}  \right\}.
		\]
	\end{enumerate}
	
	A set  is called \emph{$\sigma$-compact} if it is a countable union of compact sets. Let  $X$ be a locally compact Hausdorff space such that every open set in $X$ is $\sigma$-compact. Then any Borel measure on $X$ with finite measure on every compact set is a Radon measure, see \cite[Theorem~$2.18$]{Rud86}.
	
	The following theorem tells that any real-valued integrable function can be approximated by an upper semicontinuous and a lower semicontinuous function, see \cite[Theorem $2.25$]{Rud86}.
	
	\begin{Theorem}[Vitali--Carath\'eodory]\label{VCT}
		Let X be a locally compact Hausdorff space equipped with a Radon measure $\mu$. Suppose $f: X \to \mathbb{R}$ is a real-valued integrable function on a measure space $(X, \mu)$, then for any $\eps > 0$ there exist functions $u$ and $v$ on $X$ such that $u \leq f \leq v$ where $u$ is upper semicontinuous and $v$ is lower semicontinuous, and
		\[
		\int_X (v-u) d\mu < \eps.
		\]
	\end{Theorem}
	
	\subsection{Quasisymmetries and their basic properties}
	A homeomorphism $f: X \to Y$ is called \emph{$\eta$-quasisymmetric}, where $\eta:[0, \infty) \to [0, \infty)$ is a given homeomorphism, if
	\[
	\frac{d_Y(f(x),f(y))}{d_Y(f(x),f(z))} \leq \eta \left(\frac{d_X(x,y)}{d_X(x,z)} \right)
	\]
	for all $x, y, z \in X$ with $x \neq z$. The map $f$ is called \emph{quasisymmetric} if it is $\eta$-quasisymmetric for some \emph{distortion function} $\eta$.
	
	Some of the properties of quasisymmetric maps used below are summarized in the following proposition. We refer to \cite{Hei01} and \cite{HK98} for these and other  properties of quasisymmetric maps in metric spaces.
	
	\begin{Proposition}\label{diam}
		Suppose $f:X\to Y$ and $g:Y\to Z$ are $\eta$ and $\eta'$-quasisymmetric mappings, respectively.
		\begin{itemize}
			\item[(1).] The composition $g\circ f : X\to Z$ is an $\eta' \circ \eta$-quasisymmetric map.
			\item[(2).] The inverse $f^{-1}:Y\to X$ is a $\theta$-quasisymmetric map, where $\theta(t) = 1/\eta^{-1}(1/t)$.
			\item[(3).] Quasisymmetries map bounded spaces to bounded
			spaces. If $A$ and $B$ are subsets of $X$ and $A\subseteq B$, then
			\begin{align}
				\frac{1}{2\eta\left( \frac{\diam(B)}{\diam(A)} \right)} \leq \frac{\diam(f(A))}{\diam(f(B))} \leq \eta\left( \frac{2\diam(A)}{\diam(B)} \right).
			\end{align}
		\end{itemize}
	\end{Proposition}
	We will also repeatedly use the following well-known result.
	\begin{Lemma}\label{relativedistance}
		Let $A, B$ be subsets of $X$ and $f:X\to Y$ be a quasisymmetric mapping. Then
		\begin{align}\label{ineq:rel-dist-distortion}
			\frac{1}{2\eta(\Delta(A,B)^{-1})}\leq \D(f(A),f(B)) \leq \eta(2\D(A,B)).
		\end{align}
	\end{Lemma}
	
	The proof of Lemma \ref{relativedistance} is essentially the same as \cite[Theorem~$2.3$]{HL23}.
	
	\section{Conformal dimension and modulus}\label{CDM}
	
	Let $(X, d, \mu)$ be a metric measure space and $p \geq 1$. Given a family $\Gamma$ of locally rectifiable curves in $X$,  the \emph{$p$-modulus} of $\Gamma$ is defined as
	\[
	\mod_p (\Gamma) = \inf \int_{X} \rho^p d\mu,
	\]
	where the infimum is taken over all Borel functions $\rho:X\to[0,\infty)$ such that
	\begin{equation}\label{admissible}
		\int_\gamma \rho \ ds \geq 1, \ \forall \ \gamma \in \Gamma.
	\end{equation}
	Functions $\rho$ that satisfy inequality \eqref{admissible} are called \emph{admissible functions} for $\Gamma$. See \cite{Hei01} for a detailed discussion of modulus and its properties in the context of general metric spaces.
	
	The following theorem, essentially due to Tyson, asserts that if a metric space satisfies an upper mass bound and has a sufficiently rich family of curves, then a lower bound for its conformal dimension can be obtained, see \cite[Proposition $4.1.8$]{MT10}.
	
	\begin{Theorem}\label{modest}
		Let $(X, d, \mu)$ be a compact, doubling metric measure space satisfying the upper mass bound
		\[
		\mu(B(x,r)) \le C \cdot r^q
		\]
		for some constant $C = C(X) > 0$ and all balls $B(x,r)$ in $X$.  If  $X$ contains a curve family $\Gamma$ such that $\mod_p (\Gamma) > 0$ for some $1 < p \leq q$, then $\dim_C(X) \geq q$.
	\end{Theorem}
	
	In \cite{Hak09}, lower bounds for conformal dimension were obtained for spaces that have ``few'' or no curve families at all.
	The idea was to use Fuglede modulus of families of measures introduced in \cite{Fug57} rather than the classical modulus of path families. Next we recall the definition of Fuglede modulus.
	
	Let $(X, d, \mu)$ be a metric measure space and $p \geq 1$. Suppose $\mathbf{E}=\{\lambda\}_{\lambda\in \mathbf{E}}$ is a family of measures on $X$ such that the domain of each  $\lambda\in\mathbf{E}$ contains the domain of  $\mu$.  The \emph{$p$-modulus} of $\mathbf{E}$ is defined as
	\[
	\Mod_p(\mathbf{E}) = \inf \int_{X} \rho^p d\mu,
	\]
	where the infimum is taken over all Borel functions $\rho:X\to[0,\infty)$ such that
	\begin{equation}\label{admissiblem}
		\int \rho \ d \lambda \geq 1, \ \forall \ \lambda \in \mathbf{E}.
	\end{equation}
	Functions $\rho$ that satisfy inequality \eqref{admissiblem} are called \emph{admissible functions} for $\mathbf{E}$.
	
	The following result is a generalization of Theorem $5.5$ of \cite{Hak09}. The main difference is that the assumptions here are local in their nature.
	
	\begin{Theorem}\label{fmodest}
		Let $(X, d, \mu)$ be a doubling metric measure space.  Suppose there is a constant $C = C(X) > 0$ such that for every $x \in X$, there exists $r_0 = r_0(x) > 0$ such that for any $r < r_0$, we have
		\begin{equation}\label{um}
			\mu(B(x,r)) \leq C  r^q.
		\end{equation}
		
		Let $\mathcal{E}$ be a family of subsets of $X$ and $\mathbf{E} = \{\lambda_E\}_{E \in \mathcal{E}}$ be a collection of measures associated to $\mathcal{E}$ where $\lambda_E$ is supported on $E$ such that
		\begin{enumerate}
			\item $\dim_C(E) \geq  1$ for all $E \in \mathcal{E}$;
			
			\item for any $s > 1, E \in \mathcal{E}$ and $x \in E$ there exist constants $C_1 = C_1(\mathcal{E}, s)$ and $r_1 = r_1(\mathcal{E}, s, x)>0$ such that for any $r < r_1$, we have
			\begin{equation}\label{linear}
				\lambda_{E}(B(x,r) \cap E) \geq C_1r^s.
			\end{equation}
		\end{enumerate}
		
		If $\Mod_{p}(\mathbf{E}) > 0$ for some $1 \leq p < q$, then $\dim_C(X) \geq q$.
	\end{Theorem}
	
	\begin{proof}[Proof of Theorem \ref{fmodest}]
		
		Assume that $f$ is an $\eta$-quasisymmetry on $X$ such that $\dim_H f(X) < q$. Let $\alpha < q$ be a  number chosen so that $\alpha > \max(p, \dim_H f(X))$. Choose $t < 1$ so that $\alpha > \frac{1}{t} \dim_H f(X)$.
		
		For any $E \in \mathcal{E}$, we have  $\mathcal{H}^t_{1/k}(f(E)) \to  \infty$ as $k \to \infty$ since $\dim_H f(E) \geq 1$. Let $k_E \in \mathbb{N}$ denote the smallest integer such that $\mathcal{H}^t_{1/k_E}(f(E)) \geq 1$.
		
		Let $\mathcal{E}_i = \left\{E \in \mathcal{E}: k_E  \leq  i \right\}$ and $\mathbf{E}_i= \left\{ \lambda_E \in \mathbf{E}: E \in \mathcal{E}_i \right\}$. Then $\mathbf{E}  = \bigcup_{j=1}^\infty \mathbf{E}_j$ and, by subadditivity of modulus,  see  \cite[Theorem~$1$]{Fug57}, we have
		\[
		\Mod_\alpha \left( \mathbf{E} \right) \leq \sum_{j=1}^{\infty}\Mod_\alpha \left( \mathbf{E}_j \right).
		\]
		
		Note that  by H\"older's inequality if $\Mod_\alpha \left( \mathbf{E}_j \right) = 0$ then $\Mod_p \left( \mathbf{E}_j \right) = 0$,  since $\alpha>p$. Therefore,  to prove the theorem it is enough to show that $\Mod_\alpha \left( \mathbf{E}_j \right) = 0$ for every $j\in\mathbb{N}$.
		
		In the rest of the proof we let $X_1 = \bigcup_{E \in \mathcal{E}} E$. Since $\dim_H(f(X_1)) < \alpha t$, for any $\eps, \delta > 0$ there exists a collection of balls $\mathcal{B}' = \{B'_i\}$ covering $f(X_1)$ such that each $B'_i$ is centered in $f(X_1)$, $\diam(B'_i) < \delta$, and
		\[
		\sum_i \diam(B'_i)^{\alpha t} < \eps.
		\]
		
		Let $\eps, \delta, j$ be fixed and $\delta < \frac{1}{Hj}$ for $H = \eta(20)$. We choose some $s > 1$ with $q > \alpha s$.  Next,  using $\mathcal{B}'$ we will construct a cover $ \mathcal{F}'$ of $f(X_1)$,  such that for every $F' \in \mathcal{F'}$ and a \emph{circumscribing ball} $F$ of $f^{-1}(F')$ (i.e., $F$ is a compact metric ball of the smallest radius which contains $f^{-1}(F')$), the following conditions are satisfied:
		\begin{enumerate}
			\item $\mu(F) \leq C\diam(F)^q$ for some constant $C = C(X) > 0$;\label{sum}
			
			\item  $\lambda_{E}(5F \cap E) \geq C_1\diam(5F)^s$ for any $E \in \mathcal{E}$ and $C_1 = C_1(s)$;\label{slinear}
			
			\item $\sum_{F' \in \mathcal{F}'} \diam(F')^{\alpha t}  < \eps$;
		\end{enumerate}
		where the constants $C, C_1$ are the same as those in (\ref{um}) and (\ref{linear}).
		
		To construct $\mathcal{F}'$ we start by picking any $B'\in\mathcal{B}'$.
		If the circumscribing ball of $f^{-1}(B')$ satisfies conditions \eqref{sum} and \eqref{slinear},  we include $B'$ in $\mathcal{F}'$.  Otherwise,  we will define a cover $\mathcal{F}'_{B'}$ of $B'$ as follows.  Let
		\[
		B_n' = \left\{x \in B': r_0\left(f^{-1}(x)\right) \geq \frac{5}{n}  \ \textrm{and} \ r_1\left(f^{-1}(x)\right) \geq \frac{5}{n}\right\},
		\]
		where $r_0, r_1$ are as in the statement of the theorem.  Since $f^{-1}$ is uniformly continuous, there exists $\sigma_n>0$ such that
		\[
		\diam\left( f^{-1}\left( B\left(x, \sigma_n\right) \right) \right) < \frac{5}{n}
		\]
		for any ball $B\left(x, \sigma_n\right) \subseteq B'$ with radius $\sigma_n$.  Since $\dim_H(B_n') \cap f(X_1)) < \alpha t$ there exists a sequence of balls $\left\{F'_{n,i}\right\}$ centered at points of $B_n'$ with diameters smaller than $\sigma_n$ such that $\left\{F'_{n,i}\right\}$ covers $B_n'$ and $\sum_{i} \diam(F'_{n,i})^{\alpha t} < {\diam(B')^{\alpha t}}/{2^n}$.
		
		It then follows from inequalities \eqref{um} and \eqref{linear} that for any $F_{n,i}'$, every circumscribing ball of $f^{-1}\left(F'_{n,i}\right)$ satisfies conditions \eqref{sum} and  \eqref{slinear}.
		
		Collecting all such sets $\{F'_{n,i}\}$ for each $n$, we obtain a cover $\mathcal{F}'_{B'}$ of $B'$ such that for any $F' \in \mathcal{F}_{B'}$, every circumscribing ball of $f^{-1}(F')$ satisfies conditions \eqref{sum}, \eqref{slinear} and
		\[
		\sum_{F' \in \mathcal{F}_{B'}} \diam(F')^{\alpha t} < \diam(B')^{\alpha t}.
		\]
		
		Applying the procedure above to every $B' \in \mathcal{B}'$ for which either \eqref{sum} or \eqref{slinear} fails produces the desired cover $\mathcal{F}'$ of $f(X_1)$.
		
		We denote by $\mathcal{F}$ the collection of all circumscribing balls of $f^{-1}(F')$ when $F' \in \mathcal{F}'$. By Lemma \ref{cl} there is a disjoint subcollection $\mathcal{F}_1$of $\mathcal{F}$ such that
		\begin{equation}\label{disjoint}
			\bigcup_{F \in \mathcal{F}} F \subset \bigcup_{F \in \mathcal{F}_1} 5F.
		\end{equation}
		
		We denote 
		\[
		\mathcal{A} = \left\{ 5F: F \in \mathcal{F}_1  \right\} \ \textnormal{and} \ \mathcal{A} '= \left\{ f(5F): F \in \mathcal{F}_1  \right\}.
		\]
		To simplify the notation, for any $A \in \mathcal{A}$ we let $A' = f(A)$.	Since  $F$ is a circumscribing balls of $f^{-1}(F')$ we have that $\diam\left(f^{-1}(F') \right) \leq \diam(F) \leq 2\diam\left(f^{-1}(F') \right)$ for any $F \in \mathcal{F}$. It follows from Proposition \ref{diam} that
		\[
		\frac{\diam(A')}{\diam(F')} \leq \eta \left(2 \frac{\diam(5F)}{\diam\left(f^{-1}(F')\right)} \right) \leq \eta \left(4 \frac{\diam(5F)}{\diam\left(F\right)} \right) \leq \eta(20).
		\]
		Then
		\[
		\diam(A') \leq H \diam(F') < H \delta < 1/j
		\]
		and
		\[
		\sum_{A' \in \mathcal{A}'} \diam(A')^{\alpha t} \leq \sum_{F' \in \mathcal{F}'} (H \diam(F'))^{\alpha t} \leq H^{\alpha t}\eps.
		\]
		
		Next we construct an admissible function $\rho_j$ for $\mathbf{E}_j$ and show that $\int_X \rho_j^\alpha d \mu \lesssim \eps$. Define
		\[
		\rho_j(x) = \sum_i \frac{\diam(A'_i)^t}{\diam(A_i)^s}\frac{\chi_{A_i}(x)}{C_1}.
		\]
		Then for any $E \in \mathcal{E}_j$,
		\begin{align*}
			\int_E \rho_j d \lambda_{E} & = \int_E \sum_i \frac{\diam(A'_i)^t}{\diam(A_i)^s}\frac{\chi_{A_i}(x)}{C_1} d \lambda_{E}
			\\
			& \geq \int_E \sum_{A_i \cap E \neq \emptyset} \frac{\diam(A'_i)^t}{\diam(A_i)^s}\frac{\chi_{A_i}(x)}{C_1} d \lambda_{E}
			\\
			& = \frac{1}{C_1} \sum_{A_i \cap E \neq \emptyset}  \frac{\diam(A'_i)^t}{\diam(A_i)^s} \int_{E \cap A_i}d \lambda_{E}
			\\
			& = \frac{1}{C_1} \sum_{A_i \cap E \neq \emptyset}  \frac{\diam(A'_i)^t}{\diam(A_i)^s} \lambda_{E}\left( E \cap A_i \right)
			\\
			& \geq  \sum_{A'_i \cap f(E) \neq \emptyset}  \diam\left(A'_i\right)^t \geq 1. \tag{by \eqref{linear}}
		\end{align*}
		
		The last inequality comes from the fact that $\diam\left(A'_i\right) < \frac{1}{j}$ for any $A'_i \in \mathcal{A}'$ and $\mathcal{H}^t_{1/j}\left(f(E)\right) \geq 1$. Finally, we will show that $\int_X \rho_j^\alpha d \mu \lesssim \eps$. Recalling that $s > 1$ and $q > \alpha s$ we have the following estimates:
		\begin{align*}
			\int_X \rho_j^\alpha d \mu & = \int_X \left( \sum_i \frac{\diam(A'_i)^t}{\diam(A_i)^s}\frac{\chi_{A_i}(x)}{C_1} \right)^\alpha d \mu
			\\
			& \leq A(\alpha, \mu, C_1) \int_X \left( \sum_i \frac{\diam(A'_i)^t}{\diam(A_i)^s}\chi_{\frac{1}{5}A_i}(x) \right)^\alpha d \mu
			\tag{Theorem \ref{Lpinequality}}
			\\
			& = A(\alpha, \mu, C_1) \sum_i \left( \frac{\diam(A'_i)^t}{\diam(A_i)^s} \right)^\alpha \mu\left( \frac{1}{5}A_i \right) \tag{by (\ref{disjoint})}
			\\
			& \leq A_1(\alpha, \mu, q, C, C_1) \sum_i \diam(A'_i)^{\alpha t} \diam(A_i)^{q-\alpha s} \tag{by (\ref{um})}
			\\
			& \leq A_1(\alpha, \mu, q, C, C_1) \sum_i \diam(A'_i)^{\alpha t}
			\\
			& \leq A_2(\alpha, t, \mu, H,C, C_1) \eps.
		\end{align*}
		
		This implies that $\Mod_\alpha \left( \mathbf{E}_j \right) = 0$ for every $\mathbf{E}_j$ and finishes the proof.
	\end{proof}
	
	\begin{Remark}
		Using the techniques from \cite{BHW16} the assumptions that $X$ is doubling and bounded in Theorem \ref{fmodest} could be dropped, and one could just require that $X$ is a separable metric space instead.   However,  in our applications the extra assumptions are satisfied and we do not need this more general version of the theorem.
	\end{Remark}
	\subsection{Conformal dimension of hierarchical spaces}\label{section:cantor}
	
	Let $E$ be a metric space such that $E = \bigcap_{i=1}^\infty \bigcup_{j=1}^{2^i} E_{i,j}$, where $E_{i,j} \cap E_{i,k} = \emptyset$ for any $j \neq k$. 
	We denote 
	\[
	\mathcal{E}_i = \left\{E_{i,j}  \right\}_{j=1}^{2^i}
	\]
	and refer to this family as the collection of the \emph{$i^{\mathrm{th}}$-generation} elements. Moreover, we let  $\mathcal{E} = \bigcup_{i=1}^\infty \mathcal{E}_i$. We will  assume that  $\mathcal{E}$ has the  { property}  that  for every sequence of non-empty sets $E_1,E_2,\ldots$ in $\mathcal{E}$  with $E_1\supset E_2 \supset \ldots$ we have $\bigcap_{i=1}^{\infty} E_{i}  \neq \emptyset$.
	
	%\hl{We say that $E$ is a metric space with \emph{hierarchical structure} if $E$ is bounded and the following conditions are satisfied. For each $E_{i,j}$, there are}
	%\begin{itemize}
	%	\item \hl{only two children spaces, contained in $E_{i,j}$, denoted by $E_{i+1, 2j-1}, E_{i+1, 2j} \in \mathcal{E}_{i+1}$;}
	%	
	%	\item  \hl{only one parent space, containing $E_{i,j}$, denoted by $\tilde{E}_{i,j} \in \mathcal{E}_{i-1}$;}
	%	
	%	\item \hl{only one sibling space which has the same parent as $E_{i,j}$, denoted by $E'_{i,j} \in \mathcal{E}_i$.}
	%\end{itemize}
	
	We say that a bounded metric space $E$ is equipped with a \emph{hierarchical structure} if each  $E_{i,j}\in\mathcal{E}_i$:
	\begin{itemize}
		\item  contains exactly two  ``children'' $E_{i+1, 2j-1}, E_{i+1, 2j} \in \mathcal{E}_{i+1}$;
		
		\item  is contained in exactly one ``parent'' $\tilde{E}_{i,j} \in \mathcal{E}_{i-1}$;
		
		\item has exactly one ``sibling''  $E'_{i,j} \in \mathcal{E}_i$, such that $E'_{i,j}\subset \tilde{E}_{i,j}$.
	\end{itemize}
	
	We also let $E_{0,1} = E_{1,1} \cup E_{1,2}$.
	
	For any $x \in E$, we denote by $E_i(x)$ the unique $i^{\mathrm{th}}$-generation element that contains $x$.
	
	We say that a metric space $E$ with hierarchical structure is \emph{flat} if for any $x \in E$ and any $r > 0$, there exists $A(x) > 0$ such that
	\[
	\Card\left( \left\{ E_{i,j} :  E_{i,j} \subseteq B(x,r) \ \mathrm{and} \ \tilde{E}_{i,j} \nsubseteq B(x,r) \right\}  \right) \leq A(x).
	\]
	The reason to focus on flatness is that we need some restrictions on the local geometry of $E$. Intuitively, a flat Cantor space does not ``oscillate'' too much in any small neighborhood of any point.
	
	\begin{Theorem}\label{cd1}
		Let $E = \bigcap_{i=1}^\infty \bigcup_{j=1}^{2^i} E_{i,j}$ be a flat metric space with hierarchical structure and $E_{i,j} \subseteq E$ for every $E_{i,j} \in \mathcal{E}$. If for any $x \in E$ we have
		\begin{equation}\label{rdist}
			\lim_{i \to \infty}\Delta(E_{i}(x), E'_{i}(x)) \to 0,
		\end{equation}
		and there exists $L = L(x)$ such that for all $i \in \mathbb{N}$,
		\begin{equation}\label{cdiam}
			\frac{1}{L} \leq \frac{\diam(E_{i}(x))}{\diam(E'_{i}(x))} \leq L.
		\end{equation}
		Then $\dim_C(E) \geq 1$.
	\end{Theorem}
	
	Theorem \ref{cd1} is a generalization of Theorem $3.2$ in \cite{Hak09} and the proof follows the same reasoning. First, we need the following lemma.
	
	\begin{Lemma}\label{mla}
		Let $E$ be a metric space with hierarchical structure satisfying condition \eqref{rdist} and $E_{i,j} \subseteq E$ for every $E_{i,j} \in \mathcal{E}$. Then for any $\eta$-quasisymmetry $f: E \to I$ and any $0 < \alpha < 1$ there exists a probability measure $\mu$ on $I$ such that for any $x \in E$,
		\begin{equation}\label{msubset}
			\mu(f(E_i(x))) \leq C \diam(f(E_i(x)))^{\alpha},
		\end{equation}
		for some constant $0 < C < \infty$ depending on $x,\eta$ and $\alpha$.
	\end{Lemma}
	
	\begin{proof}
		Since $f$ is a homeomorphism and $E$ has a hierarchical structure and $E_{i,j} \subseteq E$, $f(E)$ has a hierarchical structure. The notation used for $E$ will also be used for $f(E)$. In particular, for any $I_{i,j} \subseteq I$ of the form $I_{i,j} = f(E_{ij})$, we denote by $\tilde{I}_{i,j} = f(\tilde{E}_{ij})$ and $I'_{i,j} = f(E'_{ij})$ the parent and the sibling of $I_{i,j}$, respectively. The image of any $n^{\mathrm{th}}$-generation element of $E$ under $f$ is a $n^{\mathrm{th}}$-generation element of $I$. Moreover, we denote $I_i(x) = f(E_i(x))$, $\mathcal{I}_i=\{I_{i,j}\}_{j=1}^{2^i}$, and $\mathcal{I}=\bigcup_{i=1}^{\infty}\mathcal{I}_i$. Without loss of generality, we assume $\diam(I) = 1$.
		
		We define $\mu$ as follows. Let $\mu(I)= 1$. Then for any $I_{i,j} = f(E_{i,j})$, we define
		\begin{equation}\label{measure}
			\mu(I_{i,j}) = \frac{\diam (I_{i,j})^{\alpha}}{\diam (I_{i,j})^{\alpha} + \diam (I'_{i,j})^{\alpha}} \mu(\tilde{I}_{i,j}).
		\end{equation}
		
		Since each nested sequence of $\mathcal{E}$ has nonempty intersection, $\mu$ defines a Borel measure on $I$ by \cite[Theorem~$1.2$]{Mal75}.
		
		For any $n^{\mathrm{th}}$-generation element $I_n \subseteq I$, there exists a unique sequence of nested subsets
		\[
		I_n \subseteq I_{n-1} \subseteq I_{n-2} \subseteq \ldots \subseteq I_1 \subseteq I_0 = I
		\]
		containing it so that $I_i = \tilde{I}_{i+1}$. Therefore, we have
		\begin{align*}
			\frac{\mu(I_n)}{\diam(I_n)^{\alpha}} & = \frac{1}{\diam(I_n)^{\alpha} +\diam(I'_n)^{\alpha}} \cdots \frac{\diam(I_{1})^{\alpha}}{\diam(I_{1})^{\alpha} +\diam(I'_{1})^{\alpha}}\mu(I)
			\\
			& = \frac{\diam(I_{n-1})^{\alpha}}{\diam(I_n)^{\alpha} +\diam(I'_n)^{\alpha}} \cdots \frac{\mu(I)}{\diam(I_{1})^{\alpha} +\diam(I'_{1})^{\alpha}}
			\\
			& = \prod_{i=1}^{n} \frac{\diam(\tilde{I}_i)^{\alpha}}{\diam(I_i)^{\alpha} +\diam(I'_i)^{\alpha}}.
		\end{align*}
		
		Therefore, to prove inequality \eqref{msubset}, it is sufficient to show that for a fixed $f(x) \in \bigcap_{i=1}^\infty I_i$,
		\[
		\prod_{i=1}^{\infty} \frac{\diam(\tilde{I}_i)^{\alpha}}{\diam(I_i)^{\alpha} +\diam(I'_i)^{\alpha}}  < \infty.
		\]
		
		We claim that for $i$ sufficiently large,
		\begin{equation}\label{gdc}
			\frac{\diam(\tilde{I}_i)^{\alpha}}{\diam(I_i)^{\alpha} +\diam(I'_i)^{\alpha}} \leq 1.
		\end{equation}
		
		Let $I_i = f(E_i)$, $I'_i = f(E'_i)$ and $\tilde{I}_i = f(\tilde{E}_i)$. Since $\Delta(E_i, E'_i) \to 0$ as $i \to \infty$, we have $\Delta(I_i, I'_i) \to 0$ by Lemma \ref{relativedistance} and the definition of relative distance. Without loss of generality, we assume that $\diam(I_i) \geq \diam(I'_i)$. Since $\Delta(I_i, I'_i) \to 0$, we may assume that for sufficiently large $i$,
		\[
		\diam(I_i) \geq \frac{1}{3}\diam(\tilde{I}_i)
		\]
		and
		\[
		\diam(I_i) + \diam(I'_i)/(1-\eps) > \diam(\tilde{I}_i)
		\]
		for some $\eps > 0$.
		
		Let $p = \frac{\diam(I_i)}{\diam(\tilde{I}_i)}$, then $\frac{\diam(I'_i)}{\diam(\tilde{I}_i)} > (1-p)(1-\eps) $. It is then sufficient to prove that
		\begin{equation}\label{mde}
			p^{\alpha} + (1-p)^{\alpha}(1-\eps)^{\alpha} \geq 1
		\end{equation}
		for $p \in \left(\frac{1}{3}, 1\right)$.
		
		It is clear that the left hand side of inequality \eqref{mde} attains its minimum at either $p=\frac{1}{3}$ or $1$. When $i$ is sufficiently large, we may let $\eps$ be small enough such that
		\[
		\left( \frac{1}{3} \right)^{\alpha} + \left( \frac{2}{3} \right)^{\alpha}(1-\eps)^{\alpha} \geq 1.
		\]
		Then $p^{\alpha} + (1-p)^{\alpha}(1-\eps)^{\alpha} \geq 1$ for $p=\frac{1}{3}$ or $1$. This finishes the proof.
	\end{proof}
	
	We are now ready to prove Theorem \ref{cd1}
	
	\begin{proof}[Proof of Theorem \ref{cd1}]
		Let $f: E \to I$ be an $\eta$-quasisymmetry and $\alpha < 1$. Fix $\mu$ from Lemma \ref{mla}. Define
		\[
		E^*(N) = \left\{x \in E: C(x, \eta, \alpha) \leq N, A(x) \leq N, \delta(x) \leq N, L(x) \leq N \right\}
		\]
		where $A(x)$ is the flatness constant of $x$ in $E$, $C(x,\eta, \alpha)$ is the constant  in \eqref{msubset}, $L(x)$ is the constant  in \eqref{cdiam}, and $\delta(x) = \max_{i}\Delta(E_{i}(x), E'_{i}(x))$. Since $\Delta(E_i(x), E'_i(x)) \to 0$ for any $x \in E$, we have $\delta(x) < \infty$ for any $x \in E$.
		
		Let
		\[
		E(N) = \bigcap_{n=1}^\infty \bigcup_{x \in E^*(N)} B\left(x,\frac{1}{n}\right).
		\]
		Then $E(N)$ is a Borel measurable set containing $E^*(N)$.
		
		To simplify notation, we write $I^*(N) = f(E^*(N))$ and $I(N) = f(E(N))$. Notice that $I = \bigcup_{N=1}^\infty I(N)$ since $E = \bigcup_{N=1}^\infty E(N)$.
		
		We now prove that for any $x \in E(N)$, there exists a $G = G(\eta, N)$ such that 	for any $r > 0$,
		\begin{equation}\label{umb}
			\mu\left(B(f(x),r) \cap I(N)\right) \leq  Gr^{\alpha}.
		\end{equation}
		
		Given $x \in E(N)$ and $r > 0$, we denote by $\mathcal{I}(x,r,N)$ the collection of all the elements $I_i \in \mathcal{I}$ satisfying the following conditions:
		\begin{enumerate}
			\item $I_{i} \subseteq B(f(x),r)$,
			
			\item $\tilde{I}_{i} \nsubseteq B(f(x),r)$,
			
			\item $I_{i} \cap I^*(N) \neq \emptyset$.
		\end{enumerate}
		
		We observe that if $I_i, I_j \in \mathcal{I}(x,r,N)$ then either $I_i \cap I_{j} = \emptyset$ or $I_i = I_{j}$.
		We claim that
		\[
		B(f(x), r) \cap I(N) \subseteq \bigcup_{I_i \in \mathcal{I}(x,2r,N)}\tilde{I}_i.
		\]
		For any $p \in E$ and $f(p) \in B(f(x),2r) \cap I^*(N)$, there exists $M \in \mathbb{N}$ such that $I_n(p) \subseteq B(f(x),2r)$ when $n \leq M$ and $I_n(p) \nsubseteq B(f(x),2r)$ when $n > M$. Then $I_M(p) \in \mathcal{I}(x,2r, N)$ which implies that $f(p) \in \bigcup_{I_i \in \mathcal{I}(x,2r,N)}\tilde{I}_i$.
		
		Notice that for any $\eps >0$, we have $B(f(x),r) \cap I(N) \subseteq B(f(x), r+\eps) \cap I^*(N) $ by the definition of $E(N)$. Then we have
		\[
		B(f(x),r) \cap I(N) \subseteq B(f(x), 2r) \cap I^*(N) \subseteq \bigcup_{I_i \in \mathcal{I}(x, 2r,N)}\tilde{I}_i.
		\]
		
		Since $E$ is a flat space, $\Card\left( \mathcal{I}(x,2r,N) \right) \leq N$ by the definition of $\mathcal{I}(x,2r,N)$. Since $I_i \cap I^*(N) \neq \emptyset$ for any $I_i \in \mathcal{I}(x,2r,N)$, there exists $M = M(\eta, N)$ such that
		\[
		\frac{1}{M} \leq \frac{\diam(I_i)}{\diam(I'_i)} \leq M
		\]
		and
		\[
		\Delta(I_i, I'_i) \leq M
		\]
		by Proposition \ref{diam}. Therefore, we have
		\begin{align*}
			\diam(\tilde{I}_i) & \leq \diam(I_i) + \diam(I'_i) + \dist(I_i, I'_i)
			\\
			& \leq \diam(I_i) + \diam(I'_i) + M\diam(I_i)
			\\
			& \leq (2M+1) \diam(I_i).
		\end{align*}
		
		Then
		\begin{align*}
			\mu(B(f(x),r) \cap I(N)) & \leq \sum_{I_i \in \mathcal{I}(x,2r,N)}\mu(\tilde{I}_i)
			\\
			& \leq \sum_{I_i \in \mathcal{I}(x,2r,N)}N \diam(\tilde{I}_i)^\alpha
			\\
			& \leq \sum_{I_i \in \mathcal{I}(x,2r,N)}N \left( (2M+1) \diam(I_i)\right)^\alpha
			\\
			& \leq 8^\alpha N^2 (2M+1)^\alpha r^\alpha
		\end{align*}
		
		The second inequality follows from Lemma \ref{mla}. The third inequality comes from the fact that $\diam(\tilde{I}_i) \leq (2M+1) \diam(I_i)$, while the last inequality is a consequence of the fact that $\Card\left( \mathcal{I}(x,2r,N) \right) \leq N$ and $I_{i} \subseteq B(f(x),2r)$ for any $I_i \in \mathcal{I}(x,2r,N)$. This proves \eqref{umb}.
		
		Since $X = \bigcup_{N=1}^\infty I(N)$ and $I(N-1) \subseteq I(N)$ for any $N$, it is clear that $\mu(I(N)) > 0$ when $N$ is sufficiently large. It then follows from the Mass Distribution Principle \ref{mdp} that $\dim_H(X) \geq \dim_H(I(N)) \geq \alpha$ for sufficiently large $N$. Since $\alpha<1$ was arbitrary, we obtain $\dim_H(X)\geq 1$.
	\end{proof}
	
	One may wonder if  conditions \eqref{rdist} and \eqref{cdiam} of Theorem \ref{cd1}  can be weakened.  The following theorem provides one answer to this question.
	
	\begin{Theorem}\label{unioncd1}
		Let $\mathcal{E} = \{E^n\}_{n=1}^\infty$ be a collection of flat metric spaces with hierarchical structure,  where $E^n = \bigcap_{i=1}^\infty \bigcup_{j=1}^{2^i} E^n_{i,j}$ and $E^n_{i,j} \subseteq E^n$ for every $E^n_{i,j}$. Assume that $E^n_{i,j} \subseteq E^{n+1}_{i,j}$ for all $i,j,n$, and that for any $x \in E^n$, the flatness constants $A^n(x)$ are uniformly bounded above by some $A(x) < \infty$.
		
		Assume for any $E^n \in \mathcal{E}$ and for any $x \in E^n$, we have
		\begin{equation}\label{ssdist}
			\lim_{i \to \infty}\lim_{n \to\infty}\Delta\left(E^n_{i}(x), (E^n_{i})'(x)\right) \to 0
		\end{equation}
		and there exists $L = L(x) < \infty$ such that for all $i \in \mathbb{N}$
		\begin{equation}\label{ssudiam}
			\frac{1}{L} \leq \lim_{n \to \infty}\frac{\diam(E^n_{i}(x))}{\diam((E^n_{i})'(x))} \leq L.
		\end{equation}
		Let $\widetilde{E} = \bigcup_{n=1}^\infty E^n$ and $\diam\left( \widetilde{E}\right) < \infty$, then $\dim_C\left(\widetilde{E}\right) \geq 1$.
	\end{Theorem}
	
	\begin{proof}
		The idea of the proof is similar to the proofs of Lemma \ref{mla} and Theorem \ref{cd1}. We first construct a sequence of measures  $\mu_n$ analogous to Equation \eqref{measure} on each $E^n$ and then show that $\mu_n$ subconverges to a measure $\mu$ on $\widetilde{E}$ weakly. Finally the proof is finished by applying Mass Distribution Principle \ref{mdp} to $\mu$.
		
		Let $f: \widetilde{E} \to \widetilde{I}$ be an $\eta$-quasisymmetry and $0 < \alpha < 1$. Fix a $E \in \mathcal{E}$ and let $I = f(E)$. The notation in this proof is analogous to the one used in the proof of Lemma \ref{mla} and Theorem \ref{cd1}. Without loss of generality, we let $\diam(I) = \diam(f(E)) = 1$.
		
		We construct a measure on $E$ in the same way as Equation \eqref{measure}, i.e.,
		\begin{equation}
			\mu(I_{i,j}) = \frac{\diam (I_{i,j})^{\alpha}}{\diam (I_{i,j})^{\alpha} + \diam (I'_{i,j})^{\alpha}} \mu(\tilde{I}_{i,j}).
		\end{equation}
		Notice that $\diam(I_{i,j})$ could be $0$ here, however, our definition is still valid.
		
		Thus there exists a sequence of measures $\mu_n$ constructed as above that is supported on $E^n$ for every $n \in \mathbb{N}$.
		
		Fix $x \in \widetilde{E}$, then $x \in E^n$ for sufficiently large $n$. We claim that there exists $N = N(x, i)$ such that when $n \geq N(x, i)$,
		\begin{equation}\label{einterval}
			\mu_n(f(E^n_i(x))) \leq C \diam(f(E^n_i(x)))^{\alpha}
		\end{equation}
		for some constant $C > 0$ which depends on $x,\eta$ and $\alpha$.
		
		We denote by $\sigma(x, i) = \lim_{n \to\infty}\Delta(E^n_{i}(x), (E^n_{i})'(x))$ and it is clearly that $\sigma(x,i) < \infty$ for $i$ sufficiently large. Since $E^n_i(x) \subseteq E^{n+1}_i(x)$ and $(E^n_i)'(x) \subseteq (E^{n+1}_i)'(x)$ for any $n \in \mathbb{N}$, we have
		\[
		\Delta(E^{n+1}_{i}(x), (E^{n+1}_{i})'(x)) \leq \Delta(E^n_{i}(x), (E^n_{i})'(x)).
		\]
		Then there exists $N = N(x,i)$ such that $\Delta(E^n_{i}(x), (E^n_{i})'(x)) < 2\sigma(x,i)$ when $n \geq N$.
		
		Since $\sigma(x,i) \to 0$ when $i \to \infty$ by equation \eqref{ssdist}, there exists an $M= M(x)$ such that when $i \geq M(x)$ and $n \geq N(x,i)$, the following conditions are satisfied: For any $I^n_i = f(E^n_i), (I^n_i)' = f((E^n_i))' $ with $\diam(I^n_i) \geq \diam((I^n_i)')$ and $f(x) \in \tilde{I}^n_i$,
		\begin{enumerate}
			\item $\diam(I^n_i) \geq \frac{1}{3}\diam(\tilde{I}^n_i)$,
			
			\item $\diam(I^n_i) + \diam((I^n_i)')/(1-\eps) > \diam(\tilde{I}^n_i)$ for some $\eps > 0$,
			
			\item $\left( \frac{1}{3} \right)^{\alpha} + \left( \frac{2}{3} \right)^{\alpha}(1-\eps)^{\alpha} \geq 1$.
		\end{enumerate}
		
		Notice that these are all the conditions we need to prove inequality \eqref{gdc} and $N, M, \eps$ are independent of $n$. Moreover, we have for any $n > N(x,i)$,
		\begin{equation}
			\frac{\diam(\tilde{I}^n_i)^{\alpha}}{\diam(I^n_i)^{\alpha} +\diam((I^n_i)')^{\alpha}} \leq C_1
		\end{equation}
		for some $C_1$ depending on $\delta(x,i)$ and $f$ since $\Delta(E^n_{i}(x), (E^n_{i})'(x)) < 2\sigma(x,i)$.
		
		It then follows from the proof of Lemma \ref{mla} that for any $x \in E$, and any $i \in \mathbb{N}$, there exists $N(x,i)$ such that for $n > N(x,i)$ we have
		\begin{equation}
			\mu_n(f(E^n_i(x))) \leq C \diam(f(E^n_i(x)))^{\alpha},
		\end{equation}
		for some constant $C > 0$ which depends on $x,\eta$ and $\alpha$.
		
		To simplify the notation, we define $\widetilde{E}_{i,j} = \bigcup_{n=1}^\infty E^n_{i,j}$. Then $\widetilde{E} = \bigcap_{i=1}^\infty \bigcup_{j=1}^{2^i} \widetilde{E}_{i,j}$. Notice that $\widetilde{E}$ may not have a hierarchical structure since not every nested sequence of $\widetilde{E}_{i,j}$ necessarily has a nonempty intersection. 
		
		For any $x \in \widetilde{E}$, recall that $\widetilde{E}_i(x)$ is the unique $\widetilde{E}_{i,j}$ that contains $x$ and $\widetilde{E}'_i(x)$ is the sibling space of $\widetilde{E}_i(x)$.
		
		Then we have
		\begin{equation}\label{gcdist}
			\lim_{i \to \infty}\Delta\left(\widetilde{E}_{i}(x), \widetilde{E}'_{i}(x)\right) = 0
		\end{equation}
		and there exists $L = L(x) < \infty$ such that
		\begin{equation}\label{gcdiam}
			\frac{1}{L} \leq \frac{\diam(\widetilde{E}_{i}(x))}{\diam(\widetilde{E}'_{i}(x))} \leq L.
		\end{equation}
		
		Since $\mu_n$ is a finite measure for any $n \in \mathbb{N}$, there exists a subsequence of $\mu_n$ that  converges weakly to a measure $\mu$ on $\widetilde{E}$ by \cite[Lemma $8.24$ and $8.26$]{DS97}. Moreover, we have that
		$E^n_i(x)$ converges to $\widetilde{E}_i(x)$ in the sense of definition $8.9$ in \cite{DS97}. Thus
		\begin{equation}\label{gcmdp}
			\mu(f(\widetilde{E}_i(x))) \leq C \diam(f(\widetilde{E}_i(x)))^{\alpha}
		\end{equation}
		for any $\widetilde{E}_i(x)$ by the proof of \cite[Lemma 8.28]{DS97}.
		
		Recall that
		\begin{equation}\label{gcfc}
			A(x)  < \infty.
		\end{equation}
		
		Given equation \eqref{gcdist} and inequalities \eqref{gcdiam}, \eqref{gcmdp}, \eqref{gcfc}, it then follows the proof of Theorem \ref{cd1} verbatim that $\dim_H(f(\widetilde{E})) \geq \alpha$.
	\end{proof}
	
	\begin{Theorem}\label{unioncd1new}
		Let $\mathcal{E} = \{E^n\}_{n=1}^\infty$ be a collection of flat metric spaces with hierarchical structure,  where $E^n = \bigcap_{i=1}^\infty \bigcup_{j=1}^{2^i} E^n_{i,j}$ and $E^n_{i,j} \subseteq E^n$ for every $E^n_{i,j}$. Assume for any $E^n \in \mathcal{E}$ and for any $x \in E^n$, we have
		\begin{equation}\label{srdist}
			\lim_{n \to \infty}\lim_{i \to \infty}\Delta(E^n_{i}(x), (E^n_{i})'(x)) = 0
		\end{equation}
		and there exists $L = L(x, n)$ such that for all $i \in \mathbb{N}$
		\begin{equation}\label{scdiam}
			\frac{1}{L} \leq \frac{\diam(E^n_{i}(x))}{\diam((E^n_{i})'(x))} \leq L.
		\end{equation}
		Let $\widetilde{E} = \bigcup_{n=1}^\infty E^n$, then $\dim_C\left(\widetilde{E}\right) \geq 1$.
	\end{Theorem}
	
	\begin{proof}
		The proof here shares the same idea as the proof of Lemma \ref{mla} and Theorem \ref{cd1}. We construct a measure  $\mu$ analogous to Equation \eqref{measure} on a specific $E \in \mathcal{E}$ and then  applying Mass Distribution Principle \ref{mdp} to $\mu$ to finish the proof.
		
		Let $f: \widetilde{E} \to \widetilde{I}$ be an $\eta$-quasisymmetry and $0 < \alpha < 1$. Fixed a $E \in \mathcal{E}$ and let $I = f(E)$.  The notation in this proof is analogous to the one used in the proof of Lemma \ref{mla} and Theorem \ref{cd1}. Without loss of generality, we let $\diam(I) = \diam(f(E)) = 1$.
		
		We may let $\lim_{i \to \infty}\Delta(E_{i}(x), E'_{i}(x))$ be sufficiently small for our chosen $E \in \mathcal{E}$. Then for any $I_i, I'_i \subseteq I$ with $\diam(I_i) \geq \diam(I'_i)$, we have for sufficiently large $i$,
		\begin{enumerate}
			\item $\diam(I_i) \geq \frac{1}{3}\diam(\tilde{I}_i)$,
			
			\item $\diam(I_i) + \diam(I'_i)/(1-\eps) > \diam(\tilde{I}_i)$ for some $\eps > 0$,
			
			\item $\left( \frac{1}{3} \right)^{\alpha} + \left( \frac{2}{3} \right)^{\alpha}(1-\eps)^{\alpha} \geq 1$.
		\end{enumerate}
		
		Notice that these are all the conditions we need to prove inequality \eqref{gdc}. It then follows from the proof of Lemma \ref{mla} that there exists a measure $\mu$ on $I$ such that for any $x \in E$,
		\begin{equation}\label{smsubset}
			\mu(f(E_i(x))) \leq C \diam(f(E_i(x)))^{\alpha}
		\end{equation}
		for some constant $C > 0$, which depends on $x,\eta, \alpha$ and $E$.
		
		Given equation  \eqref{srdist} and inequalities \eqref{scdiam}, \eqref{smsubset}, it then follows the proof of Theorem \ref{cd1} verbatim that $\dim_H(I) \geq \alpha$.
	\end{proof}
	
	\section{Bedford-McMullen sets with uniform fibers are minimal}\label{MSS}
	
	In this section we construct subsets of the plane which are minimal for conformal dimension and prove Theorem \ref{thm:uniformfibers}, which implies Corollary \ref{thm:mg}.  The minimal subsets constructed below include many self-affine sets, graphs of functions (continuous or not) and can attain any dimension $d$ in $[1,2]$.
	
	\subsection{Bedford-McMullen sets with uniform fibers}  We recall the construction of Bedford-McMullen sets as described in the introduction, meanwhile generalizing it and introducing some notation needed for the proofs below.
	
	Let  ${\ell,m}\in\mathbb{N}$  be positive integers so that $\ell<m$.  Suppose we are given a sequence $\{D_i\}_{i=1}^{\infty}$ of subsets (which we also call patterns) of $\{1,\ldots,m\}\times \{1,\ldots,\ell\}$.   For every $n\geq 1$ we let $N_n=\Card(D_1)\cdot\ldots\cdot\Card(D_n)$.  In particular,  $N_1=\Card(D_1)$.
	
	Define a compact set $K=K(\{D_i\})\subset[0,1]^2$ as follows.  Divide the unit square into $m \ell$  closed and congruent $m^{-1}\times \ell^{-1}$ rectangles with pairwise disjoint interiors.  Keep the rectangles corresponding to the pattern $D_1$,  and delete the rest.  Let $\mathcal{K}_1=\{K_{1,1},\ldots,K_{1,N_1}\}$  be the collection of  retained rectangles of \emph{generation $1$},  and let $K_1=\cup_j K_{1,j}$.  Continuing  by induction,  suppose that a collection $\mathcal{K}_n=\{K_{n, 1},\ldots,K_{n,N_n}\}$ of  \emph{generation $n$} rectangles has been defined,  and its elements are  closed and congruent $m^{-n}\times \ell^{-n}$ rectangles with disjoint interiors.  Let $K_n = \cup_{j} K_{n,j}$.  Divide each $K_{n,j}\in\mathcal{K}_n$  into $m \ell$  congruent, closed $m^{-(n+1)}\times \ell^{-(n+1)}$ rectangles with disjoint interiors.  In each $K_{n,j}\in\mathcal{K}_n$ keep the rectangles corresponding to the pattern $D_{n+1}$,  and delete the rest.   The resulting collection $\mathcal{K}_{n+1}=\{K_{n+1,1},\ldots,K_{n+1,N_{n+1}}\}$ is the family of generation $(n+1)$ rectangles,  and their union will be denoted by $K_{n+1}$.  We define the set $K$ as follows:
	\begin{align}\label{def:det-set}
		K=K(\{D_n\}_{n=1}^{\infty})=\bigcap_{n=1}^{\infty} K_n = \bigcap_{n=1}^{\infty} \bigcup_{j=1}^{N_n}K_{n,j}=\bigcap_{n=1}^{\infty} \bigcup_{K_{n,j}\in\mathcal{K}_n}K_{n,j}.
	\end{align}
	
	Recall that if all the patterns in the construction above are the same,  i.e.,  there is a $D\subset\{1,\ldots,m\}\times \{1,\ldots,\ell\}$ such that $D_i=D$, for all $i\geq 1$,  then the set (\ref{def:det-set}) is called a Bedford-McMullen carpet,  and is denoted by $K(D)$, see also \cite{BP17}.
	
	The construction above can be generalized by dividing each rectangle of generation $i$ into $m_i \ell_i$ congruent rectangles where $\{\ell_i\}$ and $\{m_i\}$ are arbitrary sequences of integers such that $1<\ell_i<m_i$.  We do not consider this more general situation,  since we are mostly concerned with the analogy with the Brownian graph for which the motivating example can be thought to be the case when $\ell_i=2$ and $m_i=4$ for $i\geq 1$.
	
	Computing conformal dimension of sets defined by (\ref{def:det-set}) can be challenging for general sequences of patterns $\{D_i\}$.  For instance,  even if $K=K(D)$ the conformal dimension of $K$ is unknown, unless we make the extra assumption below.
	
	Inspired by the terminology in \cite{BP17} we will say that a set $K$ constructed as above  has \emph{uniform fibers} if there exists $1\leq k < m$ such that for every $i\geq 1$ the pattern $D_i$ has $k$ elements in \emph{every} row.  Therefore,   for every $n\geq 1$ we have $\Card(D_n)=k\ell$ and $N_n=(k\ell)^n$.
	
	\begin{Remark}\label{remark:different-patterns}
		A slightly more general construction of sets with uniform fibers can be obtained if one allows to replace each rectangle $K_{n,j}$,  of generation $n$ by an \emph{arbitrary} pattern $D_{n+1,j}\subset\{1,\ldots,m\}\times \{1,\ldots,\ell\}$ of $k\ell$ rectangles, as long as  $D_{n+1,j}$ has $k$ elements in every row for all $j\in\{1,\ldots,(k\ell)^n\}$. For such sets Proposition \ref{prop:dim} and Theorem \ref{thm:uniformfibers} below hold true, and the proofs are exactly the same. This more general construction will be used below to prove Corollary \ref{thm:mg} about the existence of minimal graphs of every dimension $d\in[1,2]$. 
	\end{Remark}
	
	Note, that for every $n\geq1$ and every $\ell$-adic interval $J$ of generation $n$ on the $y$-axis there are exactly $k^n$ rectangles $K_{n,j}$  such that $\pi_y(K_{n,j})=J$, where $\pi_y$ is the orthogonal projection onto the second coordinate.  We will say that two rectangles $K_{n,j_1}$ and $K_{n,j_2}$ are \emph{at the same height} if they have the same projections:  $\pi_y(K_{n,j_1})=\pi_y(K_{n,j_2})$.
	
	The following proposition can be proved like the analogous property for self-affine sets in \cite{BP17}.  We provide a slightly different proof for the convenience of the reader.
	
	\begin{Proposition}\label{prop:dim}
		Let $1\leq k<m$ and $d=\log_mk$.  If  $K=K(\{D_i\})$ is a Bedford-McMullen set with uniform fibers then
		\begin{enumerate}
			\item $\dim_H K = 1+d$, \label{spacedim}
			
			\item $\dim_H K \cap Z_a = d$ for almost every $a\in[0,1]$, where $Z_a$ is the horizontal line passing through the point $(0,a)$.\label{fiberdim}
		\end{enumerate}
	\end{Proposition}
	
	\begin{proof}
		We first prove \eqref{fiberdim}.  Pick $a\in[0,1]$ which cannot be written as $\mathscr{i}/\ell^n$ for some  $\mathscr{i},n\in\mathbb{N}$. Note, that the collection of such $a$'s has full measure in $[0,1]$ (in fact the complement is countable). 
		
		The uniform fibered condition implies that the set $K_a:=K\cap Z_a$ is a Cantor set obtained by dividing $[0,1]$ into $m$ equal intervals, selecting $k$ of them, and then repeating this procedure ad infinitum. Thus,  for every $a$ as above there is a set of indices $J_n(a)\subset \{1,\ldots,N_n\}$ such that $\mathrm{Card}(J_n(a))=k^n$,   $\diam(K_{n,j}\cap Z_a)=m^{-n}$ for $j\in J_n(a)$, and
		$$K_a=\bigcap_{n=1}^{\infty} \bigcup_{j\in J_n(a)} K_{n,j}\cap Z_a.$$
		Therefore, $\dim_H (K_a) \leq \overline{\dim}_M (K_a) \leq \lim_{n\to\infty} {\log k^n }/({\log m^n})=\log_m k$, where $\overline{\dim}_M K_a$ denotes the upper Minkowski dimension of $K_a$.
		
		To show the opposite inequality we define a probability measure $\mu_a$ on $K_a$, by setting
		\begin{align}\label{def:conditional-measures}
			\mu_a(K_{n,j}\cap Z_a) = k^{-n}=(m^{-n})^d=\left(\diam\left(K_{n,j}\cap Z_a\right)\right)^d,
		\end{align}
		for every $j\in J_n(a)$. Therefore, if  $I\subset K\cap Z_a$ is such that $I=K_{n,j}\cap Z_a$ for some $j\in J_n(a)$ then
		$\mu_a(I) =k^{-n}=(m^{-n})^d=(\diam (I))^d$. More generally, for every $I\subset K\cap Z_a$ such that $m^{-(n+1)}<\diam (I) \leq m^{-n}$ there are at most two ``adjacent'' $m$-adic intervals $I'$ and $I''$ in $Z_a$ such that  $I\subset I'\cup I''$ and $\diam(I')=\diam (I'')=m^{-n}$.  Therefore,  we have
		\begin{align}\label{ineq:growth1}
			\mu_a (I) \leq \mu_a (I')+\mu_a (I'') \leq 2 \diam(I')^d < 2 m^d (\diam (I))^d,
		\end{align}
		where the last inequality holds since $\diam (I) > 1/m^{n+1} = \diam (I') /m$. By Mass Distribution Principle \ref{mdp} we have $\dim_H K_a\geq d$, which finishes the proof of \eqref{fiberdim}.
		
		To prove \eqref{spacedim} first note that from \eqref{fiberdim} and Marstrand slicing theorem, see e.g. \cite[Thm. 1.6.1]{BP17}, it follows that $\dim_H K\geq 1+\log_mk$.  To get the upper bound in \eqref{spacedim}, note that for every $n\geq 1$ the set $K$ can be covered by $(\Card(D))^n=(k\ell)^n$ generation $n$ rectangles of width $1/m^n$ and height $1/\ell^n$.  Dividing each such rectangle into $\lesssim (m/\ell)^n$ squares of side-length $1/m^n$ we obtain that
		\begin{align*}
			\dim_H K\leq\overline{\dim}_M K \leq \lim_{n\to\infty} \frac{\log((k\ell)^n \cdot (m/\ell)^n)}{\log m^n} =
			1+\log_m k,
		\end{align*}
		which completes the proof.
	\end{proof}
	
	The following is the main result of this section.
	
	\begin{Theorem}\label{thm:uniformfibers}
		Every  Bedford-McMullen set with uniform fibers $K$ is minimal for conformal dimension.
	\end{Theorem}
	
	Observe that Theorem \ref{thm:minimalBM} follows from Theorem \ref{thm:uniformfibers} and Proposition \ref{prop:dim}.
	
	To prove Theorem \ref{thm:uniformfibers} we would like to use Theorem \ref{fmodest}. For this we will construct the measure $\mu$, families $\mathcal{E}$ and $\{\lambda_E\}_{E\in\mathcal{E}}$ so that conditions of that theorem are satisfied.
	
	\subsection{Constructing the measure $\mu$}\label{section:mu} The measure $\mu$ is defined by equidistributing the mass among all the rectangles of the same generation.  More specifically, let $\mu(K)=1$.  For every rectangle $K_{n,j}\in\mathcal{K}_n$ of generation $n$ let
	\begin{align}\label{def:mu}
		\mu\left(K_{n,j}\cap K\right) = (k\ell)^{-n}.
	\end{align}
	This procedure defines a probability Borel measure $\mu$ on the compact set $K\subset\mathbb{R}^2$  which is therefore a Radon measure, see e.g., \cite[Theorem 2.18]{Rud86}.
	Note that if $I$ is an $\ell$-adic interval in $[0,1]$ of generation $n$ then there are $k^n$ generation $n$ rectangles in $\mathcal{K}_n$ which project under $\pi_y$ to $I$.  Therefore we have
	\begin{align}
		\mu(\pi^{-1}_y(I)) = k^n (1/(k\ell)^n) = 1/(\ell)^n = \mathcal{L}^1(I).
	\end{align}
	This implies that the pushforward of $\mu$ onto the $y$-axis is the restriction of the Lebesgue measure on $\{0\}\times[0,1]$, that is  $\pi_y^*(\mu)=\mathcal{L}^1$.
	By the Disintegration Theorem \cite[Theorem $5.3.1$]{AGS05} there is a family of probability measures $\{\nu_a\}_{a\in[0,1]}$,  which is $\mathcal{L}^1$-a.e. uniquely determined  and is supported on the sets $K_a$ s.t.
	\begin{align}
		\mu(A) = \int_0^1 \nu_a\left(A\cap \pi^{-1}_y(a)\right) \, da.
	\end{align}
	
	It turns out that in this case the family $\{\nu_a\}$ a.e.  coincides with $\{\mu_a\}$ defined in equation \eqref{def:conditional-measures}.
	\begin{Lemma}
		For every Borel set $A\subset K$ we have
		\begin{align}\label{eq:disint}
			\mu(A) = \int_0^1 \mu_a\left(A\cap \pi^{-1}_y(a)\right) \, da,
		\end{align}
		where $\mu_a$ satisfies (\ref{def:conditional-measures}).
	\end{Lemma}
	\begin{proof}
		If $A=K_{n,j}\in\mathcal{K}_n$ for some $n\geq 1$,  then by  (\ref{def:mu}) and (\ref{def:conditional-measures})  we have
		\begin{align*}
			\int_{0}^1 \mu_a \left(A\cap\pi_y^{-1}(a)\right) da = \int_{\pi_{y}(A)}k^{-n} da = k^{-n} \mathcal{L}^1(\pi_{y}(A)) = k^{-n}\ell^{-n}=\mu(A).
		\end{align*}
		Therefore,  (\ref{eq:disint}) holds if $A=\cup_{j\in J} K_{n,j}$, whenever $J$ is any subset of $\{1,\ldots,N_n\}$.  By outer regularity of $\mu$, equation (\ref{eq:disint}) holds for all Borel sets $A\subset K$ as well.
	\end{proof}
	
	Equation (\ref{eq:disint})  implies that for every $x\in K$ and $r>0$ the measure $\mu$ satisfies
	\begin{align}
		\mu(B(x,r))\lesssim r^{1+d}.
	\end{align}
	where,  as in Proposition \ref{prop:dim},  $d=\log_m k$.  Indeed,   using (\ref{eq:disint}) and  (\ref{ineq:growth1}),  for any $B=B(x,r)\subset K$  we have
	\begin{align}
		\mu(B)=\int_{\pi_{y}(B)} \mu_a\left(B\cap Z_a\right) \ da \lesssim \mathcal{L}^1(\pi_y(B)) \diam(B)^d \lesssim\diam(B)^{1+d}.
	\end{align}
	
	\subsection{Constructing the ``vertical'' Cantor sets and the associated measures}\label{section:vertical}
	
	We construct a family of Cantor sets $\mathcal{E}=\{E\}$ in $K$ as follows.  From each of the $\ell$ rows choose \emph{one} $m^{-1}\times \ell^{-1}$ rectangle of generation $1$ out of $k$ possible ones.  Thus, we chose $\ell$ rectangles $E_{1,1},\ldots,E_{1,\ell}$ from $\mathcal{K}_1$, so that their $\pi_y$ projections cover the entire vertical interval $\{0\}\times [0,1]$.  Next from every row of every rectangle $E_{1,j}$,  choose \emph{one} rectangle of generation $2$ from $\mathcal{K}_2$. The resulting rectangles will be denoted $E_{2,j}$ with $j=1,\ldots, \ell^2$.  Continuing by induction, in the $n$-th step we obtain rectangles $E_{n,j}, j=1,\ldots,\ell^n$ which have pairwise disjoint interiors.  Moreover, for every $E_{n,j}$ we have
	\begin{align}
		\pi_y(E_{n,j})=\left[\frac{j-1}{\ell^{n}},\frac{j}{\ell^{n}}\right]
	\end{align}  
	A \emph{vertical Cantor set}  $E\subset  K$ then can be defined as usual by 
	\begin{align}
		E=\bigcap_{n=1}^{\infty} \bigcup_{j=1}^{\ell^n}E_{n,j}.
	\end{align}  
	The family of all vertical Cantor sets $E$ in $K$ as above will be denoted by $\mathcal{E}$.
	
	It follows from the definitions that the restriction of $\pi_y$ onto every vertical Cantor set $E$ is injective except for at the preimages of the $\ell$-adic points,  at which it is two-to-one.  Hence we can define a measure $\lambda_{E}$ by pushing forward the Lebesgue measure from $\{0\}\times [0,1]$ to $E$. That is $\lambda_{E}=(\pi_y^{-1})_*(\mathcal{L}^1)$,  or more concretely for every $A\subset E$ we have
	\begin{align*}
		\lambda_{E}(A)=\mathcal{L}^1(\pi_y(A)).
	\end{align*}
	We will denote $\mathbf{E}=\{\lambda_E\}_{E\in\mathcal{E}}$, i.e., the family of pull backs of the Lebesgue measure under $\pi_y$ to all the ``vertical'' Cantor sets in $K$.
	
	Next, we observe that every $\lambda_{E}\in\mathbf{E}$ satisfies (\ref{linear}).  Indeed,  let $x \in E$ and $r  > 0$. Then there exists $n \in \mathbb{N}$ such that $1/\ell^{n+1} < r \leq 1/\ell^{n}$.  Choose $j$ so that  $x\in E_{n+2,j} \in \mathcal{K}_{n+2}$. Since $\diam(E_{n+2,j}) < 1/\ell^{n+1}$, it follows that $E_{n+2,j} \subseteq B(x,r)$.  Therefore,
	\[
	\lambda_E(B(x,r) \cap E) \geq \lambda_E\left( E_{n+2,j} \cap E\right) \geq r/\ell^2,
	\]
	and (\ref{linear}) holds with $s=1$.
	
	\subsection{Minimality of vertical Cantor sets}\label{vertical:minimal}
	
	We would like to apply (the proof of) Theorem \ref{cd1} to show that every vertical Cantor set $E\in\mathcal{E}$ has conformal dimension $1$. Just like before the proof is based on the Mass Distribution Principle \ref{mdp} and for every $\alpha<1$ one needs to construct a measure $\nu$ on $f(E)$ such that for every $x \in E$ and $r>0$ we have 
	\begin{align}\label{ineq:nu-general}
		\nu(B(f(y),r))\lesssim r^{\alpha}.
	\end{align}
	The key result in this direction is the following.
	
	% But from the definition of $E$ above it is clear that the conditions of the theorem are satisfied. Indeed,  for every $E_{n,j}$ we have that the diameter of the siblings of $E_{n,j}$ are equal to $\diam E_{n,j}$.  Moreover,  the relative distance between ``adjacent'' generation $n$ rectangles tends to $0$ as it is comparable to $\ell^n/m^n$.
	%
	%Before proving that vertical Cantor sets in $K$ are minimal we prove the following analogue of Lemma \ref{mla}.
	
	\begin{Lemma}\label{lemma:vertical-cantor}
		Let $E \in {E}$ be a vertical Cantor set. Then for any $\eta$-quasisymmetry $f: E \to I$ and any $0 < \alpha < 1$, there exists a probability measure $\nu$ on $I$ such that for any $n\geq 1$ and $j\in\{1,\ldots, \ell^n\}$ we have
		\begin{equation}\label{ineq:measure-bound}
			\nu(f(E_{n,j}\cap E)) \leq C \diam(f(E_{n,j}\cap E))^{\alpha}
		\end{equation}
		for some constant $0 < C < \infty$ depending on $K$, $\eta$, and $\alpha$.
	\end{Lemma}
	
	\begin{proof}
		%	To show that every  $E\in\mathcal{E}$ has conformal dimension at least one we proceed as in the proof of Theorem 
		%	
		%	Note that every  $E\in\mathcal{E}$ can be written as 
		%	$$E=\bigcap_{i=1}^{\infty}\bigcup_{j=1}^{\ell^{i}}E_{i,j},$$
		%	where $E_{i,j}$ is a closed $m^{-i}\times \ell^{-i}$ rectangle such that $\pi_y(E_{i,j})=\left[\frac{j-1}{\ell^{i}},\frac{j}{\ell^{i}}\right]$. 
		%	
		Slightly abusing the notation, below we will write $E_{n,j}$ instead of $E_{n,j}\cap E$. Using this convention, from the construction of $E$   we have
		\begin{align}\label{selfsimilardist}
			\dist(E_{n,j},E_{n,j+1}) \lesssim 1/m^n,
			%	\item $\dist(E_{n,j},E_{n,j'})\gtrsim \diam(E_{n,j})$ if $|j-j'|>1$.
		\end{align}
		if $E_{n,j}$ and $E_{n,j+1}$ have the same parent.
		
		Letting $I=f(E)$, $I_{n,j}=f(E_{n,j}\cap E)$, and $\tilde{I}_{n,j}=f(\tilde{E}_{n,j})$. We let $\nu(I)=1$ and by induction define 
		\begin{align}\label{equation:nu}
			{\nu(I_{n,j})} =\frac{\diam \left(I_{n,j}\right)^{\alpha}}{\sum_{\tiny{I_{n,j'}\subset \tilde{I}_{n,j}}} \diam\left( I_{n,j'} \right) ^{\alpha}} \nu(\tilde{I}_{n,j}),
		\end{align}
		where the sum in the denominator is over all the children of $\tilde{I}_{n,j}$. For every $n$ and $j\in\{1,\ldots,\ell^n\}$ there is a unique nested sequence of rectangles $\left\{ E_{i,j_i} \right\}$ such that
		\[
		E_{n,j}\subset E_{n-1,j_{n-1}} \subset \ldots E_{1,j_1} \subset E_{0,0} = [0,1]^2.
		\]
		
		We denote $I_i=f(E_{i,j_i}\cap E)$ for $i=0,\ldots,n$ and let 
		$$\mathscr{C}(I_{i-1})=\{J_{i,1},\ldots,J_{i,\ell}\}$$ be the collection of all children of $I_{i-1}$ for $i\geq 1$. We order the children of $I_{i-1}$ in the ascending order according to the $\pi_y$ projection of the corresponding rectangles $E_{i,j}$. Using equation (\ref{equation:nu}), by induction we have
		\begin{align}\label{measureproduct}
			\frac{\nu(I_{n})}{\diam\left(I_{n}\right)^{\alpha}} = \prod_{i=1}^{n}\frac{\diam\left(I_{i-1} \right)^{\alpha}}{\sum_{j'=1}^{\ell} \diam\left(J_{i,j'}\right)^{\alpha}}.
		\end{align}
		
		It is sufficient to show that the right hand side of equation \eqref{measureproduct} is uniformly bounded. We next prove that for any fixed $\alpha<1$, there exists $N=N(\alpha,\eta)$ and $\delta<1$ such that for any $n>N$ we have
		\begin{align}\label{ineq:vert2}
			\frac{\diam\left(I_{i-1} \right)^{\alpha}}{\sum_{j'=1}^{\ell} \diam\left(J_{i,j'}\right)^{\alpha}} <\delta<1.
		\end{align}
		
		Note that by Lemma \ref{relativedistance} and inequality \eqref{selfsimilardist}, for every $\eps>0$ we  can find $N_1=N_1(\eps,\eta)\in\mathbb{N}$ such that for every $j'\in\{1,\ldots \ell\}$ we have
		\[
		\dist(J_{i,j'},J_{i,j'+1})<\eps\min\left\{ \diam\left(J_{i,j'}\right),\diam\left(J_{i,j'+1}\right) \right\}.
		\]
		Therefore,
		\begin{align*}
			\diam\left(I_{i-1}\right) &\leq \sum_{j'=1}^{\ell-1} \left[ \diam\left(J_{i,j'}\right)+\dist(J_{i,j'},J_{i,j'+1})  \right]+\diam\left(J_{i,\ell}\right)
			\\
			& \leq (1+\eps) \sum_{j'=1}^{\ell} \diam\left(J_{i,j'}\right)
		\end{align*}
		and 
		\begin{align}\label{ineq:Jensen-type}
			\frac{\diam\left(I_{i-1}\right)^{\alpha}}{\sum_{j'=1}^{\ell} \diam\left(J_{i,j'}\right)^{\alpha}}\leq (1+\eps)^{\alpha}	\frac{\left(\sum \diam\left(J_{i,j'}\right)\right)^\alpha}{\sum \diam\left(J_{i,j'}\right)^{\alpha}}.
		\end{align} 
		
		To estimate the right hand side from above, assume, without loss of generality, that $\diam(J_{i,\ell})\geq \diam(J_{i,j})$ for all $j\in\{1\ldots,\ell-1\}$. Moreover, from the construction and (\ref{ineq:rel-dist-distortion}) it follows that there is a constant $c=c(\ell,\eta)\in(0,1)$, such that $\diam(J_{i,j})\geq c\diam (J_{i,\ell})$. 
		
		Letting $x_j=\frac{\diam (J_{i,j})}{\diam(J_{i,\ell})}$, 
		we have $c\leq x_j \leq 1$ for all $1\leq j\leq \ell$.  Therefore, since $\alpha<1$, the quotient in the right hand side of (\ref{ineq:Jensen-type}) is bounded from above by 
		\begin{align*} \frac{(x_1+\ldots+x_{\ell-1}+1)^{\alpha}}{x_1^{\alpha}+\ldots+x_{\ell-1}^{\alpha}+1}
			&\leq \frac{(x_1+\ldots+x_{\ell-1}+1)^{\alpha}}{x_1+\ldots+x_{\ell-1}+1}\\
			&\leq\frac{1}{(1+c)^{1-\alpha}}<1.
		\end{align*}
		
		%	
		%	\[
		%	M^{-1} \leq {\diam (J')}/{\diam(J)} \leq M
		%	\]
		%	whenever $J,J'\in \mathscr{C}(I_{i-1})$. 
		%	
		%	Finally, considering the function
		%\begin{align*}
		%	g(x_1,\ldots,x_{\ell})=\frac{(x_1+\ldots+x_{\ell})^{\alpha}}{x_1^\alpha+\ldots+x_{\ell}^{\alpha}}.
		%\end{align*}	
		%
		%It is not hard to see that there is a $\delta'=\delta'(\alpha,\ell,M)<1$ such that if $x_i/x_j\in[1/M,M]$ then $g(x_1,\ldots,x_{\ell})<\delta'$. 
		
		Choosing $\eps_0=(\delta(1+c)^{1-\alpha})^{1/\alpha}-1$ and combining the above estimates we conclude that for $n>N_1(\eps_0,\eta)$ we have (\ref{ineq:vert2}) which completes the proof.
	\end{proof}
	
	To complete the proof of (\ref{ineq:nu-general}) we first note  that $E$ is flat. 
	%
	%Indeed, let $x \in E$ and $A(x)$ be the flatness constant of $x$ in $E$. 
	%
	Indeed, for disjoint $E_{n,j}$ and $E_{m,l}$, we have
	\[
	\diam\left(E_{n,j} \cup E_{m,l}\right) \geq \diam\left(\pi_y\left(E_{n,j}\right)\right) + \diam\left(\pi_y\left(E_{m,l}\right)\right).
	\]
	Then the construction of $E$ and the disjointedness of $E_{n,j}$ imply that for $x\in E$ the flatness constant, $A(x)$, is no greater than  two.
	
	Finally, using the flatness,  arguing  as in the proof of Theorem \ref{cd1}, and observing that $\dim_H(E) =1$ we obtain that (\ref{ineq:nu-general}) holds for every $B(x,r)\subset E$. Therefore we have the following. 
	
	\begin{Lemma}\label{lemma:vertical-cantor1}
		$\dim_C(E)= 1$ for every  $E\in\mathcal{E}$.
	\end{Lemma}
	
	\subsection{Conclusion of the proof} With $\mu$ and $\mathbf{E}$ constructed as above, we would like to show that $\Mod_1(\mathbf{E})>0$. For that let $\rho:K\to[0,\infty)$ be admissible for $\mathbf{E}$, that is $\int_E \rho d \lambda_E \geq 1$ for every $E\in\mathcal{E}$. Since $\mu$ is a Radon measure as explained above, by the Vitali-Carath\'eodory theorem \ref{VCT} we can assume that $\rho$ is in fact lower-semicontinuous.
	We will need the following result.
	\begin{Lemma}\label{lemma:construction}
		For every lower-semicontinuous function $\rho:K\to[0,\infty)$ there is a Borel function $\rho_\infty:K\to[0,\infty)$ such that
		\begin{enumerate}
			\item $\int \rho d\mu \geq \int \rho_\infty d \mu$,\label{smeasure}
			
			\item  $\rho_\infty(x_1,y_1)=\rho_\infty(x_2,y_2)$ if $y_1=y_2$,\label{emeasure}
			
			\item $\exists F\in\mathcal{E}$ s.t. $\rho(x)\leq \rho_\infty(x)$ for $x\in F$.\label{admissiblemeasure}
		\end{enumerate}
	\end{Lemma}
	Before proving Lemma \ref{lemma:construction} we observe that it implies the inequality $\Mod_{1}(\mathbf{E})>1$. This would complete the proof of Theorem \ref{thm:uniformfibers}, provided we verify all the other conditions of Theorem \ref{fmodest}. Note that in Subsections \ref{section:mu}, \ref{section:vertical}, and \ref{vertical:minimal} it was proved that $\mu$ satisfies the upper mass bound (\ref{um}) with $q=1+d$, the sets $E\in\mathcal{E}$ are minimal, and the measures $\lambda_E$ satisfy (\refeq{linear}) with $s=1$. Therefore, the only condition that is left to check is that $\mu$ is a doubling measure. To see this, we consider {approximate squares}, see also \cite{McMullen} or \cite{BP17}. 
	
	Let $\alpha=\frac{\log \ell}{\log m}<1$ and define the collection of \emph{approximate squares} of generation $n$ as follows
	\[
	\mathcal{Q}_n=\left\{\left[\frac{j}{\ell^n},\frac{j+1}{\ell^n}\right] \times \left[\frac{k}{m^{\lfloor \alpha n \rfloor}},\frac{k+1}{m^{\lfloor \alpha n \rfloor}}\right]: j,k \in\mathbb{Z}\right\}.
	\]
	
	Note, that each $Q\in\mathcal{Q}_n$ is a rectangle of height $\ell^{-k}$ and width $m^{-\lfloor\alpha k \rfloor}\asymp \ell^{-k}$. 
	
	Observe, that if $Q \in\mathcal{Q}_n$ then either $\mathrm{int}(Q)\cap K=\emptyset$ and $\mu(Q)=0$ or $\mathrm{int}(Q)\cap K\neq\emptyset$ in which case $\mu (Q)=c_n$, for some $c_n>0$ which is independent of $Q$. Therefore, if $\mathrm{int}(Q)\cap K\neq\emptyset$ then denoting by $\tilde{Q}$ the collection of all the  approximate squares $Q'\in\mathcal{Q}_n$ such that $Q'\cap Q \neq \emptyset$ (including $Q$ itself) we obtain $\mu(\tilde{Q}) \leq 9c_n= 9\mu(Q)$. It is then quite standard to check that $\mu$ is in fact a doubling measure.
	
	\begin{proof}[Proof of Theorem \ref{thm:uniformfibers}]
		
		Note that by \eqref{emeasure} there exists a unique $\bar{\rho}: [0,1] \to \mathbb{R}$ s.t.  for $\mathcal{L}^1$-a.e. $a \in [0,1]$ we have $\bar{\rho}(a) = \rho_\infty(z)$, whenever $\pi_y(z) = a$. Then, by \eqref{admissiblemeasure}, for any $E \in \mathcal{E}$ we have
		\[
		\int_E \rho_\infty(z) \ d \lambda_E = \int_0^1 \bar{\rho}(a) \ da =  \int_F \rho_\infty(x) \ d \lambda \geq \int_F \rho(x) \ d \lambda \geq 1.
		\]
		Thus $\rho_\infty$ is also admissible for $\mathcal{E}$. Finally, using \eqref{smeasure}, the Disintegration Theorem \cite[Theorem $5.3.1$]{AGS05},  and the inequalities above, we obtain
		\begin{align*}
			\int_K \rho d\mu \geq \int_K \rho_\infty d\mu
			=  \int_0^1 \left[ \int_{\pi_y^{-1}(a)}\rho_\infty(z) \ d \mu_a(z) \right] \ da
			&=  \int_0^1 \bar{\rho}(a) \mu_a(K \cap Z_a)\ da
			\\
			&  =\int_0^1 \bar{\rho}(a) \ da
			\geq 1,
		\end{align*}
		where the last equality holds since $\mu_a\left(K \cap Z_a\right)=1$ for almost every  $a \in [0,1]$.
	\end{proof}
	
	\begin{proof}[Proof of Lemma \ref{lemma:construction}]
		To construct the function $\rho_\infty$  we will first construct a sequence of piecewise constant functions $\rho_n$ from $\rho$ as follows.
		
		From each of the $\ell$ rows of $\mathcal{K}_1$ choose one rectangle with the smallest average of $\rho$.  Specifically,  choose rectangles $R_{1,1},\ldots,R_{1,\ell}\in\mathcal{K}_1$ so that if $R\in\mathcal{K}_1$ is such that $\pi_y(R)=\pi_y (R_{1,j})$ for some $j$, then
		\begin{align*}
			\intbar_{R_{1,j}} \rho d\mu \leq \intbar_{R} \rho d\mu.
		\end{align*}
		Define $\rho_1:K_1\to[0,\infty]$ by
		\begin{align*}
			\rho_1\big\rvert_{R} \equiv \intbar_{R_{1,j}}\rho d\mu,   \mbox{ if } \pi_y(R)=\pi_y (R_{1,j}).
		\end{align*}
		That is $\rho_1$ is the same constant on every first generation rectangle at the same height.  Hence,  $\rho_1$ is also constant on almost every fiber $\pi^{-1}_y(a)$, and $\int_K \rho d\mu \geq \int_K \rho_1 d\mu$.
		Continuing by induction, suppose that at step $n-1$ we have chosen rectangles $R_{n-1,1},\ldots,R_{n-1,\ell^{n-1}}$ and defined $\rho_{n-1}$ so that it is constant on all rectangles of generation $n-1$ in $\mathcal{K}_{n-1}$ and the constants are the same on two rectangles with the same heights.
		
		\begin{figure}[htbp]
			\centering
			\includegraphics[height= 1.2 in]{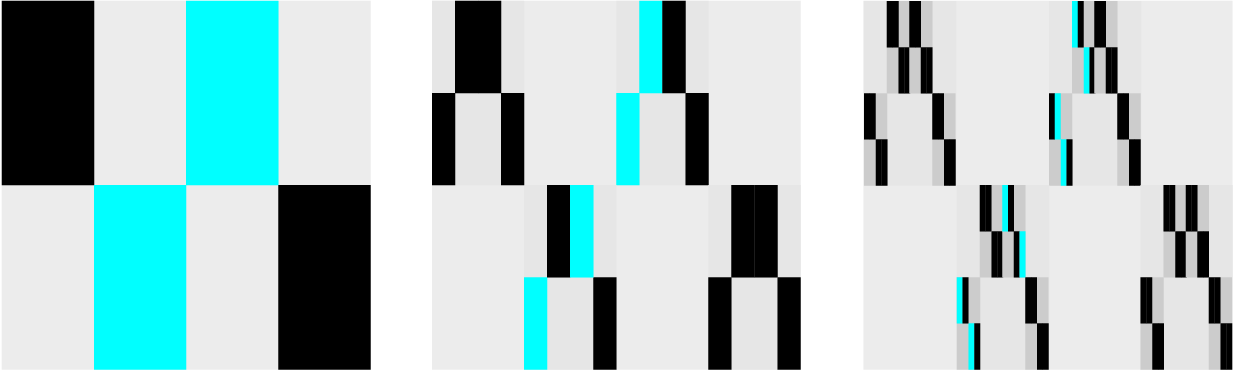}
			\caption{The first three generations of $F$ on $I$.}
			\label{ME}
		\end{figure}
		
		To construct $\rho_{n}$,   from each of the  $\ell$ rows of generation $n$ rectangles within $R_{n-1,j}$ for  $j\in\{1,\ldots,\ell^{n-1}\}$  chosen above, we will select one with the smallest average of $\rho$.  That is,  the rectangles $R_{n,j'}$ with $j'\in\{1,\ldots,\ell^{n}\}$ are such that
		\begin{align*}
			\intbar_{R_{n,j'}} \rho d\mu \leq \intbar_{R} \rho d\mu,
		\end{align*}
		whenever $R, R_{n,j'}\subset R_{n-1,j}$ for some $j\in\{1,\ldots,\ell^{n-1}\}$ and $\pi_y(R)=\pi_y (R_{n,j'})$.
		
		Define $\rho_n$ be setting
		\begin{align*}
			\rho_n\big\rvert_{R} \equiv \intbar_{R_{n,j}}\rho d\mu,   \mbox{ if } \pi_y(R)=\pi_y (R_{n,j}).
		\end{align*}
		
		Then we have a sequence of Borel measurable functions $\{\rho_n \}$.  It follows from the construction above  that we have
		\[
		\int_K \rho \ d\mu \geq \int_K \rho_1 \ d\mu \geq \ldots\geq
		\int_K \rho_n \ d\mu\geq \int_K \rho_{n+1} \ d\mu\geq \ldots.
		\]
		
		For every $z=(x,y)\in K$ we define
		\begin{align}\label{def:trho}
			\rho_\infty(z) := \liminf_{n\to\infty} \rho_n(z).
		\end{align}
		
		Condition \eqref{smeasure} then follows from the definition of $\rho_\infty$ and Fatou's Lemma.  Equality \eqref{emeasure} follows from (\ref{def:trho}) and the fact that for every $n$ the function $\rho_n$ is constant on almost all (in fact all, except for a finitely many) horizontal slices $K\cap Z_a$.
		
		To prove \eqref{admissiblemeasure} define $F=\bigcap_{n=1}^{\infty}\bigcup_{j=1}^{\ell^n} R_{n,j}$.
		
		For every $z\in F$ choose a rectangle $R_{n}(z)\in\mathcal{K}_n$ such that $z\in R_n(z)$ for all $n\geq 1$.  Pick $z_{n}'\in R_{n}(z)$ so that \begin{align}\label{trho-adm}
			\rho(z_{n}')\leq \intbar_{R_{n}(z)} \rho  d\mu = \rho_n(z_{n}')=\rho_n(z).
		\end{align}
		Note that $z_n'\to z$ as $n\to\infty$, since  $\diam R_n(z) \asymp \ell^{-n}$.
		Recalling that $\rho$ is lower semicontinuous, we  also have that $\rho(z) \leq \liminf _{z'\to z}\rho(z')$ for any $z \in K$. Therefore,  from (\ref{def:trho}) and (\ref{trho-adm}) it follows that for every $z \in F$ we have
		\begin{align*}
			\rho(z) & \leq \liminf_{n\to\infty} \rho(z_n') \leq \liminf_{n\to\infty} \rho_n (z)=\tilde\rho(z),
		\end{align*}
		This proves (3) since $F$ is a vertical Cantor set in $K$ by definition.
	\end{proof}
	
	\subsection{Minimality of continuous self-affine graphs}
	In this section we provide examples of continuous functions which are minimal for conformal dimension and prove Corollary \ref{thm:mg}. These examples fall into the class of spaces described in Remark \ref{remark:different-patterns} and the proof of minimality for these more general sets is exactly the same as that of Theorem \ref{thm:uniformfibers} above.  
	
	\begin{Example}\label{example-sa-graph}
		Following \cite[p. 156]{BP17} we construct a graph of a self-affine continuous function as follows.  Define matrices $A_j$, $j\in\{1,2,3\}$,  as follows:
	\end{Example}
	\begin{equation*}
		A_1 = 
		\left(
		\begin{array}{cccc} 
			0 & 2 & 3 & 0\\ 
			2 & 0 & 0 & 3
		\end{array}
		\right),
		\quad 
		A_2 = 
		\left(
		\begin{array}{cccc} 
			0 & 0 & 1 & 2\\ 
			1 & 2 & 0 & 0
		\end{array}
		\right),
		\quad 
		A_3 = 
		\left(
		\begin{array}{cccc} 
			3 & 1 & 0 & 0\\ 
			0 & 0 & 3 & 1
		\end{array}
		\right).
	\end{equation*}
	
	Start with the unit square $[0,1]^2$. Divide it into $8$ congruent $1/4 \times 1/2$ rectangles and keep the ones corresponding to the non-zero entries in the matrix $A_1$. Replace the rectangles corresponding to the label $j\in\{1,2,3\}$ by a pattern of  rectangles corresponding to $A_j$. It is easy to see, and is left to the reader, that the resulting set $\Gamma$ is a graph of a continuous function See Fig. \ref{sa-graph} for an illustration of $\Gamma$ and also see \cite[Exercise 5.40]{BP17}. Moreover, by Proposition \ref{prop:dim} and Theorem \ref{thm:uniformfibers}, which are applicable here as explained in Remark \ref{remark:different-patterns}, we have
	\[
	\dim_H \Gamma=\dim_C \Gamma = 1+\log_4 2 = 3/2.
	\]
	
	\begin{figure}[htbp]
		\centering
		\includegraphics[height = 2 in]{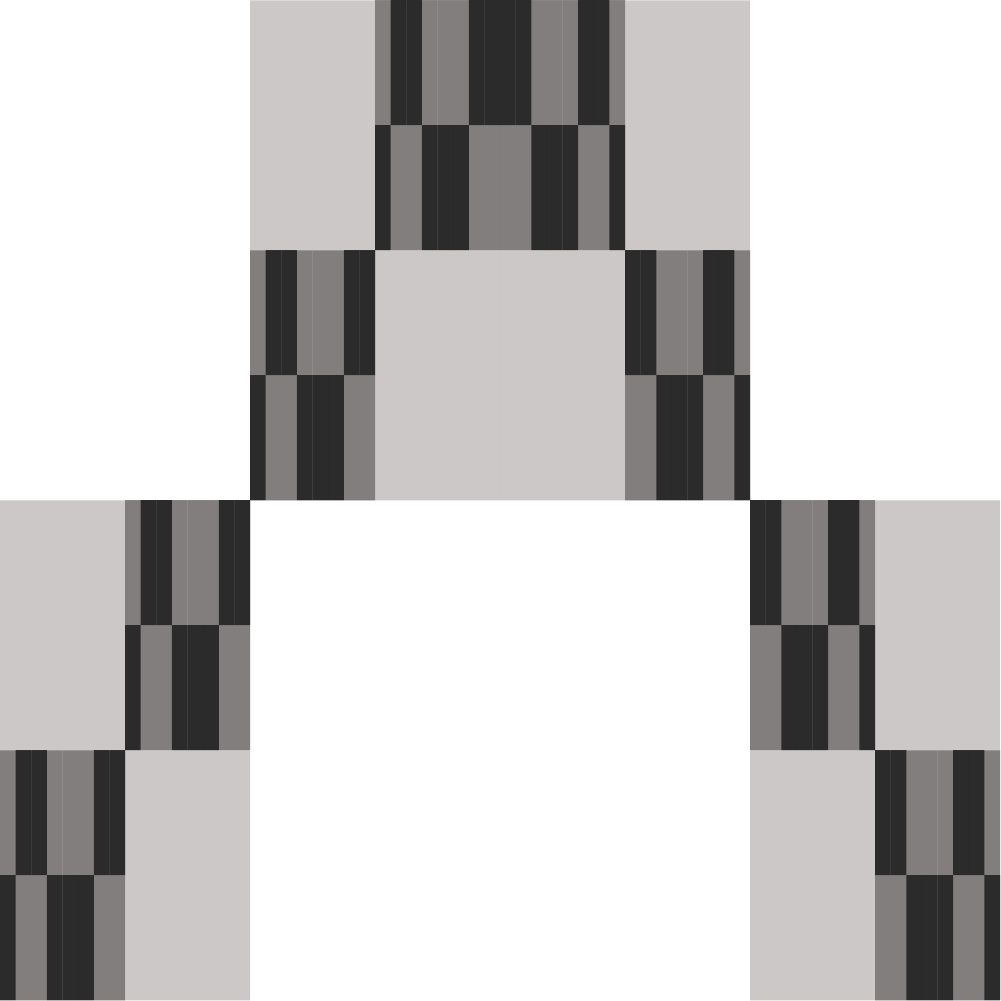}\quad
		\includegraphics[height = 2 in]{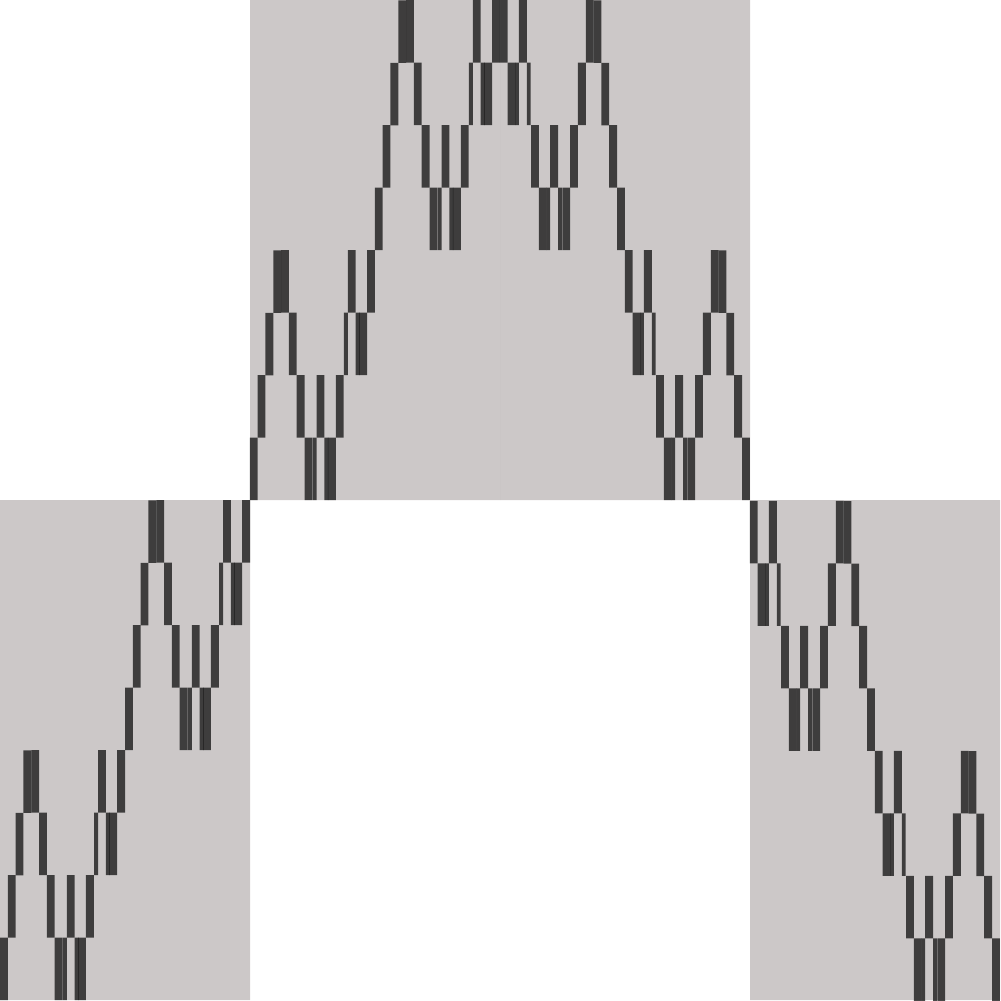}
		\caption{Approximations to the minimal self-affine graph $\Gamma$ in Example \ref{example-sa-graph}. In this example one has  $\dim_H\Gamma=\dim_C \Gamma =3/2$.}
		\label{sa-graph}
	\end{figure}
	
	It is not hard to generalize the previous example to obtain graphs of continuous functions which are minimal and have various conformal dimensions. Before giving the general procedure we provide one more motivating example that is helpful for understanding the proof of Corollary \ref{thm:mg}. Define
	\begin{align*}
		A_1 &= 
		\left(
		\begin{array}{cccccccccccc} 
			0 & 0 & 2 & 3 & 1 & 1 & 0 & 0 & 0 & 0 & 0 & 0 \\
			0 & 2 & 0 & 0 & 0 & 0 & 3 & 1 & 1 & 0 & 0 & 0  \\
			2 & 0 & 0 & 0 & 0 & 0 & 0 & 0 & 0 & 3 & 1 & 1  
		\end{array}
		\right),
		\\
		A_2 &= 
		\left(
		\begin{array}{cccccccccccccccc} 
			0 & 0 & 0 & 0 & 0 & 0 & 0 & 0 & 1 & 1 & 1 & 2  \\
			0 & 0 & 0 & 0 & 1 & 1 & 1 & 2 & 0 & 0 & 0 & 0  \\
			1 & 1 & 1 & 2 & 0 & 0 & 0 & 0 & 0 & 0 & 0 & 0 
		\end{array}
		\right),
		\\
		A_3 &= 
		\left(
		\begin{array}{cccccccccccccccc} 
			3 & 1 & 1 & 1 & 0 & 0 & 0 & 0 & 0 & 0 & 0 & 0  \\ 
			0 & 0 & 0 & 0 & 3 & 1 & 1 & 1 & 0 & 0 & 0 & 0  \\
			0 & 0 & 0 & 0 & 0 & 0 & 0 & 0 & 3 & 1 & 1 & 1  
		\end{array}
		\right).
	\end{align*}	
	
	Again, the set constructed using these matrices is easily seen to be a graph of continuous function. Moreover, this graph is minimal and its Hausdorff dimension is equal to $1+\log_{12} 4$, where $4$ is the number of nonzero entries in each row and $12$ is the total number of entries in every row of these matrices.
	
	\begin{proof}[Proof of Corollary \ref{thm:mg}]
		Let $k$ and $\ell$ be positive integers and $m=k\ell$. We define $\ell \times m$ matrices $A_1,A_2,A_3$ as follows
		
		\begin{align*}
			A_1 &= 
			\left[
			\begin{array}{cccccccccccc} 
				0 & 0 &  \cdots & 2 & D_{k-1} & 0 & \ldots & 0  \\
				\vdots & \vdots & \adots& \vdots & D_{k-1} & 0 & \ldots & 0 \\
				0 & 2 & \cdots & 0 & \vdots & \vdots & \ddots & \vdots  \\
				2 & 0 & \cdots  & 0 & 0 & 0 & \cdots & D_{k-1}  
			\end{array}
			\right],
		\end{align*}
		\begin{align*}
			A_2 &= 
			\left[
			\begin{array}{cccc} 
				0 & 0 & \cdots & C_k  \\
				\vdots & \vdots & \adots & \vdots\\
				0 & C_k & \cdots & 0   \\
				C_k & 0 & \cdots & 0   
			\end{array}
			\right],
			\quad
			A_3 = 
			\left[
			\begin{array}{cccc} 
				D_k & 0 & \cdots & 0  \\
				0 & D_k & \cdots & 0   \\
				\vdots & \vdots & \ddots & \vdots\\
				0 & 0 & \ldots & D_k   
			\end{array}
			\right],
		\end{align*}	
		where $C_k$ and $D_k$ are the blocks of length $k$ of the form 	
		\begin{align*}
			C_k&=
			\left[
			\begin{array}{ccccc} 
				1 & 1 & \cdots & 1 & 2 
			\end{array}
			\right],
			\\
			D_k&=
			\left[
			\begin{array}{ccccc}
				3 & 1 & \cdots & 1 & 1\
			\end{array}
			\right],
		\end{align*}
		respectively, while $0$ represents a block of $0$'s of appropriate length (which may vary from one occurrence to another). Using this choice of matrices $A_j$ we can define the corresponding subset of the square $[0,1]^2$, which is easily seen to be  the graph of a continuous function. Moreover, by Remark \ref{remark:different-patterns}, this set, which we will denote by $K$, will be minimal for conformal dimension, and 
		\begin{align}
			\dim_H K = 1+\log_{m}k = 1+\log_{k\ell}k.
		\end{align}
		
		Letting $k=2^r$ and $\ell=2^s$ we obtain $\dim_H K=1+\frac{r}{r+s}$, where $1\leq r <s$ are arbitrary integers. Hence, the values of $\dim_H K$ are dense  in $[1,2]$.  Therefore, for any $d\in[1,2]$, there is a sequence of minimal sets $K_i'$ constructed as above so that $\dim_H K_i'=d_i$,  where $d_i\nearrow d$ as $i \to \infty$.   Moreover,   each $K'_i$ is a graph of a \emph{continuous} function $f_i$ such that $f_i(0)=f_i(1)=0$. Then,  denoting $aK+b=\{az+b : z \in K\}\subset\mathbb{C}$, for $a,b\in\mathbb{R}$, where $az=a(x+iy) = ax+iay\in\mathbb{C}$, the set
		\[
		\bigcup_{i=1}^{\infty}\left(\frac{1}{2^{i}}K_i'+1-\frac{1}{2^{i-1}}\right)
		\]
		will be a graph of a continuous function, and the conformal dimension of that set will be $\sup_{i}\dim_C K_i'=\lim_i d_i=d$.
	\end{proof}
	
	%The key features of the matrices $A_j$ that make the set constructed in the example above a continuous graph are (1) the fact that every column of $A_j$ contains a single nonzero entry, and (2) ``compatibility" of the patterns. 

	\section{Brownian motion and local time}\label{BM}
	
	In this section, we recall some standard properties of the linear Brownian motion and the associate local time. Moreover, we establish a few auxiliary results that are used in the proof of our main Theorem \ref{minimalgraph}.
	
	We  denote the $1$-dimensional or linear Brownian motion defined on a probability space $\left( \Omega, \mathcal{F}, \mathbb{P}\right)$ by $W\!(t, \omega)$, where $\omega\in\Omega$, or  by $W\!(t)$ if the probability space is clear from the context. We call $W\!(t, \omega)$ a \emph{standard Brownian motion} if $W\!(0,  \omega) = 0$ for any $\omega \in \Omega$. See \cite[Chapters $1$ and $2$]{MP10} for the precise definitions and basic properties of the Brownian motion.
	
	Let $\{\mathcal{F}(t): t\geq 0\}$ be a family  of $\sigma$-algebras such that $\mathcal{F}(t)$ is generated by the random variables $W\!(s)$, for $0 \leq s \leq t$. A random variable $T$ with values in $[0, \infty]$ is called a \emph{stopping time} with respect to $\{\mathcal{F}(t): t\geq 0\}$ if $\{T \leq t\} \in \mathcal{F}(t)$ for every $t \geq 0$. We define  $\mathcal{F}^+(s) = \bigcap_{t>s} \mathcal{F}(t)$. If $T$ is a stopping time, define
	\[
	\mathcal{F}^+(T) = \left\{ A \in \mathcal{F}: A \cap \{T \leq t\} \in \mathcal{F}^+(t) \ \textnormal{for all} \ t \geq 0\right\}.
	\]
	
	One of the most important properties of $W\!(t)$ is the strong Markov property. See \cite[Theorem $2.16$]{MP10} or \cite[Theorem $2.20$ and $6.17$]{LeG16} for a general version.
	
	\begin{Theorem}[Strong Markov Property]\label{Strong Markov}
		For every almost surely finite stopping time $T$,
		\[
		\{W\!(T + t) - W\!(T): t \geq 0\}
		\]
		is a standard Brownian motion independent of $\mathcal{F}^+(T)$.
	\end{Theorem}
	
	We denote by $T_a$ the \emph{first hitting time} of $W\!(t)$ at level $a$. Namely,
	\[
	T_a = \inf\{ t \geq 0: W\!(t) = a \}.
	\]
	It is easy to see that $T_a$ is a stopping time by \cite[Remark~$2.14$]{MP10}. We can also induce the distribution of $T_a$ from \cite[Theorem~$2.21$]{MP10} and then $T_a< +\infty$ a.s..
	
	In what follows, we denote by $\Gamma(W)$ the graph of $W\!(t)$, i.e.,
	\[
	\Gamma(W) = \left\{(t,W\!(t)): t \geq 0\right\}.
	\]
	More generally, we let
	\[
	\Gamma(W\!(I)) = \left\{(t,W\!(t)): t \in I\right\},
	\]
	for any $I \subseteq [0, +\infty)$.
	
	The following is a well-known fact about Brownian motion, see \cite[Theorem $4.29$ and Corollary $9.30$]{MP10} for more details.
	
	\begin{Proposition}
		Almost surely $\dim_H(\Gamma(W)) = \frac{3}{2}$ and $\dim_H \left( \Gamma(W) \cap Z_a \right) = \frac{1}{2}$ for any $a \in \mathbb{R}$.
	\end{Proposition}
	
	In order to measure $\Gamma(W) \cap Z_a$, we need to introduce the local time of $W\!(t)$. The classical local time $L^0(t)$ of a linear Brownian motion $W$ is a nondecreasing stochastic process that measures the amount of time $W$ spends at $0$. In general, we can define $L^a(t)$, the local time of $W$ at $a$, to be a quantity that measures the amount of time $W$ spends at $a$. The local time $L^0(t)$ induces a Hausdorff measure on the real line, see \cite[Theorem $6.43$]{MP10}.
	
	Next we rigorously define the concept of local time. Let $W\!(t)$ be a linear Brownian motion with arbitrary starting point. Given $a < b$, we define the stopping times $\tau_0 = 0$ and, for $j \geq 1$,
	\[
	\sigma_j = \inf\{t>\tau_{j-1}: W\!(t) = b\}, \ \tau_j = \inf\{t>\sigma_{j}: W\!(t) = a\}.
	\]
	We call $W\!([\sigma_i, \tau_i])$ a \emph{downcrossing} of the interval $[a,b]$. For every $t > 0$ we denote by
	\[
	D(a,b,t)=  \max\{j\in \mathbb{N}: \tau_j \leq t\}.
	\]
	This is the number of downcrossings of the interval $[a, b]$ before time $t$.
	
	Let $a_n < 0 < b_n$ such that $a_n \uparrow 0$ and $b_n \downarrow 0$. The \emph{local time} of $W$ at zero is a stochastic process defined in the following way:
	\[
	L^0(t) = \lim_{n \to \infty}2(b_n-a_n)D(a_n, b_n,t).
	\]
	The definition of $L^0(t)$ is independent of the choice of $\{a_n \}$ and $\{ b_n\}$.
	
	In general, we define $L^a(t)$ to be the local time of $W\!(t)$ at level $a$, i.e.,
	\begin{eqnarray}\label{localtime}
		L^a(t) = \lim_{n \to \infty} 2^{-n+1}D_n(a,t)
	\end{eqnarray}
	where $D_n(a,t)$ is the number of downcrossings before time $t$ of the unique $n^{\mathrm{th}}$-generation dyadic interval $(i2^{-n} , (i+ 1)2^{-n} ]$ containing $a$. The reader may refer to \cite[Chapter $6$]{MP10} for more details. The definition of $L^a(t)$ in \cite[Chapter $6$]{MP10} slightly  differs from the one given above, but the two are equivalent.
	
	The following theorem shows that local time exists and is almost surely locally $\gamma$-H\"older continuous with respect to $a$ and $t$ for any $\gamma < \frac{1}{2}$. See \cite[Theorem $6.19$]{MP10}.
	
	\begin{Theorem}[Trotter]\label{lf}
		Let $W\!(t)$ be a linear Brownian motion and let $D_n(a,t)$ be the number of downcrossings before time $t$ of the $n^{\mathrm{th}}$-generation dyadic interval $(i2^{-n} , (i+ 1)2^{-n} ]$  containing $a$. Then, almost surely,
		\[
		L^a(t) = \lim_{n \to \infty} 2^{-n+1}D_n(a,t)
		\]
		exists for all $a \in \mathbb{R}$ and $t \geq 0$. Moreover, for every $\gamma < \frac{1}{2}$, the random field $\{L^a(t): a \in\mathbb{R}, t \geq 0\}$ is almost surely locally $\gamma$-H\"older continuous.
	\end{Theorem}
	
	The following theorem studies the local time before a hitting time . See \cite[Theorem $6.38$]{MP10}.
	
	\begin{Theorem}[Ray]\label{ltv}
		Suppose $b > 0$ and $\{W\!(t){:} 0\, {\leq} \, t {\leq}\, T_b\}$ is a standard Brownian motion stopped at time $T_b$. Then, almost surely, $L^a(T_b) > 0$ for all $0 \leq a < b$.
	\end{Theorem}
	
	In \cite[Lemma $1$]{Law99}, a standard technique is presented for estimating an upper bound of the Hausdorff dimension of random subsets of the interval $[0, 1]$. The main strategy involves dividing the unit interval into a finite number of smaller subintervals and subsequently analyzing the statistics of how many of these subintervals intersect with the given random subset.
	
	Below, we let $[a]$ denote the integer part of $a$. Let ${A_n}$ be a collection of random subsets of the interval $[0, 1]$, and define $A = \bigcap_{i=1}^\infty \bigcup_{n=i}^\infty A_n$, i.e., $A = \limsup_{n \to \infty} A_n$. Lastly, let $s$ be a nonnegative real number. Below we will need the following result, see \cite[Lemma $1$]{Law99}.
	
	\begin{Lemma}\label{sdim}
		Let $p > 1$. If there exists a $C > 0$ such that
		\[
		\sum_{i=1}^{[p^n]+1}\mathbb{P}\left(\left[\frac{i-1}{p^{n}}, \frac{i}{p^n}\right] \cap A_n \neq \emptyset\right) \leq Cp^{sn},
		\]
		then almost surely $\dim\left(A \right) \leq s$.
		
		Moreover, for any $\eps > 0$, almost surely,
		\[
		\Card \left( \left\{ \left[\frac{i-1}{p^{n}}, \frac{i}{p^n}\right] : A_n \cap \left[\frac{i-1}{p^{n}}, \frac{i}{p^n}\right] \neq \emptyset  \right\} \right) \leq p^{(s+\eps)n}
		\]
		for sufficiently large $n$.
	\end{Lemma}
	
	Let $W\!(t)$ be a standard Brownian motion. We define
	\[
	M(t) = \max_{0 \leq s \leq t}W\!(s),
	\]
	and observe that for any $a >0$,
	\[
	\mathbb{P}(M(t) > a) = 2\mathbb{P}(W\!(t) > a) = \mathbb{P}(|W\!(t)| > a)
	\]
	by \cite[Theorem $2.21$]{MP10}.
	
	Recall that $T_a = \inf\{t > 0: W\!(t) = a \}$. Then for any $a > 0$,
	\begin{equation}\label{hitdist}
		\mathbb{P}\left( T_a > t \right) = \mathbb{P}\left( M(t) < a \right) = 2\mathbb{P}(W\!(t) < a) - 1.
	\end{equation}
	
	Let $a > 0$. We define
	\[
	T_{\scriptscriptstyle \pm a}(t_0) = \inf\{t - t_0: |W\!(t) - W\!(t_0)| = a, t > t_0 \},
	\]
	i.e., the difference between $t_0$ and first hitting time of the Brownian motion starting at time $t_0$ that hits either $W\!(t_0)+ a$  or $W\!(t_0)-a$. This clearly implies that
	\[
	\mathbb{P}\left( T_{\scriptscriptstyle \pm a}(t_0) > t \right) \leq 4\mathbb{P}(W\!(t) < a) - 2.
	\]
	by equation \eqref{hitdist} and the Strong Markov Property \ref{Strong Markov}.
	\begin{Lemma}\label{hitdim}
		For any $\frac{2}{3} < \alpha \leq  2$, almost surely,
		\begin{align}\label{dimest:slowpoints}
		\dim\left\{t \in [0,1]: \limsup_{a \to 0^+} \frac{T_{\scriptscriptstyle \pm a}(t)}{a^\alpha} > 1 \right\} \leq \frac{3}{2} - \frac{1}{\alpha}.
		\end{align}
	\end{Lemma}
	
	\begin{proof}
		Let
		\[
		A_n = \left\{t \in [0, 1]: T_{\scriptscriptstyle \pm \frac{1}{2^n}}(t) > \frac{1}{2^{\alpha(n+1)}} \right\}.
		\]
		
		Let $t \in [0,1]$ be such that $\limsup_{a \to 0^+} \frac{T_{\scriptscriptstyle \pm a}(t)}{a^\alpha} > 1$. Then there exists a sequence $a_n > 0$ such that $a_n \to 0$ and $T_{\scriptscriptstyle \pm a_n}(t) > (a_n)^\alpha$. Notice that for any $a_n$, there exists $N$ such that $\frac{1}{2^{N+1}} \leq a_n \leq \frac{1}{2^N}$. Thus
		\[
		T_{\scriptscriptstyle \pm \frac{1}{2^N}}(t) > \frac{1}{2^{\alpha(N+1)}}.
		\]
		Then $t \in \bigcap_{i=1}^\infty \bigcup_{n=i}^\infty A_n$. This implies that
		\[
		\left\{t \in [0,1]: \limsup_{a \to 0^+} \frac{T_{\scriptscriptstyle \pm a}(t)}{a^\alpha} > 1\right\} \subseteq \bigcap_{i=1}^\infty \bigcup_{n=i}^\infty A_n = \limsup_{n \to \infty} A_n.
		\]
		
		Assume that $A_n \cap \left[0, \frac{1}{2^{\alpha (n+1)+1}}\right] \neq \emptyset$, then there exist a point $t_0 \in \left[0, \frac{1}{2^{\alpha (n+1)+1}}\right] $ such that $T_{\scriptscriptstyle \pm \frac{1}{2^n}}(t_0) > \frac{1}{2^{\alpha (n+1)}}$. This implies that
		\[
		\max \left\{\left| W\!(s)- W\!(t_0) \right|: \frac{1}{2^{\alpha (n+1)+1}} \leq s \leq \frac{1}{2^{\alpha (n+1)}} \right\}< \frac{1}{2^{n}}.
		\]
		Thus
		\[
		T_{\scriptscriptstyle \pm \frac{1}{2^{n-1}}} \left(\frac{1}{2^{\alpha (n+1)+1}}\right) >  \frac{1}{2^{\alpha (n+1)+1}} .
		\]
		
		Since
		\begin{align*}
			\mathbb{P}\left(A_n \cap \left[0, \frac{1}{2^{\alpha (n+1)+1}} \right] \neq \emptyset \right) & \leq \mathbb{P}\left( T_{\scriptscriptstyle \pm \frac{1}{2^{n-1}}} \left(\frac{1}{2^{\alpha (n+1)+1}}\right) >  \frac{1}{2^{\alpha (n+1)+1}} \right)
			\\
			& \leq 4 \mathbb{P}\left( W\!\!\left(\frac{1}{2^{\alpha (n+1)+1}}\right) <  \frac{1}{2^{ n-1}} \right) -2,
		\end{align*}
		we have
		\begin{align*}
			&\mathbb{P}\left(A_n \cap \left[0, \frac{1}{2^{\alpha (n+1)}} \right] \neq \emptyset \right)
			\\
			&= \mathbb{P}\left(A_n \cap \left[0, \frac{1}{2^{\alpha (n+1)+1}} \right] \neq \emptyset \right) + \mathbb{P}\left(A_n \cap \left[\frac{1}{2^{\alpha (n+1)+1}}, \frac{1}{2^{\alpha (n+1)}} \right] \neq \emptyset \right)
			\\
			& \leq 8 \mathbb{P}\left( W\!\!\left(\frac{1}{2^{\alpha (n+1)+1}}\right) <  \frac{1}{2^{ n-1}} \right) - 4 \leq C\frac{1}{2^{\left(1-\frac{\alpha}{2}\right)(n+1)}}
		\end{align*}
		for some $C > 0$. The first inequality follows from  the previous inequality and the Strong Markov Property \ref{Strong Markov}, and the second inequality {comes from the density function of the Brownian motion}.
		
		Moreover, for any $i \in \mathbb{N}$,
		\[
		\mathbb{P}\left(A_n \cap \left[\frac{i-1}{2^{\alpha (n+1)}}, \frac{i}{2^{\alpha (n+1)}}\right] \neq \emptyset\right) \leq C\frac{1}{2^{\left(1-\frac{\alpha}{2}\right)(n+1)}}
		\]
		by the Strong Markov Property \ref{Strong Markov}.
		Therefore,
		\[
		\sum_{i=1}^{[2^{\alpha (n+1)}]+1}\mathbb{P}\left(A_n \cap \left[\frac{i-1}{2^{\alpha (n+1)}}, \frac{i}{2^{\alpha (n+1)}}\right] \neq \emptyset\right) \leq (C+1) 2^{\left(\frac{3}{2}\alpha - 1 \right) (n+1)}.
		\]
		By Lemma \ref{sdim} we have that almost surely
		\[
		\dim \left( \limsup_{n \to \infty} A_n\right) \leq \frac{3}{2} - \frac{1}{\alpha},
		\]
		which, in particular, implies that (\ref{dimest:slowpoints}) holds almost surely.
	\end{proof}
	
	\begin{Remark}\label{sufficientcover}
		Since
		\[
		\sum_{i=1}^{[2^{\alpha (n+1)}]+1}\mathbb{P}\left(A_n \cap \left[\frac{i-1}{2^{\alpha (n+1)}}, \frac{i}{2^{\alpha (n+1)}}\right] \neq \emptyset\right) \leq (C+1) 2^{\left(\frac{3}{2}\alpha - 1 \right) (n+1)},
		\]
		it follows from Lemma \ref{sdim} that for any $\eps > 0$, a.s.
		\[
		\Card\left( \left[\frac{i}{2^{\alpha(n+1)}}, \frac{i+1}{2^{\alpha(n+1)}}\right] \subseteq [0,1]: A_n \cap \left[\frac{i}{2^{\alpha(n+1)}}, \frac{i+1}{2^{\alpha(n+1)}}\right] \neq \emptyset \right) \leq 2^{\left(\frac{3}{2}\alpha - 1 + \eps\right)(n+1)}
		\]
		for sufficiently large $n$.
	\end{Remark}
	
	\section{The Brownian graph is minimal almost surely}\label{BGMAS}
	
	This section is devoted to the proof of Theorem \ref{bg}.
	
	\subsection{Constructing the measure $\mu$} Our first step is to construct a measure on $\Gamma(W)$ as described in equation \eqref{spacemeasure}.
	
	For any $a \in \mathbb{R}$, we first define a Borel measure $l^a$ on $\Gamma(W) \cap Z_a$  by
	\[
	l^a\left(  \Gamma(W\!([0,t])) \cap Z_a \right) = L^a(t).
	\]
	
	Then the Borel measure $\mu$ on $\Gamma(W)$ is defined as
	\begin{equation}
		\mu\left( A \right) =  \int_{-\infty}^{\infty} l^a(A \cap Z_a) \ da
	\end{equation}
	where $A$ is a Borel subset of $\Gamma(W)$. 
	
	The measure $\mu$ is a well-defined Borel measure since it is obtained by integrating a one-parameter family of Borel measures with respect to the Lebesgue measure on the line. Since $\Gamma(W)$ is second countable, every open set in $\Gamma(W)$ is $\sigma$-compact. Therefore, $\mu$ is a Radon measure by \cite[Theorem~$2.18$]{Rud86}.
	
	We proceed to show that the measure $\mu$ is doubling a.s.. In fact, $l^a$ equals to a Hausdorff measure with a fixed gauge function for every dyadic $a \in  [0,1]$ by \cite[Theorem~$6.43$]{MP10} and the strong Markov property \ref{Strong Markov}. It then follows from \cite[Lemma~$6.19$]{MP10} that for any $a \in [0,1]$, we can approximate $l^a$ by a sequence of $\{l^{a_n}\}$ where $a_n$ is dyadic. Thus $l^a$ equals to the above mentioned Hausdorff measure for every $a \in [0,1]$ a.s.. This implies that $\mu$ is doubling a.s..
	
	%The measure $\mu$ is doubling. Let $I$ be a $n^{\mathrm{th}}$-generation dyadic square such that $\pi_x(I) = [x_1,x_2]$ and  $\pi_y(I) = [y_1, y_2]$. If $x_1=y_1 = 0$ then we denote by this square $I_1$, and if $x_1=y_2 = 0$ then we denote by this square $I_2$. If we start a new Brownian motion from the first hitting time of $W$ on $I$ for any given $I$, it is then clear that
	%\[
	%\mu(I \cap \Gamma(W)) \leq \mu(I_1 \cup I_2 \cap \Gamma(W)) = 2\mu(I_1  \cap \Gamma(W))
	%\]
	%a.s. by  the strong Markov property \ref{Strong Markov}. Let $t \geq 0$ be a dyadic number and $I_t$ be a square such that
	%\[
	%\pi_x(I_t) = \left[t, t+\frac{1}{2^n}\right] \  \textnormal{and} \ \pi_y(I_t) =  \left[W\!(t), W\!(t)+\frac{1}{2^n}\right].
	%\]
	%Since $\left\{W\!(t+s) - W\!(t), s \geq 0\right\}$ is a standard Brownian motion by the Strong Markov property \ref{Strong Markov}, we have $\mu(I_t \cap \Gamma(W)) = \mu(I_1 \cap \Gamma(W))$ a.s.. Let $x = (t, B(t)) \in \Gamma(W)$ and $r = \frac{1}{2^{n-1}}$. Since we can cover $B(x, 2r)$ by a finite number of $n^{\mathrm{th}}$-generation dyadic squares and $I_t \subseteq B(x,r)$, we have a.s. \[
	%\mu(B(x,2r)) \leq C \mu(I_t \cap \Gamma(W)) \leq \mu(B(x,r))
	%\]
	%for some $C > 0$. For any ball in $\Gamma(W)$, we can approximate it by dyadic $t$ and $r$ constructed above, thus $u$ is doubling a.s..
	
	We claim that for any  $x \in \Gamma(W)$, $ \eps > 0$, there exist $C = C(x, \eps)$ and $r_0 = r_0(x)$ such that when $r < r_0$,
	\begin{equation}\label{umassbound}
		\mu\left( B(x,r) \cap \Gamma(W) \right) \leq C \cdot r^{\frac{3}{2}-\eps}.
	\end{equation}
	
	Recall that, by Theorem \ref{lf}, $L^a(t)$ is a.s. locally $\gamma$-H\"older continuous for any $\gamma < \frac{1}{2}$. Then for any $x \in \Gamma(W)$ and $\eps > 0$, there exist $C_1 = C_1(x, \eps)$ and $r_0 = r_0(x)$ such that whenever $r < r_0$ we have
	\[
	|L^{a_1}(t_1) - L^{a_2}(t_2)| \leq C_1 \left(\sqrt{ |a_1-a_2|^2 + |t_1-t_2|^2 } \right)^{\frac{1}{2}-\eps}
	\]
	for any $(t_1,a_1),(t_2,a_2) \in B(x,r) \cap \Gamma(W)$. Thus for any $a \in \mathbb{R}$,
	\begin{equation*}
		l^a(B(x,r) \cap Z_a) = \sup\{|L^a(s) - L^a(t)| \!:\!  (s,W\!(s)), (t, W\!(t)) \in B(x,r) \cap Z_a \}\leq C_1 r^{\frac{1}{2}-\eps}.
	\end{equation*}
	
	Then
	\begin{align*}
		\mu\left( B(x,r) \cap \Gamma(W) \right) & = \int_{-\infty}^{\infty} l^a(B(x,r) \cap Z_a) da
		\\
		& \leq \int_{\pi_y(x)-r}^{\pi_y(x)+r} C_1 r^{\frac{1}{2}-\eps} da
		\\
		& \leq C r^{\frac{3}{2}-\eps}
	\end{align*}
	for some $C = C(x, \eps) > 0$ a.s..
	
	\begin{Remark}\label{smallradius}
		If $K$ is a compact subset of $\Gamma(W)$, the inequality \eqref{umassbound} is valid for $K$ in greater generality, i.e., for any $\eps>0$, $r >0$, there exist $C = C(x, \eps, K)$ such that $\mu\left( B(x,r) \cap K \right) \leq C \cdot r^{\frac{3}{2}-\eps}$.
	\end{Remark}
	
\subsection{A decomposition of the graph $\Gamma(W)$}\label{verticalCS}
	
In this section we construct a decomposition of the Brownian graph $\Gamma(W)$ into disjoint subsets $\left\{E_{n,j}^m\right\}$ which behave nicely under projections onto the $y$-axis and which will later be instrumental in defining the vertical Cantor sets $E$ in $\Gamma(W)$. We also partition the Brownian graph into two disjoint subsets $A$ and $A'$ so that $\Gamma(W)$ has ``fast growth'' on $A$ and ``slow growth'' on $A'$. We will show that the set $A$ is large enough so that it contains a rich family $\mathcal{E}=\{E\}$ of vertical Cantor sets.  Moreover,  every $E\in\mathcal{E}$ covers the whole interval $\{0\}\times(0,1)$ when projected onto the $y$-axis and has conformal dimension $1$.

We start by letting 
	\[
	X = \Gamma(W) \cap [0, T_6] \times [0,1].
	\]
	In the rest of this section we restrict our proofs and constructions to $X$. Recall that $T_6 = \inf\{ t \geq 0: W\!(t) = 6 \}$ and notice that $T_6 < \infty$ a.s., thus $X$ is compact a.s..
	
	Let $\sigma_1= 0$ and $\upsilon_1 = \inf\left\{t  > 0: W\!(t)= \frac{1}{2}\right\}$. Next, define
	\begin{align*}
		\sigma_{i+1} &= \inf\left\{t: t > \varsigma_i \ \textrm{and} \ W\!(t) = 0\right\};
		\\
		\varsigma_i&= \inf\left\{t: t > \sigma_{i} \ \textrm{and} \ W\!(t) = \frac{1}{2}\right\};
		\\
		\upsilon_{i+1} &= \inf\left\{t: t > \tau_{i} \ \textrm{and} \ W\!(t) = \frac{1}{2}\right\};
		\\
		\tau_i &= \inf\left\{t: t > \upsilon_{i} \ \textrm{and} \ W\!(t) = 1\right\}.
	\end{align*}
	
	To simplify notation, for any $i \in \mathbb{N}$ we denote 
	\begin{align*}
		W\!\!\left(\sigma_i, \varsigma_i, X\right)&=W\!\!\left(\left(\sigma_i, \varsigma_i\right]\right) \cap X \cap \mathbb{R} \times \left(0,\frac{1}{2}\right];
		\\
		W\!\!\left(\varsigma_i, \sigma_{i+1},  X\right)& =W\!\!\left(\left[\varsigma_i, \sigma_{i+1}\right)\right) \cap X \cap \mathbb{R} \times \left(0,\frac{1}{2}\right];
		\\
		W\!\!\left(\upsilon_i, \tau_i, X \right) & = W\!\!\left(\left(\upsilon_i, \tau_i\right]\right) \cap X \cap  \mathbb{R} \times \left(\frac{1}{2},1\right];
		\\
		W\!\!\left(\tau_i, \upsilon_{i+1} , X\right)& = W\!\!\left(\left[\tau_i, \upsilon_{i+1}\right)\right) \cap  X \cap \mathbb{R} \times \left(\frac{1}{2},1\right].
	\end{align*}   
	
	We first focus on $\left\{ W\!\!\left(\sigma_i, \varsigma_i, X\right) \right\}$ and $\left\{ W\!\!\left(\varsigma_i, \sigma_{i+1}, X\right) \right\}$. Note that the collection of sets 
	\begin{align}\label{elements}
	\left\{ W\!\!\left(\sigma_i, \varsigma_i, X\right) \right\} \cup \left\{ W\!\!\left(\varsigma_i, \sigma_{i+1}, X\right) \right\}
	\end{align}
  is composed of finitely many disjoint elements which cover $X \cap \mathbb{R} \times \left(0, \ \frac{1}{2}\right]$.
	
	There exists a number $N \in \mathbb{N}$ such that either $W\!\!\left(\varsigma_N, \sigma_{N+1}, X\right) \neq \emptyset$, and $W\!\!\left(\sigma_{N+1}, \varsigma_{N+1}, X\right) = \emptyset$; or $W\!\!\left(\varsigma_N, \sigma_{N}, X\right) \neq \emptyset$, and $W\!\!\left(\sigma_{N}, \varsigma_{N+1}, X\right) = \emptyset$. We may assume that the first possibility holds. The second case is handled similarly. Since $\pi_y\left(W\!\!\left(\varsigma_N, \sigma_{N+1}, X\right)\right)$ may not cover $\left(0, \ \frac{1}{2}\right]$, we consider its union with the element preceding it, i.e., $W\!\!\left(\sigma_N, \varsigma_N, X\right) \cup W\!\!\left(\varsigma_N, \sigma_{N+1}, X\right)$, and regard this union as one element. Finally, we denote the resulting collection of elements by $\mathcal{E}_1^1$ and order its elements from left to right according to their projections onto the $x$-axis and using an index $j$. Thus, $\mathcal{E}^1_1 = \left\{E^1_{1,j}\right\}$. More precisely,
	\[
	\mathcal{E}^1_1 = \left\{W\!\!\left(\sigma_1, \varsigma_1, X\right)\!, W\!\!\left(\varsigma_1, \sigma_{2}, X\right), \dots, W\!\!\left(\sigma_N, \varsigma_N, X\right) \cup W\!\!\left(\varsigma_N, \sigma_{N+1}, X\right)   \right\}.
	\]
	
	Similarly, we have finitely many disjoint $W\!\!\left(\upsilon_i, \tau_i, X\right)$'s and $W\!\!\left(\tau_i, \upsilon_{i+1}, X\right)$'s that cover $X \cap \mathbb{R} \times \left(\frac{1}{2}, \ 1\right]$. As above, we append the last one to the previous element, enumerate the elements, and denote the resulting collection of elements by $\mathcal{E}^2_1 = \left\{E^2_{1,j}\right\}$. It is clear that $\pi_y\left(E^2_{1,j} \right) = \left(\frac{1}{2},\ 1\right]$ for any $E^2_{1,j} \in \mathcal{E}^2_1$.
	
	We now continue the construction by induction on $n$. Suppose $E_{n,j}^m \subset \Gamma(W)$ is chosen. Then we have  $\pi_y\left(E_{n,j}^m \right) =\left(\frac{m-1}{2^n}, \frac{m}{2^n}\right]$. 
	
	Letting $\sigma_1\left(E_{n,j}^m\right) = \inf\left\{t: (t, W\!(t)) \in E_{n,j}^m\right\}$, we have
	\[
	W\!\!\left(\sigma_1\left(E_{n,j}^m\right) \right) = \frac{m-1}{2^n} \ \mathrm{or} \ \frac{m}{2^n}.
	\]
	
	Without loss of generality, we may assume that $W\!\!\left(\sigma_1\left(E_{n,j}^m\right) \right) = \frac{m-1}{2^n}$. Otherwise a symmetric construction could be done.
	
	Let $\upsilon_1\left(E_{n,j}^m\right)= \inf\left\{t > \sigma_1\left(E_{n,j}^m\right):  \ W\!(t) = \frac{2m-1}{2^{n+1}}\right\}$. Then define
	\begin{align*}
		\sigma_{i+1}\left(E_{n,j}^m\right) &= \inf\left\{t: t > \varsigma_i\left(E_{n,j}^m\right) \ \textrm{and} \ W\!\!\left(t\right) = \frac{m-1}{2^n}\right\};
		\\
		\varsigma_i\left(E_{n,j}^m\right) &= \inf\left\{t: t > \sigma_{i}\left(E_{n,j}^m\right) \ \textrm{and} \ W\!\!\left(t\right) = \frac{2m-1}{2^{n+1}}\right\};
		\\
		\upsilon_{i+1}\left(E_{n,j}^m\right) &= \inf\left\{t: t > \tau_{i}\left(E_{n,j}^m\right) \ \textrm{and} \ W\!\!\left(t\right) = \frac{2m-1}{2^{n+1}}\right\};
		\\
		\tau_i\left(E_{n,j}^m\right) &= \inf\left\{t: t > \upsilon_{i}\left(E_{n,j}^m\right) \ \textrm{and} \ W\!\!\left(t\right) = \frac{m}{2^n}\right\}.
	\end{align*}
	
	Similar to the notation used above for $\mathcal{E}_1^1$ and $\mathcal{E}_1^2$, we denote
	\begin{align*}
		W\!\!\left(\sigma_i, \varsigma_i, E^m_{n,j}\right)&=W\!\!\left(\left(\sigma_i\left(E^m_{n,j}\right), \varsigma_i\left(E^m_{n,j}\right)\right]\right) \cap E^m_{n,j} \cap \mathbb{R} \times \left(\frac{m-1}{2^n}, \frac{2m-1}{2^{n+1}}\right];
		\\
		W\!\!\left(\varsigma_i, \sigma_{i+1},  E^m_{n,j}\right)& =W\!\!\left(\left[\varsigma_i\left(E^m_{n,j}\right), \sigma_{i+1}\left(E^m_{n,j}\right)\right)\right) \cap E^m_{n,j} \cap \mathbb{R} \times \left(\frac{m-1}{2^n}, \frac{2m-1}{2^{n+1}}\right];
		\\
		W\!\!\left(\upsilon_i, \tau_i, E^m_{n,j} \right) &= W\!\!\left(\left(\upsilon_i\left(E^m_{n,j}\right), \tau_i\left(E^m_{n,j}\right)\right]\right) \cap E^m_{n,j} \cap  \mathbb{R} \times \left(\frac{2m-1}{2^{n+1}},\frac{m}{2^n}\right];
		\\
		W\!\!\left(\tau_i, \upsilon_{i+1}, E^m_{n,j}\right) &= W\!\!\left(\left[\tau_i\left(E^m_{n,j}\right), \upsilon_{i+1}\left(E^m_{n,j}\right)\right)\right) \cap  E^m_{n,j} \cap \mathbb{R} \times \left(\frac{2m-1}{2^{n+1}},\frac{m}{2^n}\right].
	\end{align*} 
	
	There are only finitely many elements in $\left\{ W\!\!\left(\sigma_i, \varsigma_i, E_{n,j}^m\right) \right\} \cup \left\{W\!\!\left(\varsigma_i, \sigma_{i+1}, E_{n,j}^m\right) \right\}$. These elements are disjoint and they cover $E_{n,j}^m \cap \mathbb{R} \times \left(\frac{m-1}{2^n}, \frac{2m-1}{2^{n+1}}\right]$. We append the last element to the one before it and denote $\mathcal{E}^{2m-1}_{n+1}=\left\{E^{2m-1}_{n+1,j}\right\}$.
	
	Just like above,  $\left\{W\!\!\left( \upsilon_i, \tau_i, E_{n,j}^m \right)\right\} \cup \left\{ W\!\!\left( \tau_i, \upsilon_{i+1}, E_{n,j}^m \right) \right\}$ contains finitely many elements. These elements are disjoint and they cover $E_{n,j}^m \cap \mathbb{R} \times \left(\frac{2m-1}{2^{n+1}},\frac{m}{2^n}\right]$. After appending the last element to the preceding one, we enumerate them by $\mathcal{E}^{2m}_{n+1} = \left\{E^{2m}_{n+1,j}\right\}$.
	
	\begin{figure}[htbp]
		\centering
		\includegraphics[height = 1.6 in]{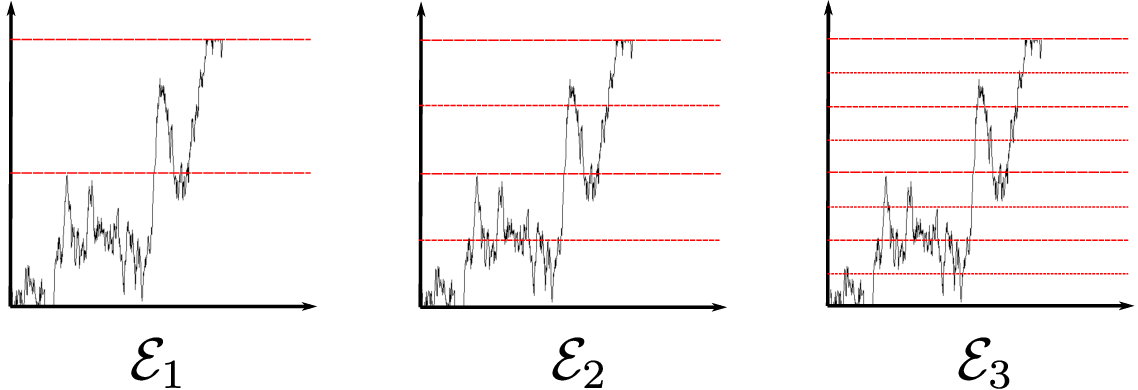}
		\caption{An illustration of the first three generations of $\mathcal{E}$. The red horizontal lines divide the Brownian graph into elements of the generations.}
		\label{BMCS}
	\end{figure}
	
	We denote $\mathcal{E}^m_n = \left\{E_{n,j}^m\right\}$ and $\mathcal{E}_n = \bigcup_{m=1}^{2^n}\mathcal{E}^m_n$ and call these the collection of all the \emph{$n^{\mathrm{th}}$-generation elements at height $m$} of $X$ and the collection of \emph{all the $n^{\mathrm{th}}$-generation elements} of $X$, respectively. 
	
	Moreover, we define
	\[
	E^m_n = \bigcup_{E_{n,j}^m \in \mathcal{E}^m_n} E_{n,j}^m.
	\]
	
	Let
	\[
	\hat{\mathcal{E}}^m_n = \left\{E_{n,j}^m \in \mathcal{E}^m_n: \diam\left( \pi_x\left(E_{n,j}^m\right) \right) \geq \frac{1}{2^n}\frac{1}{n \log 2}\right\}
	\]
	and
	\[
	\mathcal{G}^m_n = \mathcal{E} \setminus \hat{\mathcal{E}}^m_n.
	\]
	Intuitively, $\hat{\mathcal{E}}^m_n$ is the collection of $n^{\mathrm{th}}$-generation elements at height $m$ with ``flat slope'', while $\mathcal{G}^m_n$ are the ones with ``steep slope''.
	
	Similarly, we denote by
	\[
	\hat{\mathcal{E}}_n = \bigcup_{m=1}^{2^n} \hat{\mathcal{E}}^m_n, \ \mathcal{G}_n = \bigcup_{m=1}^{2^n} \mathcal{G}^m_n
	\]
	and
	\[
	\hat{E}_n = \bigcup_{E_{n,j} \in \hat{\mathcal{E}}_n} E_{n,j}, \ G_n = \bigcup_{E_{n,j} \in \mathcal{G}_n} E_{n,j}.
	\]
	
	We define
	\[
	A' = \bigcap_{i=1}^\infty \bigcup_{n=i}^\infty \hat{E}_n = \limsup_{n \to \infty} \hat{E}_n.
	\]
	and
	\[
	A = \bigcup_{i=1}^\infty \bigcap_{n=i}^\infty G_n = \liminf_{n \to \infty} G_n
	\]
	Note that $A$ and $A'$ are the sets of points which are contained in infinitely many and finitely many elements with flat slopes, respectively. Thus $A \cup A' = X$ and $A \cap A' = \emptyset$. 
	
	We would like to show that $A'$ is a small subset of $X$. To that end we define
	\[
	S' = \left\{t \in [0, T_6]: \limsup_{a \to 0^+} \frac{T_{\scriptscriptstyle \pm a}(t)|\log a|}{a} > 1\right\}
	\]
	and
	\[
	S'_n = \left\{t \in [0, T_6]: T_{\scriptscriptstyle \pm \frac{1}{2^n}}(t) > \frac{1}{2^{n+1}} \frac{1}{(n+1)\log 2} \right\}.
	\]
	
	Then $\dim_H(S') \leq \frac{1}{2}$ a.s. by Lemma \ref{hitdim} . Moreover, for any $\eps > 0 $, a.s.
	\begin{equation}\label{bji}
		\Card\left( S'_n \cap \left[\frac{i}{2^{n+1}}, \frac{i+1}{2^{n+1}}\right] \neq \emptyset \right) \leq 2^{\left(\frac{1}{2}+\eps \right)(n+1)}
	\end{equation}
	for sufficiently large $n$ by Remark \ref{sufficientcover}.
	
	Next we state and prove several properties of the sets defined above which will be used in establishing the minimality of the vertical Cantor subsets of $\Gamma(W)$ which we define in Subsection \ref{Section:vertical-Cantor}.
	%The following results provide essential properties for our forthcoming constructions of the vertical Cantor sets.
	
	\begin{Proposition}\label{sdc}
	With the notation as above, we have $\dim_H(A') \leq \frac{1}{2}$ and $\mu(A') = 0.$
	\end{Proposition}
	
	\begin{proof}
		
		We first prove the following statements:
		
		Let $E_{n,j} \in \hat{\mathcal{E}}_n$, then
		\begin{enumerate}
			\item $\pi_x\left(E_{n,j}\right) \cap S'_n \neq \emptyset$;
			
			\item there exists a $C_1 > 0$ such that for any $\left[\frac{i}{2^n}, \frac{i+1}{2^n}\right]$ where $\pi_x\left(E_{n,j}\right) \cap \left[\frac{i}{2^n}, \frac{i+1}{2^n}\right] \neq \emptyset$, we can cover $E_{n,j} \cap \left[\frac{i}{2^n}, \frac{i+1}{2^n}\right] \times \mathbb{R}$ by at most $C_1$ many $n^{\mathrm{th}}$-generation dyadic squares;
			
			\item there exists a $C_2 > 0$ such that
			\[
			\Card\left(\left\{ \left[\frac{i}{2^n}, \frac{i+1}{2^n}\right]: \pi_x\left(E_{n,j}\right) \cap \left[\frac{i}{2^n}, \frac{i+1}{2^n}\right] \neq \emptyset, S'_n \cap \left[\frac{i}{2^n}, \frac{i+1}{2^n}\right] = \emptyset  \right\}\right) \leq C_2.
			\]
		\end{enumerate}
		
		Without loss of generality, we may assume that
		\[
		E_{n,j} = W\!\!\left(\sigma_N, \varsigma_N, E_{n-1,j}\right) \cup W\!\!\left(\varsigma_N, \sigma_{N+1}, E_{n-1,j}\right)
		\]
		for some $E_{n-1,j} \in \mathcal{E}_{n-1}$ and $N \in \mathbb{N}$. Otherwise, we can prove a symmetric situation with the same idea. Here we allow $W\!\!\left(\varsigma_N, \sigma_{N+1}, E_{n-1,j}\right)$ to be empty.
		
		Since $E_{n,j} \in \hat{\mathcal{E}}_n$, then either
		\[
		\diam\left( \pi_x\left( W\!\!\left(\sigma_N, \varsigma_N, E_{n-1,j}\right) \right)\right) \geq \frac{1}{2^{n+1}}\frac{1}{n \log 2}
		\]
		or
		\[
		\diam\left( \pi_x\left(W\!\!\left(\varsigma_N, \sigma_{N+1}, E_{n-1,j}\right) \right) \right) \geq \frac{1}{2^{n+1}}\frac{1}{n \log 2}.
		\]
		Therefore $(1)$ holds.
		
		To prove (2), let $\left[\frac{i}{2^n}, \frac{i+1}{2^n}\right]$ be a $n^{\mathrm{th}}$-generation dyadic interval where $\pi_x\left(E_{n,j}\right) \cap \left[\frac{i}{2^n}, \frac{i+1}{2^n}\right] \neq \emptyset$. Recall that $\pi_y(E_{n,j})$ is a $n^{\mathrm{th}}$-generation dyadic interval (with one end open). This implies that we can cover $E_{n,j} \cap \left[\frac{i}{2^n}, \frac{i+1}{2^n}\right] \times \mathbb{R}$ by two $n^{\mathrm{th}}$-generation dyadic square and we finish the proof of statement $(2)$.
		
		Let $\left[\frac{k}{2^n}, \frac{k+1}{2^n}\right]$ be a $n^{\mathrm{th}}$-generation dyadic interval  such that $\pi_x\left(E_{n,j}\right) \cap \left[\frac{k}{2^n}, \frac{k+1}{2^n}\right] \neq \emptyset$ and $S'_n \cap \left[\frac{k}{2^n}, \frac{k+1}{2^n}\right] = \emptyset$.  Since  $S'_n \cap \left[\frac{k}{2^n}, \frac{k+1}{2^n}\right] = \emptyset$, we have
		\[
		\max \left\{\left|W\!(t)-W\!(s)\right|: t,s \in \left[\frac{k}{2^n}, \frac{k+1}{2^n}\right]   \right\} > \frac{1}{2^n}.
		\]
		This implies that $ \left[\frac{k}{2^n}, \frac{k+1}{2^n}\right] \subsetneq \pi_x\left(E_{n,j}\right)$. Intuitively, $\pi_x\left(E_{n,j}\right)$ doesn't contain this interval since the Brownian motion there has larger ``height'' than that of $E_{n,j}$. In fact, $\pi_x\left(E_{n,j}\right)$ should start or stop in this interval since $\pi_x\left(E_{n,j}\right)$ intersect this interval and doesn't contain it.
		
		Since $\pi_x\left(E_{n,j}\right)$ is one interval or the union of two intervals and $\pi_x\left(E_{n,j}\right)$ should start or stop in $\left[\frac{k}{2^n}, \frac{k+1}{2^n}\right] $, then
		\[
		\Card\left(\left\{ \left[\frac{k}{2^n}, \frac{k+1}{2^n}\right]: \pi_x\left(E_{n,j}\right) \cap \left[\frac{k}{2^n}, \frac{k+1}{2^n}\right] \neq \emptyset, S'_n \cap \left[\frac{k}{2^n}, \frac{k+1}{2^n}\right] = \emptyset  \right\}\right) \leq C_2
		\]
		for some $C_2 > 0$. This finished the proof of statement $(3)$.
		
		Let $\mathcal{D}_n$ be the collection of all the $n^{\mathrm{th}}$-generation dyadic squares that intersect $\hat{E}_n$. Combining with all the above conditions, we conclude that
		\begin{equation}
			\Card\left(\mathcal{D}_n\right) \leq C \cdot \Card\left( S'_n \cap \left[\frac{i}{2^{n+1}}, \frac{i+1}{2^{n+1}}\right] \neq \emptyset \right)
		\end{equation}
		for some $C = C(C_1, C_2)> 0$.
		
		Then, by inequality \eqref{bji}, for any $\eps > 0 $, a.s.
		\begin{equation}\label{sje}
			\Card\left( \mathcal{D}_n\right) \leq 2^{\left(\frac{1}{2}+\eps \right)(n+1)}
		\end{equation}
		for sufficiently large $n$ . Thus
		\[
		\mathcal{H}^{\frac{1}{2}+2\eps}\left(A'\right) \lesssim \lim_{n \to \infty} \sum_{l=n}^{\infty}2^{\left(\frac{1}{2}+\eps \right)(l+1)}2^{-(\frac{1}{2}+2\eps)l} = 0.
		\]
		This implies that $\dim_H(A') \leq \frac{1}{2}$ a.s..
		
		Since $\dim_H(A') \leq \frac{1}{2}$ a.s., for any $\delta > 0$, there exists a countable sequence of balls $\left\{B(x_i,r_i) \right\}$ covering $A'$ such that  $\sum_{i=1}^{\infty} r_i < \delta$.
		
		It follows from Remark \ref{smallradius} that
		\[
		\mu\left( A' \right) \leq \sum_{i=1}^{\infty}\mu\left( B(x_i,r_i) \right) \lesssim \sum_{i=1}^{\infty} r_i \lesssim \delta.
		\]
		This finishes the proof.
	\end{proof}
	
	\begin{Remark}\label{secontrol}
		Since for any $E_{n,j} \in \hat{\mathcal{E}}_n$
		\[
		\diam\left( \pi_x\left(E_{n,j}\right) \right) \geq \frac{1}{2^n}\frac{1}{n \log 2},
		\]
		we have that any $n^{\mathrm{th}}$-generation dyadic square can intersect at most $[\log 2^n]+1$ many elements from $\hat{\mathcal{E}}_n$. It then follows from inequality \eqref{sje} that for any $\eps > 0 $, a.s.
		\begin{equation}
			\Card\left( \hat{\mathcal{E}}_n \right) \leq 2^{\left(\frac{1}{2}+\eps \right)(n+1)}
		\end{equation}
		for sufficiently large $n$.
	\end{Remark}
	
	\begin{Lemma}\label{gecei}
		Almost surely, for any $E_{n,j} \in \mathcal{E}_n$,
		\[
		\pi_y(A \cap E_{n,j}) = \pi_y\left(E_{n,j}\right)
		\]
		$\mathcal{L}^1$-almost everywhere.
	\end{Lemma}
	
	\begin{proof}
		Note that by the Strong Markov Property \ref{Strong Markov}, it is enough to show that almost surely we have
		\begin{align}\label{inclusion}
			\pi_y\left(A \cap [0, T_1] \times (0,1]\right) \supset (0,1).
		\end{align}
		
		We first show that a.s. for any $a \in (0,1)$,
		\begin{equation}\label{goodmeasure}
			l^a(A \cap \Gamma(W\!([0, T_1]))\cap Z_a) \geq \frac{1}{2} l^a\left(\Gamma(W\!([0, T_1])) \cap Z_a\right).
		\end{equation}
		
		We denote by $\mathcal{E}_n(a) = \{E_{n,j}(a)\}$ the subcollection of $\mathcal{E}_n$ such that
		\begin{enumerate}
			\item $E_{n,j}(a) \subset \Gamma(W\!(0, T_1))$,
			
			\item $a  \in \pi_y\left( E_{n,j}(a) \right)$,
		\end{enumerate}
		In other words, $\mathcal{E}_n(a)$ is the collection of all the $n^{\mathrm{th}}$-generation elements in $[0, T_1] \times [0,1]$ that intersect $Z_a$. Thus we have $\pi_y\left( E_{n,i}(a) \right) = \pi_y\left( E_{n,j}(a) \right), \ \forall \ E_{n,i}(a), E_{n,j}(a) \in \mathcal{E}_n(a)$.
		
		Recall that $D_n(a, T_1)$ is the number of downcrossings before time $T_1$ of the unique $n^{\mathrm{th}}$-generation dyadic interval $(i2^{-n} , (i+ 1)2^{-n}]$ containing $a$, thus
		\begin{equation}\label{elementlt}
			\frac{1}{2}D_n\left( a, T_1\right) \leq \Card\left( \mathcal{E}_n(a)\right) \leq 2 D_n\left( a, T_1\right).
		\end{equation}
		
		For sufficiently large $i$, a.s.
		\[
		\Card\left( \hat{\mathcal{E}}_n \right) \leq 2^{\left(\frac{2}{3} \right)n}
		\]
		by Remark \eqref{secontrol}.
		
		Theorem \ref{ltv} implies that a.s. for any $a \in (0,1)$, $l^a\left(\Gamma(W\!(0, T_1)) \cap Z_a\right) > \delta$ for some random $\delta = \delta(a, \omega)$.
		
		Therefore, there exists $N = N(a, \omega)>0$ such that
		\begin{equation}\label{ltcontrol}
			\sum_{n=N}^\infty\frac{\Card\left( \hat{\mathcal{E}}_n \cap \mathcal{E}_n(a) \right)}{2^{n-1}} \leq \sum_{n=N}^\infty\frac{\Card\left( \hat{\mathcal{E}}_n\right)}{2^{n-1}} \leq \sum_{n=N}^\infty\frac{2^{\left(\frac{2}{3} \right)n}}{2^{n-1}} < \frac{1}{4}\delta.
		\end{equation}
		
		Since $A' \subset \bigcup_{n = N}^\infty \bigcup_{E_{n,j} \in \hat{\mathcal{E}}_n}E_{n,j}$, we have a.s. for any $a \in (0,1)$,
		\begin{align*}
			l^a(A' \cap \Gamma(W\!([0, T_1]))\cap Z_a) & \leq \lim_{n \to \infty}\frac{2\Card\left(\hat{\mathcal{E}}_n \cap \mathcal{E}_n(a)\right)}{2^{n-1}}
			\\
			& < \frac{\delta}{2} < \frac{1}{2} l^a(\Gamma(W) \cap [0, T_1] \times [0,1] \cap Z_a).
		\end{align*}
		The first inequality above comes from the definition of local time and inequality \eqref{elementlt}. This implies that inequality \eqref{goodmeasure} is valid. 
		
		It follows from Theorem \ref{ltv} that
		\[
		l^a(A \cap \Gamma(W\!([0, T_1]))\cap Z_a) \geq \frac{1}{2} l^a\left(\Gamma(W\!([0, T_1])) \cap Z_a\right) > 0.
		\]
		Then a.s. for any $a \in (0,1)$, $A \cap \Gamma(W\!([0, T_1]))\cap Z_a \neq \emptyset$, which gives \eqref{inclusion}.
	\end{proof}
	
	\begin{Remark}\label{cover}
		
		Applying Strong Markov Property \ref{Strong Markov} to inequality \eqref{goodmeasure}, we have the following result: Almost surely for any  $E_{n,j} \in \mathcal{E}_n$,
		\[
		l^a(A \cap E_{n,j} \cap Z_a) \geq \frac{1}{2} l^a(E_{n,j} \cap Z_a)
		\]
		for a.e. $a \in \pi_y(E_{n,j})$.
	\end{Remark}
	
	\subsection{Vertical Cantor sets in the graph $\Gamma(W)$}\label{Section:vertical-Cantor}
	
	Next, we construct a family of \emph{vertical Cantor sets} $\mathcal{E}$ in $X$ as follows. We choose one element in $\mathcal{E}^1_1$ and another element in $\mathcal{E}^2_1$.  The union of these two elements is the \emph{first generation} of $E$ and we denote it by $E_1$. Suppose the $n^{\mathrm{th}}$-generation $E_n$ of $E$ is constructed and it is a union of $2^n$ disjoint elements from $\mathcal{E}_n$ whose projection to $y$-axis covers $(0,1]$. For each $E_{n,j}$ in $E_n$, we choose two disjoint elements from $\mathcal{E}_{n+1}$ that are subsets of $E_{n,j}$ and their projections to the $y$-axis coincide with $\pi_y(E_{n,j})$. The union of all the $2^{n+1}$ elements selected as above forms the \emph{$(n+1)^{\mathrm{th}}$-generation} $E_{n+1}$ of $E$.
	Define a vertical Cantor set $E$ by setting
	\begin{align}
	E= A \cap \left(\bigcap_{n=1}^\infty E_n\right) = A \cap \left(\bigcap_{n=1}^{\infty} \bigcup_{j=1}^{2^n} E_{n,j}\right).
	\end{align}
   
	We denote by  $\mathcal{E}$ the collection of all vertical Cantor sets in $A$ constructed as above.
	%
	%We collect all the vertical Cantor sets in $A$ that can be produced by the above method and denote the whole collection by $\mathcal{E}$.
	%
	Observe that $ E \in \mathcal{E}$ may not be a metric space with hierarchical structure as defined in Section \ref{CDM}, since a nested sequence of $\left\{A \cap E_{n,j}\right\}$ could have an empty intersection.
	
	The reason for defining $E$ as a subset of $A$ is to have certain properties which would facilitate the application of the general theory of Section \ref{CDM} to the vertical Cantor sets. The rest of this subsection is devoted to establishing these properties.
	
	\begin{Lemma}\label{vcurve}
		For any $E \in \mathcal{E}$, $E \cap Z_a$ contains only one point for a.e. $a \in [0,1]$.
	\end{Lemma}
	
	\begin{proof}
		We first notice that $E \cap Z_a \neq \emptyset$ for any $E \in \mathcal{E}$ and a.e. $a \in (0,1]$ by Lemma \ref{gecei}.
		
		Let $a \in (0,1]$ and assume that $p,q \in E \cap Z_a$. Then there exists a number $n \in \mathbb{N}$ such that $p \in E_{n,j}$ for some $E_{n,j} \in \mathcal{E}_n$ and $\diam\left(E_{n,j}\right) < |p-q|$ by our construction. This implies that $q \in E_{n,i}$ for some $E_{n,i} \in \mathcal{E}_{n}$ and $E_{n,j} \neq E_{n,i}$. Recall that $\pi_y\left(E_{n,j}\right) \cap \pi_y\left(E_{n,i}\right) = \emptyset$ when $E_{n,j} \neq E_{n,i}$ and both of them are in the $n^{\mathrm{th}}$-generation of $E$. This contradicts our assumption that  $p,q \in E \cap Z_a$ and therefore finishes the proof.
	\end{proof}
	
	Lemma \ref{vcurve} implies that the restriction of $\pi_y$ to every $E$ is injective. Moreover, $\pi_y$ is also  surjective on a full measure subset of $[0,1]$ by Lemma \ref{gecei}.  Therefore,  we can define a measure $\lambda_{E}$ by pushing forward the Lebesgue measure from $\{0\}\times [0,1]$ to $E$, i.e., $\lambda_{E}=(\pi_y^{-1})_*(\mathcal{L}^1)$,  or more concretely for every measurable $U \subset E$ we have
	\begin{align*}
		\lambda_{E}(U)=\mathcal{L}^1(\pi_y(U)).
	\end{align*}
	We will denote $\mathbf{E}=\{\lambda_E\}_{E\in\mathcal{E}}$, i.e., the family of pull backs of the Lebesgue measure under $\pi_y$ to all the vertical Cantor sets in $\mathcal{E}$.
	
	\begin{Lemma}\label{dimcgood}
		$\dim_C\left(E\right) \geq 1$ for any $E \in \mathcal{E}$.
	\end{Lemma}
	
	\begin{proof}
		For any $E_{n,j} \in \mathcal{E}_n$, we denote by $F_{n,j}$ a compact subset of $E_{n,j}$ such that
		\[
		\diam(F_{n,j}) \geq \diam(E_{n,j}) - \frac{1}{128^n}.
		\]
		Recall that $A = \bigcup_{i=1}^\infty \bigcap_{n=i}^\infty G_n$, where $G_n = \bigcup_{E_{n,j} \in \mathcal{G}_n} E_{n,j}$. Let $F_n = \bigcup_{E_{n,j} \in \mathcal{G}_n} F_{n,j}$. Then we denote $K = \bigcup_{i=1}^\infty \bigcap_{n=i}^\infty F_n$ and $K^N = \bigcup_{i=1}^N \bigcap_{n=1}^\infty F_n$.
		
		Let $E \in \mathcal{E}$ where $E = A \cap \left( \bigcap_{n=1}^\infty \bigcup_{j=1}^{2^n} E_{n,j} \right)$. We define $F = K \cap \left( \bigcap_{n=1}^\infty \bigcup_{j=1}^{2^n} F_{n,j} \right)$ and $F^N = K^N \cap \left( \bigcap_{n=1}^\infty \bigcup_{j=1}^{2^n} F_{n,j} \right)$. Thus $\bigcup_{N=1}^\infty F^N \subseteq F \subseteq E$.
		
		Since $F^N$ and $F_{n,j}$ are compact for any $n,j$ and $N$, the intersection of a nested sequence of $\left\{F^N \cap F_{n,j}\right\}$ is nonempty. Letting $F^N_{n,j} = F^N \cap F_{n,j}$, then
		\[
		F^N =  \bigcap_{n=1}^\infty \bigcup_{j=1}^{2^n} F^N_{n,j} 
		\]
		is a metric space with hierarchical structure. 
		
		Let $x \in F^N(x)$ and $A(x)$ be the flatness constant of $x$ in $F^N$. Notice that for any disjoint $F^N_{n,j}$ and $F^N_{m,l}$,
		\[
		\diam\left(F^N_{n,j} \cup F^N_{m,l}\right) \geq \diam\left(\pi_y\left(F^N_{n,j}\right)\right) + \diam\left(\pi_y\left(F^N_{m,l}\right)\right).
		\]
		Then the construction of $F^N$ and the disjointedness of $F_{n,j}$ implies that $A(x) \leq 2$.
		
		Notice that $\diam\left( \pi_y\left(A \cap E_{n,j}\right) \right) = \frac{1}{2^n}$ by Lemma \ref{gecei}. Thus by the definition of $F_{n,j}$, we have
		\[
		\diam\left( \pi_y(K \cap F_{n,j}) \right) \geq \frac{1}{2^n} - \frac{1}{16^n}
		\]
		and
		\[
		\diam\left( \pi_y(F \cap F_{n,j}) \right) \geq \frac{1}{2^n} - \frac{1}{8^n}.
		\]
		This implies that
		\begin{equation*}
			\lim_{i \to \infty}\lim_{N \to \infty}\Delta(F^N_{i}(x), (F^N_{i})'(x)) \leq \lim_{i \to \infty} \frac{2}{i \log 2} = 0
		\end{equation*}
		and
		\begin{equation}
			\frac{1}{2} \leq \lim_{N \to \infty}\frac{\diam(F^N_i(x))}{\diam((F^N_i)'(x))} \leq 2.
		\end{equation}
		for $i$ sufficiently large.
		
		Applying Theorem \ref{unioncd1} to $F^N$, we have $\dim_C(E) \geq \dim_C(F) \geq 1$.
	\end{proof}
	
	\begin{Lemma}\label{linear1}
		Let $E \in \mathcal{E}$ and $\lambda_{E} \in \mathbf{E}$. Then for any $x \in E$, there exists $r_1 = r_1(x)$ such that whenever $r < r_1$,  we have
		\begin{equation}\label{eqn:lower-bound}
			\lambda_E\left(B(x,r) \cap E \right) \geq \frac{1}{3}r.
		\end{equation}
	\end{Lemma}
	
	\begin{proof}
		Let $E_n(x)$ be the unique element in $\mathcal{E}_n(E)$ that contains $x$.  Since $x \in A$, we have that $E_n(x) \subseteq \mathcal{G}_n$ for sufficient large $n$. This implies that there exists $N = N(x)$ such that (\ref{eqn:lower-bound}) holds for  $r < \diam\left( E_N(x) \right)$.
	\end{proof}
	
	\subsection{Proof of Theorem \ref{bg}}
	
	We are now ready to prove Theorem \ref{bg}.
	
	\begin{proof}[Proof of Theorem \ref{bg}]
		
		We list the conditions that are already known: Almost surely
		\begin{itemize}
			\item for any $\eps>0$, there exist $C = C(x, \eps)$ and $r = r_0(x)$ such that when $r < r_0$,
			\[
			\mu\left( B(x,r) \cap \Gamma(W) \right) \leq C \cdot r^{\frac{3}{2}-\eps}.
			\]
			See inequality \eqref{umassbound}.
			
			\item for any $E \in \mathcal{E}$ with $\lambda_E \in \mathbf{E}$ equipped on $E$, and $x \in E$, there exists a $r_1 = r_1(x)$ such that when $r < r_1$,
			\begin{equation*}
				\lambda_E(B(x,r) \cap E) \geq \frac{1}{3}r.
			\end{equation*}
			See Lemma \ref{linear1}.
			
			\item $\dim_C(E) \geq 1$ for any $E \in \mathcal{E}$. See Lemma \ref{dimcgood}.
			
			\item $\mu(A') = 0$. See Proposition \ref{sdc}.
		\end{itemize}
		
		With Theorem \ref{fmodest} and all the conditions listed above, it is sufficient to prove that $\Mod_1(\mathbf{E}) > 0$ a.s.. For that let $\rho:X\to[0,\infty)$ be admissible for $\mathbf{E}$, that is $\int_E \rho d \lambda_E>1$ for every $E\in\mathcal{E}$. By the Vitali-Carath\'eodory theorem \ref{VCT} we can assume that $\rho$ is in fact lower-semicontinuous.
		
		As in Lemma \ref{lemma:construction}, we will construct a  measurable function $\rho_{\infty}{:} X \to [0, \infty)$ such that
		\begin{enumerate}
			\item $\int \rho d\mu \geq \int \rho_\infty d \mu$,\label{srandommeasure}
			
			\item  $\rho_\infty(x_1,y_1)=\rho_\infty(x_2,y_2)$ if $y_1=y_2$,\label{equalrandommeasure}
			
			\item $\exists F\in\mathcal{E}$ such that $\rho(x)\leq \rho_\infty(x)$ for $x\in F$.\label{arandommeasure}
		\end{enumerate}
		
		We let
		\[
		\mathcal{A}^m_n = \left\{A \cap E^m_{n,j}: E^m_{n,j} \in \mathcal{E}^m_n \right\} \ \textnormal{and} \ \mathcal{A}_n = \bigcup_{m=1}^{2^n} \mathcal{A}^m_n.
		\]
		
		\begin{figure}[htbp]
			\centering
			\includegraphics[height = 1.2 in]{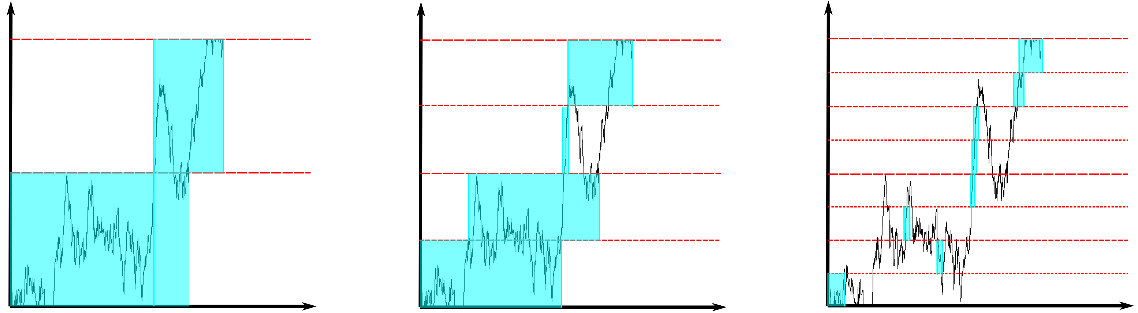}
			\caption{A illustration of $F$ for three generations.}
			\label{SBM}
		\end{figure}
		
		We will first construct a sequence of ``piecewise'' constant function $\rho_n$ from $\rho$ as follows.
		
		From each of $\mathcal{A}_1^m, m=1,2$, choose one element with the smallest average of $\rho$. Specifically, let $F_{1,1} \in \mathcal{A}^1_1$ such that
		\[
		\intbar_{F_{1,1}} \rho \ d\mu \leq \intbar_{A^1_{1,j}} \rho \ d\mu
		\]
		for any $A^1_{1,j} \in \mathcal{A}_1^1$. Similarly, let $F_{1,2} \in\mathcal{A}_1^2$ be chosen so that
		\[
		\intbar_{F_{1,2}} \rho \ d\mu \leq \intbar_{A^2_{1,j}} \rho \ d\mu
		\]
		for any $A^2_{1,j} \in\mathcal{A}_1^2$. Such elements exist since $\Card\left(\mathcal{A}_1^m \right) < \infty, m=1,2$.
		
		Define $\rho_1: X \to [0, \infty)$ by
		\[
		\rho_1\big\rvert_{A^m_{1,j}} \equiv \intbar_{F_{1,m}}\rho d\mu,   \mbox{ if } A^m_{1,j} \in \mathcal{A}^m_1.
		\]
		for $m =1,2$ and $\rho_1 = 0$ on $X \setminus A$.
		
		Thus $\rho_1$ is the same constant on every element of $\mathcal{A}_1^m, m=1,2$.  Hence,  $\rho_1$ is also constant on $A \cap \pi^{-1}_y(a)$ for any $a \in [0,1]$, and $\int_X \rho d\mu \geq \int_X \rho_1 d\mu$.
		
		Continuing by induction, suppose that at step $n-1$ we have chosen rectangles $F_{n-1,1},\ldots,F_{n-1,2^{n-1}}$ and defined $\rho_{n-1}$ so that it is a constant on all elements of $\mathcal{A}^m_{n-1}, m=1, \ldots, 2^{n-1}$ and $0$ otherwise.
		
		To construct $\rho_{n}$,   from each of the elements from $\mathcal{A}^m_n$ that is also within $F_{n-1,j}$ for  $j\in\{1,\ldots,2^{n-1}\}$  chosen above, we will select one with the smallest average of $\rho$, i.e.,  the element $F_{n,l}$ with $l \in\{1,\ldots,2^{n}\}$ are such that
		\begin{align*}
			\intbar_{F_{n,l}} \rho d\mu \leq \intbar_{A_{n,k}} \rho d\mu,
		\end{align*}
		whenever $A_{n,k}, F_{n,l}\subset F_{n-1,j}$ for some $j\in\{1,\ldots,2^{n-1}\}$, $A_{n,k} \in \mathcal{A}_n$ and $\pi_y(A_{n,k})=\pi_y (F_{n,l})$ a.e..
		
		Define $\rho_n$ be setting
		\begin{align*}
			\rho_n\big\rvert_{A_{n,k}} \equiv \intbar_{F_{n,j}}\rho d\mu,   \mbox{ if } \pi_y(A_{n,k})=\pi_y (F_{n,j})
		\end{align*}
		and $\rho_n = 0$ on $X \setminus A$.
		
		Then we have a sequence of Borel measurable functions $\{\rho_n \}$. It follows from the construction above  that we have
		\[
		\int_X \rho \ d\mu \geq \int_X \rho_1 \ d\mu \geq \ldots\geq
		\int_X \rho_n \ d\mu \geq \ldots.
		\]
		
		For every $z=(x,y)\in X$ we define
		\begin{align}\label{definenewrho}
			\rho_\infty(z) := \liminf_{n\to\infty} \rho_n(z).
		\end{align}
		
		Condition \eqref{srandommeasure} then follows from the definition of $\rho_\infty$ and Fatou's Lemma.  Equality \eqref{equalrandommeasure} follows from \eqref{definenewrho} and the fact that for every $n$ the function $\rho_n$ is constant on all horizontal slices $A \cap Z_a$.
		
		Define
		\[
		F=\bigcap_{n=1}^{\infty}\bigcup_{j=1}^{2^n} F_{n,j}
		\]
		and $F$ is clearly an element in $\mathcal{E}$. The proof of \eqref{arandommeasure} then is exactly the argument used at the end of the proof of Lemma \ref{lemma:construction} and we omit the details here.
		
		Note that by \eqref{equalrandommeasure} there exists a unique $\bar{\rho}: [0,1] \to \mathbb{R}$ s.t.  for $\mathcal{L}^1$-a.e. $a \in [0,1]$ we have $\bar{\rho}(a) = \rho_\infty(z),$ whenever $\pi_y(z) = a$. Then, by $(3)$, for any $E \in \mathcal{E}$ we have
		\begin{equation}\label{newadm}
			\int_E \rho_\infty(z) \ d \lambda_E = \int_0^1 \bar{\rho}(a) \ da =  \int_F \rho_\infty(x) \ d \lambda_F \geq \int_F \rho(x) \ d \lambda_F \geq 1.
		\end{equation}
		Thus $\rho_\infty$ is also admissible for $\mathcal{E}$.
		
		Recall that for any $a \in [0,1]$, we have $L^a(T_6) > 0$ a.s. by Theorem \ref{ltv}. Since $L^a(T_6)$ is continuous a.s. by Theorem \ref{lf}, a.s. there exists a $\delta= \delta(\omega) > 0$ such that $L^a(T_6) > \delta$ for any $a \in [0,1]$. This implies that a.s.
		\begin{equation}\label{hitmeasure}
			l^a\left(\Gamma(W) \cap [0, T_6] \times [0,1] \cap Z_a\right) > \delta(\omega)
		\end{equation}
		for any $a \in [0,1]$. It then follows from Remark \ref{cover} that a.s.
		\begin{equation}\label{hestimate}
			l^a(A \cap Z_a) \geq \frac{1}{2}\delta(\omega)
		\end{equation}
		for a.e. $a \in [0,1]$.
		
		Finally, using \eqref{srandommeasure}, the Disintegration Theorem \cite[Theorem $5.3.1$]{AGS05},  and the inequalities above, we obtain
		\begin{align*}
			\int_X \rho d\mu & \geq \int_X \rho_\infty d\mu = \int_A \rho_\infty d\mu
			\\
			&=  \int_0^1 \left[ \int_{A \cap \pi_y^{-1}(a)}\rho_\infty(z) \ d l^a(z) \right] \ da
			\\
			&=  \int_0^1 \bar{\rho}(a) l^a(A \cap Z_a)\ da
			\\
			&  \geq \frac{\delta}{2}\int_0^1 \bar{\rho}(a) \ da \tag{inequalities \eqref{hestimate} and \eqref{newadm}}
			\geq \frac{\delta}{2}.
		\end{align*}
		
		This implies that $\Mod_1(\mathbf{E}) > 0$ a.s. and thus finishes the proof.
	\end{proof}
	
	\section{Remarks and open problems. }\label{Conclusion}
	In light of Theorems \ref{bg} and \ref{thm:mg} one may expect that the graph $\G$ of a continuous function $f: [0,1] \to \mathbb{R}$  is always minimal for conformal dimension. While this may not be true in general, we expect this to be the case provided the intersections of $\G$ with horizontal lines $Z_a=(-\infty,\infty)\times\{a\}$ are often large. Specifically, we propose the following problem.
	
	\begin{Conjecture}
		Let $\G$ be the graph of a continuous function $f: [0,1] \to \mathbb{R}$. 
		Suppose	there is a set $E\subset \mathbb{R}$ of positive Lebesgue measure so that  
		\begin{align}\label{large-fibres}
			\dim_H(\G\cap Z_a) =\dim_H \G -1
		\end{align}
		for all $a\in E$. Then $\G$ is minimal for conformal dimension.
	\end{Conjecture}
	
	To see the importance of continuity and condition (\ref{large-fibres}) in the conjecture above we next observe that if they fail then the graph of $f$ is not necessarily minimal. Indeed,  by a theorem of Tukia \cite{Tukia} there is a set $E\subset[0,1]$ and a quasisymmetric mapping $g:\mathbb{R}\to\mathbb{R}$ such that $\dim_H E<1$ and $\dim_H g([0,1]\setminus E)<1$. Therefore, by Kovalev's theorem \cite{Kov06} both $E$ and $[0,1]\setminus E$ have conformal dimension $0$. Now, $\chi_E([0,1])$ can be written as the disjoint union of $A=E\times\{1\}$ and $B=([0,1]\setminus E)\times \{0\}$. By Proposition 5.2.3 in \cite{MT10} we have $\dim_C(A\cup B)=\max\{\dim_C A, \dim_C B\}$ provided $A$ and $B$ are disjoint compact subsets of a metric space. The proof of that result holds for non compact subsets as well, provided $\dist(A,B)>0$. Since in our example we have $\dist(A,B)\geq 1$, $\dim_C A=\dim_C B=0$ it follows that $\dim_C \chi_E([0,1])=\dim_C ( A\cup B) = 0$. On the other hand, since $\dim_H \chi_E([0,1]) \geq  \dim_H B=1$ it follows that  $\chi_E([0,1])$ is not minimal for conformal dimension.
	
	%We believe that  constructions as above are not possible for continuous functions.
	
	%, we can exemplify a non-minimal planar graph by considering the graph of the characteristic function of the standard Cantor set. If the graph is not planar, we can provide an example of a non-minimal continuous three-dimensional graph by using the graph of the characteristic function of a Von Koch snowflake. To address the case where we require our graph to be both continuous and planar, we introduce the following conjecture.
	
	Another natural open problem is to find the conformal dimension of the trace (not the graph) of $d$-dimensional Brownian motion for every $d\in \mathbb{N}$ almost surely. Clearly, the trace of $1$-dimensional Brownian motion is a.s. minimal.  Since Hausdorff dimension of the trace of planar Brownian motion is $2$, it follows from \cite{Gehring} that it is minimal for quasiconformal mappings of the plane. We conjecture that it is also minimal for conformal dimension (that is for mapping into arbitrary metric spaces). Since Brownian motion is transient for $n \geq 3$, it is reasonable to assume that it is not minimal for sufficiently high  dimensions. In fact, as the dimension increases, there is greater flexibility to stretch and compress the Brownian motion, leading to the following conjectures.
	
	\begin{Conjecture}
		The $n$-dimensional Brownian motion is not minimal almost surely for sufficiently large $n$. Moreover, it has conformal dimension $1$ almost surely for sufficiently large $n$. 
	\end{Conjecture}
	
	One of the most widely studied stochastic objects in recent years has been the \emph{Schramm--Loewner Evolution (SLE)}. SLE was introduced by Oded Schramm in \cite{Sch00}, and there has been significant progress in the last decade. It is a random growth process generated from a family of Riemann mappings $\{g_t(z)\}$, where $\{g_t(z)\}$ is the solution of the following ordinary differential equation: For any $z \in \{\omega \in \mathbb{C}: \Im\omega \geq 0 \} \setminus \{0\}$,
	\begin{equation}\label{Loewner}
		\partial_t g_t(z) = \frac{2}{g_t(z) - \sqrt{\kappa}B(t)}, \  g_0(z) = z.
	\end{equation}
	Equation \ref{Loewner} is the Loewner differential equation with a driving parameter of a standard one-dimensional Brownian motion running with speed $\kappa$. This process is also considered as conformal invariant random curves in simply connected planar domains. It has been proven to be the scaling limit of a variety of two-dimensional lattice models in statistical mechanics. Readers may refer to \cite{Law05} for more information.
	
	It is proved in \cite{RS05} that for all $\kappa \neq 8$ the $\SLE_{\kappa}$ trace is a continuous path. It is a simple path for $\kappa \in [0,4]$, a self-intersecting path for $\kappa \in (4,8)$ and space-filling for $\kappa > 8$. Thus it is only interesting to study $\SLE_{\kappa}$ for $\kappa \in [0,8]$. Beffara proved in \cite{Bef08} that the Hausdorff dimension of the $\SLE_{\kappa}$ trace is $1+ \frac{\kappa}{8}$ for $\kappa \in [0,8]$ almost surely. \cite{AK08} and \cite{AS08} study the intersection of the SLE trace with $\mathbb{R}$ or semi-circles and their results show that the SLE trace has a product-like structure.
	
	Quasisymmetric deformations of SLE appear naturally as the scaling limits of critical models in statistical physics on skewed lattices. The question of the conformal dimension of SLE can be heuristically formulated in Physical terms as follows: Can we decrease the scale of interactions in a critical lattice model by perturbing the lattice? In an effort to answer this question, we set the following conjecture:
	
	\begin{Conjecture}
		The $\SLE_\kappa$ trace is minimal for conformal dimension almost surely, i.e., the conformal dimension of the $\SLE_\kappa$ curve is $1+\frac{\kappa}{8}$ almost surely for any $\kappa \in [0,8]$.
	\end{Conjecture}
	
	A compact set $K\subset \mathbb{C}$  is \emph{(quasi)conformally removable} if any homeomorphism of $\mathbb{C}$, which is (quasi)conformal on $\mathbb{C} \setminus K$, is (quasi)conformal on $\mathbb{C}$.  It follows from the measurable Riemann mapping theorem that a set is conformally removable if and only if it is quasiconformally removable.  See \cite{You15} for a survey of results on (quasi)conformal removability.  Intuitively,  non-removable or minimal subsets of the plane can both be thought of as  being ``large'' in a quasiconformally invariant sense.  For instance, a product set of the form $E\times[0,1]$, where $E$ is uncountable, is both minimal as well as non-removable.  Nonetheless, there exists no direct relationship between the concepts of non-removability and minimality for subsets of the pane.
	
	In \cite{JS00} it was established that the boundaries of John domains and H\"older domains are conformally removable.   Combining this with \cite[Theorem $5.2$]{RS05} it follows that $\SLE_{\kappa}$ is conformally removable for $\kappa \in (0,4)$.  More recently,   it was shown in \cite{KMS22} and \cite{KMS23} that $\SLE_{\kappa}$ is almost surely conformally removable for $\kappa \in [4,\delta)$, where $\delta > 4$. In \cite{DM23}, the Sobolev removability of $1$-dimensional Brownian graph $\Gamma$ was studied.  They showed that almost surely $\Gamma$ is not $W^{1,p}$-removable for all $p \in [1, \infty)$ but $\Gamma$ is $W^{1, \infty}$-removable. However, the conformal removability of $\Gamma$ is still an open question. It is known that $W^{1,2}$ removability implies conformal removability. Therefore, we conjecture that the graph of the $1$-dimensional Brownian motion is not removable almost surely.


\begin{thebibliography}{MTW12}
		
		\bibitem[AK08]{AK08} T. Alberts and M. Kozdron, Intersection probabilities for a chordal SLE path and a semicircle, Electron. Commun. Probab., \textbf{13}, (2008), 448--460.
		
		\bibitem[AS08]{AS08} T. Alberts and S. Sheffield, Hausdorff Dimension of the SLE Curve Intersected with the Real Line, Electron. J. Probab., \textbf{13}, (2008), 1166--1188.
		
		\bibitem[AGS05]{AGS05} L. Ambrosio, N. Gigli and G. Savaré, Gradient Flows in Metric Spaces and in the Space of Probability Measures, Lectures in Mathematics, ETH Z\"urich, Birkh\"auser Verlag, (2005).
		
		\bibitem[Bed84]{Bedford}T. Bedford, Crinkly curves, Markov partitions and box dimensions in self-similar sets. PhD thesis,
		University of Warwick (1984).
		
		\bibitem[Bef08]{Bef08} V. Beffara, The dimension of the SLE curves, Ann. Probab., \textbf{36}, (2008), 1421--1452.
		
		\bibitem[BHW16]{BHW16} C. Bishop, H. Hakobyan and M. Williams, Quasisymmetric dimension distortion of Ahlfors regular subsets of a metric space, Geom. Funct. Anal., \textbf{26}, (2016), 379--421.
		
		\bibitem[BP17]{BP17} C. Bishop and Y. Peres, Fractals in Probability and Analysis, Cambridge Stud. Adv. Math., \textbf{162}, Cambridge Univ. Press, (2017).
		
		\bibitem[BT01]{BT01} C. Bishop and J. Tyson, Locally minimal sets for conformal dimension, Ann. Acad. Sci. Fenn., \textbf{26}, (2001), 361--373.
		
		\bibitem[Boj88]{Boj88} B. Bojarski, Remarks on Sobolev imbedding inequalities, Complex Analysis Joensuu, Lecture Notes in Math., \textbf{1351}, Springer, Berlin, (1988), 52--68.
		
		\bibitem[BK05]{Bonk-Kleiner:confdim} M. Bonk and B. Kleiner,  Conformal dimension and Gromov hyperbolic groups with 2-sphere boundary,  Geom. Topol.,  \textbf{9}, (2005), 219--246.
		
		\bibitem[BM17]{Bonk-Meyer} M. Bonk and D.  Meyer,  Expanding Thurston maps,  Math. Surveys Monogr., \textbf{225}, Amer. Math. Soc.,  Providence,  RI, (2017).
		
		\bibitem[DS97]{DS97} G. David and S. Semmes, Fractured Fractals and Broken Dreams, Oxford Lect. Ser. Math. Appl., \textbf{7}, The Clarendon Press Oxford Univ. Press, New York, (1997).
		
		\bibitem[DM23]{DM23} C. Doherty and J. Miller, On the Sobolev removability of the graph of one-dimensional Brownian motion, ArXiv:2312.07270.
		
		\bibitem[EB23]{Eriksson-Bique} S. Eriksson-Bique, Equality of different definitions of conformal dimension for quasiself-similar and CLP spaces, ArXiv:2309.11447
		
		\bibitem[Fra21]{Fraser} J.M. Fraser, Fractal geometry of Bedford-McMullen carpets: in M. Pollicot, S. Valenti (Eds.), Thermodynamic Formalism. Lecture Notes in Mathematics, vol. 2290 , Springer (2021), 495--516.
		
		\bibitem[Fug57]{Fug57} B. Fuglede, Extremal length and functional completion, Acta Math., \textbf{98}, (1957), 171--219.
		
		\bibitem[Geh73]{Gehring} F. Gehring, 
		The $L^p$-integrability of the partial derivatives of a quasiconformal mapping. Acta Math. \textbf{130}, (1973), 265--277.
		
		\bibitem[Hak06]{Hak06} H. Hakobyan, Cantor sets which are minimal for quasisymmetric maps. J. Contemp. Math. Anal. \textbf{41}, (2006), no.2, 5--13.
		
		\bibitem[Hak09]{Hak09} H. Hakobyan, Conformal dimension: Cantor sets and Fuglede modulus, Int. Math. Res. Not. IMRN, \textbf{1}, (2010), 87--111.
		
		\bibitem[HL23]{HL23} H. Hakobyan and W Li, Quasisymmetric embeddings of slit Sierpiński carpets, Trans. Amer. Math. Soc., \textbf{376}, (2023), 8877--8918.
		
		\bibitem[Hei01]{Hei01} J. Heinonen, Lectures on Analysis on Metric Spaces, Springer-Verlag, New York, (2001).
		
		\bibitem[HK98]{HK98} J. Heinonen and P. Koskela, Quasiconformal maps in metric spaces with controlled geometry, Acta Math., \textbf{181}, (1998), 1--61.
		
		\bibitem[JS00]{JS00} P. Jones and S. Smirnov, Removability theorems for Sobolev
		functions and quasiconformal maps, Ark. Mat., \textbf{38}, (2000), 263-279
		
		\bibitem[KOR18]{KOR18} A. Käenmäki, T.  Ojala, and E.  Rossi, Rigidity of quasisymmetric mappings on self-affine carpets, Int. Math. Res. Not. IMRN, \textbf{12}, (2018), 3769--3799.
		
		\bibitem[KMS22]{KMS22} K. Kavvadias, J. Miller and L. Schoug, Conformal removability of $\mathrm{SLE}_4$, ArXiv: 2209.10532.
		
		\bibitem[KMS23]{KMS23} K. Kavvadias, J. Miller and L. Schoug, Conformal removability of non-simple Schramm-Loewner evolutions, ArXiv:2302.10857.
		
		\bibitem[KL04]{Keith-Laakso} S. Keith and T. Laakso, Conformal Assouad dimension and modulus, Geom. funct. anal., \textbf{14}, (2004), 1278--1321.
		
		\bibitem[Kle06]{Kleiner:icm}  B.  Kleiner,  {The asymptotic geometry of negatively curved spaces: uniformization, geometrization and rigidity}, International Congress of Mathematicians, \textbf{2}, 743--768, Eur. Math. Soc., Z\"urich, (2006).
		
		\bibitem[Kle20]{Kle20} Achim Kelnke, Probability Theory: A Comprehensive Course, Springer, (2020).
		
		\bibitem[Kwa20]{Kwa20} Kwapisz J., Conformal dimension via p-resistance: Sierpinski carpet, Ann. Acad. Sci. Fenn. Math., \textbf{45}, (2020), 3--51.
		
		\bibitem[Kov06]{Kov06} L. Kovalev, Conformal dimension does not assume values between zero and one, Duke Math. J., \textbf{134}, (2006), 1--13.
		
		\bibitem[Law99]{Law99} G. Lawler, Geometric and fractal properties of Brownian motion and random walk paths in two and three dimensions, Bolyai Math. Soc. Stud., \textbf{9}, (1999), 219--258.
		
		\bibitem[Law05]{Law05} G. Lawler, Conformally Invariant Processes in the Plane, Math. Surveys Monogr, \textbf{114}, Amer. Math. Soc., Providence, RI, (2005).
		
		\bibitem[LeG16]{LeG16} J. Le Gall, Brownian Motion, Martingales, and Stochastic Calculus, Grad. Texts in Math., \textbf{274}, Springer, (2016).
		
		\bibitem[McM84]{McMullen} C. McMullen, The Hausdorff dimension of general Sierpi\'nski carpets, Nagoya Math. J., \textbf{96}, (1984), 1--9.
		
		\bibitem[Mac11]{Mac11} J. Mackay, Assouad dimension of self-affine carpets, Conform. Geom. Dyn, \textbf{15}, 2011, 177--87.
		
		\bibitem[Mac12]{Mackay12} J. Mackay, Conformal dimension and random groups, Geom. Funct. Anal. \textbf{22}, no. 1 (2012), 213--239
		
		\bibitem[MT10]{MT10} J. Mackay and J. Tyson, Conformal Dimension: Theory and Application, Univ. Lecture Ser., \textbf{54}, Amer. Math. Soc., Providence, PI, (2010).
		
		\bibitem[Mal75]{Mal75} D. Mallory, Extension of Set Functions to Measures and Applications to Inverse Limit Measures, Canad. Math. Bull., \textbf{18}, (1975), 547--553.
		
		\bibitem[MP10]{MP10} P. Morters and Y. Peres, Brownian Motion, Cambridge Univ. Press, (2010).
		
		\bibitem[Pan89]{Pan89} P. Pansu, Dimension conforme et sph\`ere \`a l’infini des vari\'et\'es \`a courbure n\'egative, Ann. Acad. Sci. Fenn., Ser. AI Math., \textbf{14}, (1989), 177--212.
		
		\bibitem[RS05]{RS05} S. Rohde and O. Schramm, Basic properties of SLE, Ann. of Math., \textbf{161}, (2005), 883--924.
		
		\bibitem[RS21]{RS21} E. Rossi and V. Suomala, Fractal Percolation and Quasisymmetric Mappings, Int. Math. Res. Not. IMRN, \textbf{10}, (2021), 7372--7393.
		
		\bibitem[Rud86]{Rud86} W. Rudin, Real and Complex Analysis, McGraw-Hill, (1986).
		
		\bibitem[Sch00]{Sch00} O. Schramm, Scaling limits of loop-erased random walks and uniform spanning trees, Israel J. Math., \textbf{118}, (2000), 221--288.
		
		\bibitem[Tuk89]{Tukia} P. Tukia, Hausdorff dimension and quasisymmetric mappings.
		Math. Scand. \textbf{65} (1989), no. 1, 152--160.
		
		\bibitem[TV80]{Tukia-Vaisala} P. Tukia and J. V\"ais\"al\"a, Quasisymmetric Embeddings of Metric Spaces, Ann. Acad. Sci. Fenn. Math., \textbf{5}, (1980), 97--114.
		
		\bibitem[Tys00]{Tyson} J. Tyson,  Sets of minimal Hausdorff dimension for quasiconformal maps.  Proceedings of the American Mathematical Society \textbf{128} (2000): 3361--7.
		
		\bibitem[TW06]{TW06} J. Tyson and J. Wu,  Quasiconformal dimensions of self-similar fractals, Rev. Mat. Iberoam, \textbf{22}, (2006), 205--258.
		
		\bibitem[Vai71]{Vaisala} J. V\"ais\"al\"a, Lectures on $n$-Dimensional Quasiconformal Mappings, Lecture Notes in Math., \textbf{229}, Springer, Berlin, (1971).
		
		\bibitem[You15]{You15}M. Younsi, On removable sets for holomorphic functions, EMS Surv. Math. Sci. \textbf{2}, (2015), 219--254.
	\end{thebibliography}
\end{document}